\documentclass[11pt]{article}

\usepackage[letterpaper, hmargin=1in, top=1in, bottom=1.1in, footskip=0.6in]{geometry}

\usepackage{titlesec}
\titleformat{\subsection}[runin]{\large\bfseries}{\thesubsection.}{.5em}{}[.]\titlespacing{\subsection}{0pt}{2ex plus .1ex minus .2ex}{.8em}
\titleformat{\subsubsection}[runin]{\normalfont\bfseries}{\thesubsubsection.}{.3em}{}[.]\titlespacing{\subsubsection}{0pt}{2ex plus .1ex minus .2ex}{.5em}
\titleformat{\paragraph}[runin]{\normalfont\itshape}{\theparagraph.}{.3em}{}[.]\titlespacing{\paragraph}{0pt}{1ex plus .1ex minus .2ex}{.5em}

\usepackage{amsmath}
\usepackage{amssymb}
\usepackage{amsfonts}
\usepackage{latexsym}
\usepackage{amsthm}
\usepackage{amsxtra}
\usepackage{amscd}
\usepackage{bbm}
\usepackage{mathrsfs}
\usepackage{bm}
\usepackage{tensor}

\usepackage{graphicx, color}

\definecolor{darkred}{rgb}{0.9,0,0.3}
\definecolor{darkblue}{rgb}{0,0.3,0.9}

\definecolor{vdarkred}{rgb}{0.7,0,0.2}
\definecolor{vdarkblue}{rgb}{0,0.2,0.7}

\usepackage[pdftex, colorlinks, linkcolor=vdarkblue,urlcolor=black,citecolor=vdarkred]{hyperref}
\usepackage{cleveref}

\usepackage[nottoc,notlof,notlot]{tocbibind}
\usepackage{cite}

\numberwithin{equation}{section}
\numberwithin{figure}{section}

\theoremstyle{plain} 
\newtheorem{theorem}{Theorem}[section]
\newtheorem*{theorem*}{Theorem}
\newtheorem{lemma}[theorem]{Lemma}
\newtheorem*{lemma*}{Lemma}
\newtheorem{corollary}[theorem]{Corollary}
\newtheorem*{corollary*}{Corollary}
\newtheorem{proposition}[theorem]{Proposition}
\newtheorem*{proposition*}{Proposition}

\newtheorem*{conjecture*}{Conjecture}

\theoremstyle{definition} 
\newtheorem{definition}[theorem]{Definition}
\newtheorem*{definition*}{Definition}
\newtheorem{example}[theorem]{Example}
\newtheorem*{example*}{Example}
\newtheorem{remark}[theorem]{Remark}
\newtheorem*{remark*}{Remark}

\newtheorem*{assumption*}{Assumption}

\renewcommand{\b}[1]{\boldsymbol{\mathrm{#1}}} 
\newcommand{\bb}{\mathbb} 
\renewcommand{\cal}{\mathcal}

\newcommand{\ul}[1]{\underline{#1} \!\,} 

\newcommand{\txt}[1]{\text{\rm{#1}}}

\newcommand{\E}{\mathbb{E}}
\newcommand{\R}{\mathbb{R}}
\newcommand{\C}{\mathbb{C}}
\newcommand{\N}{\mathbb{N}}

\newcommand{\D}{\mathrm{D}}

\newcommand{\ii}{\mathrm{i}}
\newcommand{\dd}{\mathrm{d}}
\newcommand{\col}{\mathrel{\vcenter{\baselineskip0.75ex \lineskiplimit0pt \hbox{.}\hbox{.}}}}
\newcommand*{\deq}{\mathrel{\vcenter{\baselineskip0.65ex \lineskiplimit0pt \hbox{.}\hbox{.}}}=}
\newcommand*{\eqd}{=\mathrel{\vcenter{\baselineskip0.65ex \lineskiplimit0pt \hbox{.}\hbox{.}}}}

\renewcommand{\leq}{\leqslant}
\renewcommand{\le}{\leqslant}
\renewcommand{\geq}{\geqslant}

\renewcommand{\epsilon}{\varepsilon}

\DeclareMathOperator{\tr}{Tr}

\DeclareMathOperator{\re}{Re}
\DeclareMathOperator{\im}{Im}

\newcommand{\G}{ {\mathcal G} }

\newcommand{\psum}{\sideset{}{^*}\sum}

\newcommand*{\rom}[1]{\expandafter\@slowromancap\romannumeral #1@}

\title{\bf \Large Extremal eigenvectors of sparse random matrices\vspace{0.5em}}

\author{Yukun He\footnote{Shanghai Center for Mathematical Sciences, Fudan University. Email:	\href{mailto:heyukun@fudan.edu.cn}{heyukun@fudan.edu.cn}}  \and Jiaoyang Huang\footnote{Department of Statistics and Data Science, University of Pennsylvania. Email:	\href{mailto:huangjy@wharton.upenn.edu}{huangjy@wharton.upenn.edu}} \and Chen Wang\footnote{Department of Mathematics, City University of Hong Kong. Email:	\href{mailto:cwang228-c@my.cityu.edu.hk}{cwang228-c@my.cityu.edu.hk}}\vspace{1em}}

\begin{document}
	
\maketitle

\begin{abstract}
	We consider a class of sparse random matrices, which includes the adjacency matrix of the Erd\H{o}s-R\'enyi graph ${\bf G}(N,p)$. For $N^{-1+o(1)}\leq p\leq 1/2$, we show that the non-trivial edge eigenvectors are asymptotically jointly normal. The main ingredient of the proof is an algorithm that directly computes the joint eigenvector distributions, without comparisons with GOE. The method is applicable in general. As an illustration, we also use it to prove the normal fluctuation in quantum ergodicity at the edge for Wigner matrices. Another ingredient of the proof is the isotropic local law for sparse matrices, which at the same time improves several existing results.
\end{abstract}

\section{Introduction}\label{sec1}

Fix small $0<\tau\leq 1/2$. In this paper, we consider the following class of random matrices.

\begin{definition}[Sparse matrix] \label{def_sparse} Assume $N^{\tau}\leq q\leq N^{1/2}$. Let $H=H^T \in \bb R^{N \times N}$ be a real-symmetric random matrix whose entries $H_{ij}$ satisfy the following conditions.
	\begin{enumerate}
		\item[(i)] The upper-triangular entries ($H_{ij}: 1 \leq i \leq j\leq N$) are independent.
		\item[(ii)] We have $\E H_{ij}=0$ and $ \E H_{ij}^2=(1+O(\delta_{ij}))/N$ for all $i,j$.
		\item[(iii)] For any fixed $k\geq 3$, we have
		$\E|H_{ij}|^k =O_k(N^{-1}q^{2-k})$ for all $i,j$.
	\end{enumerate}
Let $f$ satisfy $\tau q\leq f\leq q/\tau$. The \emph{sparse matrix} is defined as 
\[
A = H + f \b e \b e^T\,,
\]
where $\b e \deq N^{-1/2}(1,1,\dots,1)^T$. 
\end{definition}

One important motivation of Definition \ref{def_sparse} is the study of the adjacency matrix $\cal A$ of the Erd\H{o}s-R\'enyi graph ${\bf G}(N,p)$. Explicitly, $\cal A=\cal A^T\in \bb R^{N\times N} $ is a symmetric random matrix with independent upper triangular entries $(\cal A_{ij} \col i \leq j)$ satisfying
\begin{equation*}
	{\cal A}_{ij}=\begin{cases}
		1 & \txt{with probability } p
		\\
		0 & \txt{with probability } 1-p\,.
	\end{cases}
\end{equation*} 
It is easy to check that when $N^{-1+o(1)}\leq p\leq 1/2$\footnote{Here $o(1)$ means for any fixed (small) $\varepsilon > 0$.}, the rescaled adjacency matrix $A\deq\cal A/\sqrt{Np(1-p)}$ satisfies Definition \ref{def_sparse} with the choice $q\deq \sqrt{Np}$.

The Erd\H{o}s-R\'enyi graph is the simplest model of a random graph, and it has numerous applications in graph theory, network theory, mathematical physics and combinatorics. During the past decades, there has been enormous results on the spectrum and eigenvectors of the model. The matrix $A$ has typically $N^2p$ nonzero entries, and hence $A$ is \textit{sparse} whenever $p\to 0$ as $N \to \infty$. When $p\geq N^{-1+o(1)}$, the microscopic eigenvalue statistics are well-understood both in the bulk of the spectrum \cite{EKYY1,EKYY2,HLY15} and near the spectral edges \cite{EYY3,EKYY2,LS1,HLY,HK20,Lee21,HY22}. In terms of eigenvectors, it was proved in \cite{BHY} that inside the bulk they are asymptotically Gaussian. On the other hand, there are no results on the distribution of the extreme eigenvectors. In this paper, we show that all non-trivial edge eigenvectors of $A$ are asymptotically jointly normal, in all deterministic directions.

In the sequel, let $\lambda_1\geq \cdots \geq \lambda_N$ denote the eigenvalues of $A$, with corresponding eigenvectors $\b u_1,...,\b u_N$. Analogously, let $\lambda^H_1\geq \cdots \geq \lambda^H_N$ be the eigenvalues of $H$, with corresponding eigenvectors $\b u^H_1,...,\b u^H_N$.
Define
\[
\b e \deq N^{-1/2}(1,1,\dots,1)^T\,, \quad  \mbox{and} \quad  \bb S_\perp^{N-1}\deq \{\b v\in \bb S^{N-1}: \b v \perp \b e\}.
\] 
We may now state our first main result.

\begin{theorem}[Universality of edge eigenvectors] \label{thm1.2}
Fix $k\in \bb N_+$, and let $\b v_1,\b w_1,...,\b v_k,\b w_k \in \bb S^{N-1}_\perp$ be deterministic. Fix $T>0$. There exists fixed $\varepsilon>0$ such that
\[
\bb E \exp\bigg(\sum_{a=1}^k\ii  t_aN\langle \b  v_{a},\b u_{a+1}\rangle \langle \b w_a, \b u_{a+1}\rangle\bigg)=\bb E  \exp\bigg(\sum_{a=1}^k\ii t_a\langle \b  v_{a},\b z_{a}\rangle \langle \b w_a, \b z_{a}\rangle\bigg)+O(N^{-\varepsilon})
\]
uniformly for all $t_a\in [-T,T]$. Here $(\b z_a)_{a=1}^{k}$ are i.i.d.\,standard Gaussian vectors in $\bb R^{N}$. 
\end{theorem}

\begin{remark}
In Theorem \ref{thm1.2}, it is enough to take the test vectors $\b v_1,\b w_1,...,\b v_k,\b w_k $ in $\bb S_\perp^{N-1}$. In fact, we shall show that all nontrivial eigenvectors of $A$ are almost orthogonal to $\b e$; see Corollary \ref{cor1.5} below. Combining Theorem \ref{thm1.2} and Corollary \ref{cor1.5} yields a complete description of the edge eigenvector distributions of $A$.
\end{remark}

One important technical step towards showing Theorem \ref{thm1.2} is the isotropic local semicircle law, which has become a cornerstone of studying the eigenvectors of random matrices. To state it, let us define the spectral domain
\begin{equation} \label{spectral}
\b D \equiv \b D_\tau\deq \{z=E+\ii \eta: |E|\leq 3\,, N^{-1+\tau}\leq \eta\leq 1\}\,.
\end{equation}
The following is our second main result.

\begin{theorem}[Isotropic local law] \label{thm1.3}
(i)	Fix $\b v, \b w\in \bb S^{N-1}_\perp$. We have
	\begin{equation} \label{1.1}
	\sup_{z\in \b D}|\langle \b v, (A-z)^{-1} \b w \rangle - \langle \b v, \b w \rangle m_{sc}(z)| \leq  N^{o(1)}\Big(\frac{1}{(N\eta)^{1/3}}+\frac{1}{q^{1/3}}\Big)
	\end{equation}
	with very high probability. Here $m_{sc}$ denotes the Stieltjes transform of the semicircle density. In addition, we have
	\begin{equation} \label{1.2}
	\sup_{z\in \b D}|\langle \b e, (A-z)^{-1}  \b e
	 \rangle-f^{-1}| \leq  N^{o(1)}f^{-2} 
	\end{equation}
and  
\begin{equation} \label{1.3}
	\sup_{z\in \b D}|\langle \b e, (A-z)^{-1}  \b v
	\rangle| \leq  N^{o(1)}f^{-1}
\end{equation}
	with very high probability. 
	
	(ii) Fix $\b x,\b y \in \bb S^{N-1}$. We have
		\begin{equation} \label{Hisotropic}
		\sup_{z\in \b D}|\langle \b x, (H-z)^{-1} \b y \rangle - \langle \b x, \b y \rangle m_{sc}(z)| \leq  N^{o(1)}\Big(\frac{1}{(N\eta)^{1/3}}+\frac{1}{q^{1/3}}\Big)
	\end{equation}
with very high probability. 
\end{theorem}

Previously, the isotropic local law was obtained for dense Wigner matrices and sample covariance matrices \cite{KY2,KY4,BEKYY}; for sparse matrices, only the entrywise local law was available \cite{EKYY1}. Compared to the entrywise law, the main obstacle in proving the isotropic law is that the higher order terms break the isotropic structure, which leads to weak estimates of the error terms. In the dense case, this is compensated by the fast decay of the higher moments of the matrix entries. For sparse matrices, we no longer have the fast moment decay, which forces us to exploit detailed structural properties in the computations. To this end, we uncover a simple index mismatching that appear in all the error terms. Moreover, it persists after any finitely many expansions. This allows us to recursively expand the error terms as in \cite{HK2} (see also \cite{H19,HK20,Lee21,HY22,H23,SX22,BSX25}), and eventually prove the isotropic law for matrices at all sparsity $p\geq N^{-1+o(1)}$. The method presented here also applies to sparse sample covariance matrices.

In addition, as $\bb EA$ is large and proportional to $\b e \b e^T$, the Green function is much smaller than 1 in the direction of $\b e$ and thus hard to detect (see \eqref{1.2}). Our main idea here is splitting the bootstrap step into two parts: Green functions in the direction $\b e$ and Green functions in the corresponding orthogonal directions, where the former will inherit an additional smallness from the bootstrap. In the probabilistic component of the proof, we track down the number of $\b e$ that appear in the Green functions, and make use of their prior estimates. This allows us to prove the optimal estimate \eqref{1.2}. 

By spectral decomposition, \eqref{1.2} implies that all non-trivial eigenvectors of $A$ are almost orthogonal to $\b e$. Together with \eqref{1.1}, these eigenvectors are also isotropically delocalized. In addition, \eqref{Hisotropic} implies that all eigenvectors of $H$ are isotropically delocalized.

\begin{corollary} \label{cor1.5}
Fix $\b x \in \bb S^{N-1}$. We have
	\[
\max_{i=2,...,N}	|\langle \b e, \b u_i\rangle| =O(N^{-1/2+o(1)}f^{-1})\,, \quad  \max_{i=2,...,N}|\langle \b x, \b u_i \rangle| =O(N^{-1/2+o(1)})\,, 
\]
and
\[ \max_{i=1,...,N} |\langle \b x,\b u_i^H \rangle| =O(N^{-1/2+o(1)})
\]
with very high probability.
\end{corollary}

The isotropic delocalization will be a key input in studying edge eigenvectors; see \eqref{isotropic delocalization application} below. In addition, by \eqref{1.1}, Corollary \ref{cor1.5} and \cite[Theorem 1.5]{BHY}, we can also prove the normality of bulk eigenvectors.

\begin{theorem}[Universality of bulk eigenvectors]\label{thmbulk}
	Fix $k\geq 1$ and $\b v \in \bb S^{N-1}_\perp$. For any polynomial $P$ of $k$ variables, there exists fixed $\varepsilon>0$ such that
	\begin{align}\label{e:EP}
	\bb E P(N\langle \b v, \b u_{i_1}\rangle^2,...,N\langle \b v, \b u_{i_k}\rangle^2)=\bb E P( Z_1^2,..., Z_k^2)+O(N^{-\varepsilon})\,.
	\end{align}
	uniformly for all $i_1,j_1,...,i_k,j_k\in [\tau N,(1-\tau)N]$. Here $(Z_a)_{a=1}^{k}$ are i.i.d.\,standard Gaussian random variables.
\end{theorem}

\begin{remark}
	In \cite{BHY}, the isotropic local law and bulk eigenvector universality were proved for the adjacency matrix $\cal A$ of $\b G(N,p)$ in the directions of $\bb S^{N-1}_\perp$. One key input was the exchangeability of the model, which only holds for the adjacency matrix. Here our result holds for general sparse matrices satisfying Definition \ref{def_sparse}. Moreover, combining with Corollary \ref{cor1.5}, our result covers all deterministic directions.
\end{remark}

In terms of eigenvalues, the edge universality was proved in \cite{HLY,Lee21,HY22} only for the centered matrix $H$. The key obstacle of extending the result to $A$ was the the lack of sharp isotropic estimates. Combining Theorem \ref{thm1.3} and the results of \cite{HLY,Lee21,HY22}, we can deduce the edge universality of $A$, which genuinely covers the adjacency matrix of sparse Erd\H{o}s-R\'enyi graphs. 

\begin{theorem}[Edge universality] \label{thm1.4}
Let $\mu_1\geq\mu_2...\geq \mu_N$ be the eigenvalues of $GOE$. Fix $k\geq 1$. There exists an explicit random variable $\cal L$ such that for any smooth, compactly supported $F: \bb R^{k}\to \bb R$, we have
\[
\bb E F(N^{2/3}(\lambda_2-\cal L),...,N^{2/3}(\lambda_{k+1}-\cal L))=\bb E F(N^{2/3}(\mu_1-2),...,N^{2/3}(\mu_{k}-2))+O(N^{-\varepsilon})
\]
for some fixed $\varepsilon>0$.
\end{theorem}

Armed with the isotropic law, we are now able to prove our first main result, Theorem \ref{thm1.2}. Our starting point follows the strategy of \cite{KY1}, by converting the eigenvectors into integrals of the Green functions near the edge. However, the comparison argument used in \cite{KY1,BCLX23} is not sufficient in the sparse case. To this end, we present a method that directly computes the distribution of the eigenvectors. More precisely, after non-trivial transformations of the Green functions via cumulant expansion, we convert the leading Green function integrals back to the eigenvectors (see e.g.\,\eqref{5.19} below), which allows us to form self-consistent equations of the characteristic functions of the eigenvectors and obtaining their limiting distributions.

This method of directly computing the distribution of edge eigenvectors is applicable to other situations. In a companion paper \cite{HHY}, the first and second named authors, in collaboration with Horng-Tzer Yau, adapt this method to show that the edge eigenvectors of random regular graphs converge to Gaussian waves with variance 1. 

Historically, the proofs of universality of local statistics in random matrix theory always rely on comparisons with the Gaussian model (through Green function comparison or Dyson Brownian motion). To our knowledge, this paper and its companion \cite{HHY} are the first works that directly establish universality in the microscopic scale.

As another illustration of the method, we also prove the following normal fluctuation in quantum ergodicity at the spectral edge for Wigner matrices. 

\begin{theorem} \label{thm1.9}
Set $q=N^{1/2}$ in Definition \ref{def_sparse}, which makes $H$ a standard Wigner matrix. Moreover, we assume that $H_{11}\overset{d}{=}\cdots \overset{d}{=}H_{NN}$ and $H_{ij}\overset{d}{=}H_{i'j'}$ for all $i> j$ and $i'> j'$\footnote{We add this assumption since it is required in the Eigenstate Thermalisation Hypothesis \cite{CEH1}; for the arguments in this article, we do not need the entries to have identical distributions.}. Fix $k
\in \bb N_+$.  Let $B_1,...,B_k\in \bb R^{N\times N}$ be deterministic, real-symmetric and traceless matrices. For all $a=1,...,k$, suppose $\tr B_a^2\geq N^{\tau} \|B_a\|^2>0$.   Fix $T>0$. There exists fixed $\varepsilon>0$ such that
	\[
\bb E \exp\bigg(\sum_{a=1}^k\ii  t_a\frac{N}{\sqrt{2\tr B_a^2}} \langle {\b u^H_a}, B_a {\b u^H_a}\rangle \bigg)=\bb E  \exp\bigg(\sum_{a=1}^k\ii t_a Z_a\bigg)+O(N^{-\varepsilon})	
	\]
uniformly for all $t_a\in [-T,T]$. Here $(Z_a)_{a=1}^{k}$ are i.i.d.\,standard Gaussian random variables. 
\end{theorem}

For Wigner matrices, the universality of eigenvectors was obtained in \cite{TV3,KY1,BoY1}.  The fluctuation in quantum ergodicity inside the bulk has been studied intensively in \cite{CEK2,CEK5,BL22,BC24}; near the edge, the distribution was obtained when the observable is a projection \cite{BCLX23}\footnote{More recently, an updated version of \cite{BCLX23} also covers general observables.}. In Theorem \ref{thm1.9}, we settle the fluctuation in quantum ergodicity near the edge in full generality.  Notably, unlike previous results, our proof does not require the observables $B_a$ to be Hermitian: we only assume so as $\langle \b u, B \b u \rangle=\langle \b u, B^* \b u \rangle$ if $B$ and $\b u$ are real; see \eqref{1.7} below for the complex case. Besides the method introduced in this paper, the proof also relies on the Eigenstate Thermalisation Hypothesis proved in \cite{CEH1}. 

\begin{remark}
(i)	For convenience, we only state the results for the largest eigenvalues and the corresponding eigenvectors in Theorems \ref{thm1.2}, \ref{thm1.4}, and \ref{thm1.9}. The analogue also holds true for the smallest eigenvalues and the corresponding eigenvectors.

(ii) For Theorem \ref{thm1.9}, our method also applies to the study of complex Wigner matrices (and complex observables). Let $\cal H\in \bb C^{N\times N}$ be a complex Hermitian random matrix, with independent upper-triangular entries. Assume $\bb E \cal H_{ij}=0$, $\bb E |\cal H_{ij}^2|=N^{-1}$, $ |\bb E \cal H_{ij}^{2}|\leq (1-\tau)N^{-1}$, and $\bb E |\cal H_{ij}^k|\leq C_kN^{-k/2}$ for all $i,j \in \{1,2,...,N\}$ and fixed $k\in \bb N_+$. Let ${\b u}^{\cal H}_1$ denote the top eigenvector of $\cal H$. Let $B\in \bb C^{N\times N}$ be a deterministic, traceless matrix satisfying $\tr |B|^2\geq N^{\tau} \|B\|^2$ for some fixed $\tau>0$. In this case, one can prove that
\begin{equation} \label{1.7}
\frac{N}{\sqrt{\tr |B|^2}} \langle {\b u}^{\cal H}_1, B {\b u}^{\cal H}_1 \rangle \overset{d}{\approx} \frac{1}{\sqrt{\tr |B|^2}} \langle \b g, B \b g\rangle\,,
\end{equation}
where $\b g$ is the standard complex Gaussian vector in $\bb C^N$. The RHS of \eqref{1.7} has variance 1, and its real and imaginary parts have Gaussian distributions. On the other hand, its phase depends on $B$, and the real and imaginary parts are not necessarily independent. To prove \eqref{1.7}, two minor adjustments are needed on top of Theorem \ref{thm1.9}: one needs to compute the complex moments instead of the characteristic function, and generalize the notion of imaginary part of Green functions as in \cite[(3.1)]{KY1}. We do not pursue it here.
\end{remark}

\subsection*{Other related results}
This paper focuses on the regime $p\geq N^{-1+o(1)}$. The situation is dramatically different for very sparse Erd{\H o}s--R{\'e}nyi graphs. In the very sparse regime $p=O(\ln N/N)$, for Erd{\H o}s--R{\'e}nyi graphs, there exists a critical value $b_*=1/(\ln 4-1)$ such that  if $p> b_*\ln N/N$, the extreme eigenvalues of the normalized adjacency matrix converge to $\pm 2$ \cite{benaych2019largest, alt2021extremal, tikhomirov2021outliers, benaych2020spectral}, and all the eigenvectors are delocalized \cite{alt2022completely, EKYY1}. For $ p<b_*\ln N/N$, there exist outlier eigenvalues \cite{tikhomirov2021outliers, alt2021extremal}, and the edge eigenvectors are localized or semi-localized\cite{alt2021delocalization, ADK23}. In terms of distributions, it was proved in \cite{alt2023poisson} that the extreme eigenvalues have Poisson statistics in the subcritical regime $ p<b_*\ln N/N$.  We remark that the fluctuations sit on much smaller scales in the supercritical regime, and it is currently difficult to obtain distributions of the eigenvalues or eigenvectors when $b_*\log N/N<p <(\log N)^C/N$.

Another closely related model is the random $d$-regular graph on $N$ vertices $\b G_d(N)$. The local law for random regular graphs was proved for $N^{o(1)}\leq d\leq N^{2/3-o(1)}$ in \cite{BKY15}, for $N^{o(1)}\leq d\leq N/2$ in \cite{H22}, and for fixed $d\geq 3$ in \cite{BHY19,huang2024spectrum}. The bulk universality was proved for $N^{o(1)}\leq d\leq N^{2/3-o(1)}$ in \cite{BHKY15}. The universality of extreme eigenvalues was proved for $N^{2/9+o(1)}\leq d\leq N^{1/3-o(1)}$ in \cite{BHKY19}, for $N^{2/3+o(1)}\leq d \leq N/2$ in \cite{H22}, for $N^{o(1)}\leq d \leq N^{1/3-o(1)}$ in \cite{HY23}, and for fixed $d\geq 3$ in \cite{huang2024ramanujan}.

\subsection*{Organizations}
The paper is organized as follows. In Section \ref{sec2} we introduce the notations used in this paper. In Section \ref{sec 3} we prove Theorem \ref{thm1.3}, and in Section \ref{sec4} we apply the isotropic local law to prove Theorems \ref{thmbulk} and \ref{thm1.4}. In Section \ref{sec5} we prove Theorem \ref{thm1.2}, with the aid of Theorem \ref{thm1.3}. Section \ref{appA} is an independent part of the paper, where we prove Theorem \ref{thm1.9}.

\subsection*{Conventions} Unless stated otherwise, all quantities in this paper depend on the fundamental large parameter $N$, and we omit this dependence from our notation. We use the usual big $O$ notation $O(\cdot)$, and if the implicit constant depends on a parameter $\alpha$, we indicate it by writing $O_\alpha(\cdot)$. Let 
$$
X=(X^{(N)}(u): N \in \bb N, u \in U^{(N)})\,, \quad Y=(Y^{(N)}(u): N \in \bb N, u \in U^{(N)})
$$
be two families of  random variables, where $U^{(N)}$ is a possibly $N$-dependent parameter
set, and $Y\geq 0$. We say that $X$ is stochastically dominated by $Y$ , uniformly in $u$, if for any fixed $D,\varepsilon>0$,
\[
\sup_{u \in U^{(N)}}\bb P(|X|\geq Y N^{\varepsilon}) =O_{\varepsilon,D} (N^{-D})\,.
\]
We write $X\asymp Y$ if $X =O(Y)$ and $Y=O(X)$. If $X$ is stochastically dominated by $Y$ , we use the notation $X=O_{\prec}(Y)$, or equivalently $X \prec Y$. We say an event $\Omega$ holds with very high probability if for any fixed $D>0$, $1-\bb P(\Omega)=O_D(N^{-D})$. We shall use $\ell$ to denote a generic large positive integer, which may depend on some fixed parameters and whose value may change from one expression to the next.

\subsection*{Acknowledgment} YH and CW are supported by National Key R\&D Program of China No.\,2023YFA1010400,  NSFC No.\,12322121 and Hong Kong RGC Grant No.\,21300223. The research of JH is supported by NSF grants DMS-2331096 and DMS-2337795, and the Sloan Research Award. The authors thank the reviewers for their careful reading of the manuscript and for their helpful comments.

\section{Preliminaries} \label{sec2}

For a complex random variable $X$, we denote $\|X\|_n\deq (\bb E |X^n|)^{1/n}$ for any $n\in \bb N_+$. For any matrix $M\in \bb C^{N\times N}$, we abbreviate
\[
\widehat{M}=\re M\,, \quad \widetilde{M}=\im M\,, \quad 
 \ul{M}\deq N^{-1}\tr M\quad \mbox{and} \quad M_{\b v\b w}\deq \langle \b v, M \b w\rangle
\]
for all $\b v,\b w \in \bb R^N$. Recall from Section \ref{sec1} that $\b e=N^{-1/2}(1,...,1)^T\in \bb R^N$, and we further use $\b e_i$ to denote the standard $i$th basis in $\bb R^N$. Let $\varrho_{sc}(x)\deq \frac{1}{2\pi} \sqrt{(4-x^2)}$ denote the semicircle distribution. For $z \in \bb C\backslash \bb R$, we define
\begin{equation} \label{2.1}
m_{sc}(z)\deq \int_{\bb R} \frac{\varrho_{sc}(x)}{x-z} \,\dd x \,, \quad G\equiv G(z)\deq (A-z)^{-1}\,,\quad  \mbox{and} \quad \cal G\equiv \cal G(z)\deq (H-z)^{-1}\,.
\end{equation}
Let us abbreviate
\begin{equation} \label{abbre}
\partial_{ij}\deq \frac{\partial }{\partial H_{ij}}\,.
\end{equation}
We have the differential rule
\begin{equation} \label{diff}
\partial_{kl}G_{ij}=\frac{\partial G_{ij}}{\partial H_{kl}}=-(G_{ik}G_{jl}+G_{il}G_{kj})(1+\delta_{kl})^{-1}\,.
\end{equation}
Obviously, the same rule also holds when $G$ is replaced by $\cal G$.

\paragraph{Cumulant expansion} Recall that for a real random variable $h$, all of whose moments are finite, the $s$th-cumulant of $h$ is
\begin{equation*}
	\cal C_s(h)\deq(-\mathrm{i})^s\bigg(\frac{\dd^{s}}{\dd t^s}\log\mathbb{E}[e^{\mathrm{i}th}]\bigg)\Bigg|_{t=0}.
\end{equation*}
We shall use a standard cumulant expansion from \cite{KKP, Kho2, HK}. The proof was given in e.g.\,\cite[Appendix A]{HKR}.

\begin{lemma}[Cumulant expansion] \label{lem:cumulant_expansion}
	Let $F\col\R\to\C$ be a smooth function, and denote by $F^{(n)}$ its $n$th derivative. Then, for every fixed $\ell \in\N$, we have 
	\begin{equation}\label{eq:cumulant_expansion}
		\mathbb{E}\big[h\cdot F(h)\big]=\sum_{s=0}^{\ell}\frac{1}{s!}\mathcal{C}_{s+1}(h)\mathbb{E}[F^{(s)}(h)]+\cal R_{\ell+1},
	\end{equation}	
	assuming that all expectations in \eqref{eq:cumulant_expansion} exist, where $\cal R_{\ell+1}$ is a remainder term (depending on $f$ and $h$), such that for any $t>0$,
	\begin{equation*} 
		\cal R_{\ell+1} = O(1) \cdot \bigg(\E\sup_{|x| \le |h|} \big|F^{(\ell+1)}(x)\big|^2 \cdot \E \,\big| h^{2\ell+4} \mathbf{1}_{|h|>t} \big| \bigg)^{1/2} +O(1) \cdot \bb E |h|^{\ell+2} \cdot  \sup_{|x| \le t}\big|F^{(\ell+1)}(x)\big|\,.
	\end{equation*}
\end{lemma}

\begin{remark}
	In this paper, we shall always assume the remainder term is negligible for some fixed, large $\ell$. This can be checked through a standard argument; see e.g.\,\cite[Section 4.3]{HK}.
\end{remark}

The next lemma follows from Definition \ref{def_sparse}.

\begin{lemma} \label{lemma 2.3}
	For any fixed $k \geq 2$ we have
	\begin{equation*}
		\cal C_{k}(H_{ij})=O(1/(Nq^{k-2}))
	\end{equation*}
	uniformly for all $i,j$.
\end{lemma}

We recall the local semicircle law for sparse random matrices from \cite[Theorem 2.9]{EKYY1}.
\begin{proposition}[Local semicircle law] \label{refthm1}
	We have
	\begin{equation*} 
		\max\limits_{i,j}|G_{ij}(z)-\delta_{ij}m_{\mathrm{sc}}(z)| \prec \sqrt{\frac{1}{N\eta}}+\frac{1}{N\eta}+\frac{1}{q}
	\end{equation*} 
	uniformly in $z  \in \{z=E+\ii \eta: E\in \bb R, \eta >0\}$. 
\end{proposition}

A standard consequence of the local law is the complete delocalization of eigenvectors.

\begin{corollary} \label{delocalization}
	We have
	\[
	\|\b u_{i}\|_\infty\prec N^{-1/2}
	\]
	uniformly for all $i\in \{1,2,...,N\}$.
\end{corollary}

We also import from \cite{EKYY1} about the result on top eigenvalue and eigenvector of $A$.

\begin{proposition} \label{proptopeigenvector}
	We have
		\[
	\lambda_1=f+O_{\prec}(f^{-1})\,.	
	\]
In addition, 
	\[
	\langle \b e ,\b u_1 \rangle =1-\frac{1}{2f^2}+O_{\prec}\Big(\frac{1}{f^3}+\frac{1}{\sqrt{N}f}\Big)\quad \mbox{and}\quad \sum_{i=2}^N \langle \b e,\b u_i\rangle^2=O(f^{-2})\,.
	\]
\end{proposition}

Finally, we recall the Ward identity.

\begin{lemma} \label{ward}
	We have 
	\begin{equation*}
		\sum_j |G_{ij}|^2 =\frac{\im G_{ii}}{\eta}
	\end{equation*}
	for all $z=E+\ii \eta$ with $\eta>0$.
\end{lemma}

\section{Proof of Theorem \ref{thm1.3}} \label{sec 3}
In this section, let $\tau$ be as in the beginning of Section \ref{sec1}, and we fix parameters
\begin{equation} \label{parameters}
	\xi \in (0,\tau/100) \quad \mbox{and} \quad  \delta\in (0,\xi/100)\,. 
\end{equation}
We shall prove Theorem \ref{thm1.3} (i); the proof of Theorem \ref{thm1.3} (ii) is identical to that of \eqref{1.1}. 

Fix $\b v, \b w\in \bb S^{N-1}_\perp$ as in Theorem \ref{thm1.3}. We denote
\[
\bb X\deq \{\b v, \b w, \b e_1,...,\b e_N\}
\]
Let $z \in \b D_{\tau}$. Suppose 
\begin{equation} \label{3.0}
 \quad G_{\b e \b e}(z)-f^{-1}\prec \phi f^{-2}\,, \quad \max_{\b x \in \bb X}|G_{\b e\b x}(z)|\prec \phi f^{-1}
\end{equation}
and 
\begin{equation} \label{3.1}
	\max_{\b x, \b y\in \bb X}	|G_{\b x\b y}(z)|\prec \phi 
\end{equation}
for some deterministic $\phi \in [1,N^{\xi}]$. In particular, $\phi \ll f$. We shall prove that at $z$ we have
\begin{equation} \label{3.2}
 \quad G_{\b e \b e}(z)-f^{-1}\prec f^{-2}\,,
\end{equation}
and
\begin{equation} \label{3.2.2}
	\max_{\b x\in \bb X}	|G_{\b e\b x}(z)| \prec f^{-1}
\end{equation}
as well as
\begin{equation}  \label{3.3}
 \max_{\b x, \b y \in \bb X}|G_{\b x\b y}(z)-\langle \b x,\b y \rangle m_{sc}(z)| \prec \phi^3\Big(\frac{1}{(N\eta)^{1/3}}+\frac{1}{q^{1/3}}\Big)\,.
\end{equation}

Armed with \eqref{3.2} -- \eqref{3.3}, the proof of Theorem \ref{thm1.3} (i) follows from a bootstrap argument. More precisely, let $E \in [-3,3]$ be deterministic. For $z_0=E+\ii$, by spectral decomposition and Proposition \ref{proptopeigenvector}, we easily get
\[
G_{\b e\b x}(z_0)=\sum_{\alpha=1}^N \frac{\langle \b e ,\b u_\alpha \rangle \langle \b u_\alpha ,\b x\rangle}{\lambda_{\alpha}-E-\ii} \prec \sum_{\alpha=2}^N |\langle \b e ,\b u_\alpha \rangle \langle \b u_\alpha ,\b x\rangle|+f^{-1}\leq \Big(\sum_{\alpha=2}^N \langle \b e ,\b u_\alpha \rangle^2\Big)^{1/2} +f^{-1}\prec f^{-1}
\]
uniformly for all $\b x \in \bb X$. In addition, by Proposition \ref{proptopeigenvector}, we have
\[
G_{\b e \b e}(z_0)-f^{-1} \prec \sum_{\alpha=2}^N |\langle \b e ,\b u_\alpha \rangle|^2+\bigg|\frac{\langle \b e ,\b u_1 \rangle^2}{\lambda_1-E-\ii }-f^{-1}\bigg| \prec f^{-2} \,, \quad \mbox{and} \quad 	\max_{\b x, \b y\in \bb X}	|G_{\b x\b y}(z_0)|\leq 1\,.
\]
Thus \eqref{3.0} and \eqref{3.1} hold at $z_0$.

 Now suppose \eqref{3.2} -- \eqref{3.3} hold true at some $z_1=E+\ii \eta_1\in \b D_{\tau}$. By \eqref{3.2} we get
\begin{equation} \label{3.5}
\im G_{\b e \b e}(z_1)\prec f^{-2}\,.
\end{equation}
Let $\eta_2\deq \eta_1 N^{-\delta}$. Denote ${\bf g}(\eta)\deq \im G_{\b e \b e}(E+\ii \eta)$. It is easy to see that $|\dd {\bf g}/\dd \eta|\leq {\bf g}/\eta$, which implies 
$$
\frac{\dd}{\dd \eta}(\eta {\bf g}(\eta))\geq 0\,.
$$
Together with \eqref{3.5} we have
\begin{equation} \label{3.6}
	\im G_{\b e \b e}(z)\prec N^{\delta} f^{-2}
\end{equation}
uniformly for all $z=E+\ii \eta$, where $\eta \in [\eta_2,\eta_1]$. By \eqref{3.6}, we see that
\[
\frac{d (G_{\b e \b e}(z)-f^{-1})}{\dd \eta} \prec \frac{\im G_{\b e \b e}(z)}{\eta} \prec N^{\delta} f^{-2}\eta^{-1}\,,
\]
which implies 
\begin{equation} \label{3.7}
	G_{\b e \b e}(z)-f^{-1} \prec N^{2\delta}f^{-2}
\end{equation}
uniformly for all $\eta \in [\eta_2,\eta_1]$. Similarly, by \eqref{3.3} we can show that
\begin{equation} \label{3.8}
	\max_{\b x, \b y\in \bb X}	|G_{\b x\b y}(z)| \prec N^{\delta}
\end{equation}
uniformly for all $\eta \in [\eta_2,\eta_1]$. By \eqref{3.6} and \eqref{3.8}, we see that 
\[
\frac{d (G_{\b e \b x}(z))}{\dd \eta} \prec \frac{\sqrt{\im G_{\b e \b e}(z) \im G_{\b x \b x}(z)}}{\eta} \prec N^{\delta} f^{-1}\eta^{-1}\,,
\]
and thus
\begin{equation} \label{3.9}
\max_{\b x \in \bb X}	|G_{\b e \b x}(z) |\prec N^{2\delta}f^{-1}
\end{equation}
uniformly for all $\eta \in [\eta_2,\eta_1]$. Observe that \eqref{3.7} -- \eqref{3.9} prove \eqref{3.0} and \eqref{3.1} for all $z=E+\ii \eta$ satisfying  $\eta \in [\eta_2,\eta_1]$. We conclude the proof of Theorem \ref{thm1.3}  by induction and the fact that $\xi$ can be arbitrarily small.

The rest of this section focuses on the proofs of \eqref{3.2} -- \eqref{3.3}. As an input from Proposition \ref{refthm1}, we have the  entrywise bound
\begin{equation} \label{3.192}
	\max_{ij}|G_{ij}(z)| \prec 1\,.
\end{equation}

\subsection{Proof of \eqref{3.2}} \label{sec3.1}
 Fix $n \in \bb N_+$, and set $\cal P_1\deq \|G_{\b e\b e}-f^{-1}\|_{2n}$. We shall show that
\begin{equation} \label{3.122}
\cal P_1^{2n}=\bb E |G_{\b e\b e}-f^{-1}|^{2n}	\prec \sum_{a=1}^{2n} f^{-2a} \cal P_1^{2n-a}\eqd \cal E_1\,,
\end{equation}
which trivially implies \eqref{3.2}. By resolvent identity, the first relation of \eqref{3.0}, and the fact that $z\in \b D_{\tau}$ is bounded, we have
\begin{equation} \label{qqq}
	G_{\b e\b e}-f^{-1}=f^{-1}((f-z)G_{\b e \b e}-1)+O_{\prec}(f^{-2})=-f^{-1}(HG)_{\b e\b e} +O_{\prec}(f^{-2})\,.
\end{equation}
By Lemma \ref{lem:cumulant_expansion},  we get
\begin{equation} \label{3.12}
	\begin{aligned}
	\cal P_1^{2n}&=-\frac{1}{fN^{1/2}}\sum_{ij}\bb E  H_{ij}G_{j\b e}(G_{\b e\b e}-f^{-1})^{n-1}(G^*_{\b e\b e}-f^{-1})^{n}+O_{\prec}(\cal E_1)\\
	&=-\frac{1}{fN^{1/2}}\sum_{s=1}^{\ell}\sum_{ij} \cal C_{s+1}(H_{ij})\bb E \partial_{ij}^{s} (G_{j\b e}(G_{\b e\b e}-f^{-1})^{n-1}(G^*_{\b e\b e}-f^{-1})^{n})+O_{\prec}(\cal E_1)\\
	&\eqd 	\sum_{s=1}^{\ell} L_s+O_{\prec}(\cal E_1)\,.
\end{aligned}
\end{equation}
The estimate of $L_1$ is relatively easy: as $\cal C_2(H_{ij})=N^{-1}(1+O(\delta_{ij}))$, we get from \eqref{diff} that
\begin{equation} \label{3/15}
\begin{aligned}
	L_1=&\,\frac{1}{fN^{3/2}}\sum_{ij} (1+O(\delta_{ij}))\bb E (G_{ji}G_{j\b e}+G_{jj}G_{i\b e})(G_{\b e\b e}-f^{-1})^{n-1}(G^*_{\b e\b e}-f^{-1})^{n}\\
	&+\frac{1}{fN^{3/2}}\sum_{ij} (1+O(\delta_{ij})) 2(n-1)\bb EG_{j\b e} G_{\b e i }G_{\b e j}(G_{\b e\b e}-f^{-1})^{n-2}(G^*_{\b e\b e}-f^{-1})^{n}\\
	&+\frac{1}{fN^{3/2}}\sum_{ij} (1+O(\delta_{ij})) 2n\bb EG_{j\b e} G^*_{\b e i }G^*_{\b e j}|G_{\b e\b e}-f^{-1}|^{2n-2}\eqd L_{1,1}+L_{1,2}+L_{1,3}\,.
\end{aligned}	
\end{equation}
By Proposition \ref{refthm1}, \eqref{3.0} and \eqref{3.1}, we have
\begin{equation*}
	\begin{aligned}
	L_{1,1}&=\frac{1}{fN^{3/2}}\sum_{ij} \bb E (G_{ji}G_{j\b e}+G_{jj}G_{i\b e})(G_{\b e\b e}-f^{-1})^{n-1}(G^*_{\b e\b e}-f^{-1})^{n}+O_{\prec}(\cal E_1)\\
	&=\frac{1}{fN}\sum_{j} \bb E (G_{j\b e}^2+G_{jj}G_{\b e\b e})(G_{\b e\b e}-f^{-1})^{n-1}(G^*_{\b e\b e}-f^{-1})^{n}+O_{\prec}(\cal E_1)\\
	&\prec f^{-2}\bb E |G_{\b e\b e}-f^{-1}|^{2n-2}+O_{\prec}(\cal E_1)\prec \cal E_1\,.
\end{aligned}
\end{equation*}
Similarly, we have
\[
L_{1,2}=\frac{1}{fN}\sum_j \bb E  G_{j\b e}^2 G_{\b e i}(G_{\b e\b e}-f^{-1})^{n-2}(G^*_{\b e\b e}-f^{-1})^{n}+O_{\prec}(\cal E_1)\prec \cal E_1
\]
and $L_{1,3}\prec \cal E_1$. As a consequence,
\begin{equation} \label{L_1}
	L_1\prec \cal E_1\,.
\end{equation}
When $s\geq 2$, the above straight-forward estimates fail. For instance, one dangerous term in $L_2$ is
\begin{equation}  \label{3.13}
	\frac{1}{fN^{3/2}q}\sum_{ij} a_{ij} \bb E G_{j\b e}G^{*}_{ii}G^*_{j \b e}G^*_{j\b e}|G_{\b e \b e}-f^{-1}|^{2n-2}\,,
\end{equation}
where $a_{ij}$ is uniformly bounded in $i,j$. Naively, even make use of the Ward identity, we would have
\[
\eqref{3.13} \prec \frac{1}{f^4N^{1/2}\eta q}\bb E |G_{\b e \b e}-f^{-1}|^{2n-2}\prec \frac{1}{N^{1/2}\eta q}\cal E_1\,.
\]
As there is not always true that $N^{1/2}\eta q\geq 1$, we cannot bound the above by $O_{\prec}(\cal E_1)$ as desired. This is the main difficulty we encounter for proving the isotropic local law for sparse matrices. In the sequel, we introduce the notion of abstract polynomials of Green functions, which allows us to expand recursively through Lemma \ref{lem:cumulant_expansion}. To simplify notation, we drop the complex conjugates in $L_s$ (which play no role in the subsequent analysis), and we shall prove that
\begin{equation} \label{3.14}
	\widetilde{L}_s\deq -\frac{1}{fN^{1/2}}\sum_{s=1}^{\ell}\sum_{ij} \cal C_{s+1}(H_{ij})\bb E \partial_{ij}^{s} (G_{j\b e}(G_{\b e\b e}-f^{-1})^{2n-1})\prec\cal E_1
\end{equation}
uniformly for all fixed $s\geq 2$.

\subsubsection{Abstract polynomial of Green functions}
\begin{definition} \label{def3.1}
	Let $\{i_1,i_2,...\}$ be an infinite set of formal indices. To $\nu,\nu_1,\sigma,\omega\in \bb N$, $\theta_1,\theta_2 \in \bb R$, $x_1,y_1,...,x_\sigma,y_\sigma,z_1,...,z_\omega\in \{i_1,...i_\nu\}$, and a family $(a_{i_1,...,i_\nu})_{1\leq i_1,...,i_\nu \leq N}$ of uniformly bounded complex numbers we assign a formal monomial 
	\begin{equation} \label{3.15}
	W=a_{i_1,...,i_\nu}N^{-\theta_1}f^{-\theta_2}G_{x_1y_1}\cdots G_{x_\sigma y_\sigma}G_{z_1 \b e}\cdots G_{z_\omega \b e} (G_{\b e \b e}-f^{-1})^{\nu_1}
	\end{equation}
	We denote $\nu(W) = \nu$, $\nu_1(W) = \nu_1$, $\sigma(W) = \sigma$, $\omega(W) \deq \omega$, $\theta_1(W) = \theta_1$ and $\theta_2(W)=\theta_2$. We use $\nu_2(W)$ to denote the number of indices that appear exactly twice and as off-diagonal indices in the Green functions of $W$.  We denote by $\cal W$ the set of formal monomials $W$ of the form \eqref{3.15}.
	
	 Let $O(W)\subset \{i_1,...,i_\nu\}$ be the set of indices that appear odd many times in the Green functions of $W$.  We denote by $\cal W_o\subset\cal W$ the set of formal monomials such that $O(W)\neq \emptyset$ for all $W\in \cal W_o$. 
\end{definition}

\begin{definition} \label{def:evaluation}
	We assign to each monomial $W \in  \cal W$ with its \emph{evaluation}
	\begin{equation*}
		W_{i_1,\dots,i_{\nu}} \equiv W_{i_1,\dots,i_{\nu}}(z)\,,
	\end{equation*}
	which is a random variable depending on an $\nu$-tuple $(i_1,\dots,i_{\nu})\in \{1,2,\dots,N\}^{\nu}$. It is obtained by replacing, in the formal monomial $W$, the formal indices $i_1,\dots,i_{\nu_1}$ with the integers $i_1,\dots,i_{\nu_1}$ and the formal variables $G_{xy},G_{z\b e}, G_{\b e \b e}$ with elements of the Green's function \eqref{2.1} with parameter $z\in \bf D_\tau$. We define
	\begin{equation*} 
		\cal M (W) \deq \sum_{i_1,\dots,i_{\nu}}  W_{i_1,\dots,i_{\nu}}\,.
	\end{equation*}
\end{definition}

\begin{lemma} \label{lemma 3.4}
Let $W\in \cal W$. We have
\begin{equation} \label{3.16}
\bb E \cal M(W) \prec N^{-\theta_1+\nu}f^{-\theta_2} \cdot (\phi f^{-1})^\omega \Big(\frac{1}{N\eta}\Big)^{\nu_2/2} \cdot \bb E |(G_{\b e\b e}-f^{-1})^{\nu_1}| \,.
\end{equation}
\end{lemma}

\begin{proof}
	The factors $N^{-\theta_1+\nu}f^{-\theta_2}$ and $\bb E |(G_{\b e\b e}-f^{-1})^{\nu_1}|$ on RHS of \eqref{3.16} is self-explanatory. Let us explain the rest contributions. The factor $(\phi f^{-1})^\omega$ comes from the second relation of \eqref{3.0}. The factor $(N\eta)^{-\nu_2/2}$ comes from Lemma \ref{ward}, \eqref{3.0}, \eqref{3.1} and \eqref{3.192}. This finishes the proof.
\end{proof}

\begin{example}
	Let us illustrate the above definitions and Lemma \ref{lemma 3.4} with an example. Consider
	\[
	W=N^{-2}f^{-1/2}G_{i_1i_1}G_{i_2i_4}G_{i_1i_5}G_{i_2i_3}G_{i_3i_5}G_{i_1\b e}^3G_{i_2\b e}(G_{\b e \b e}-f^{-1})^6\,.
	\]
	It is clear that $\nu(W)=5$, $\nu_1(W)=6$, $\sigma(W)=5$, $\omega(W)=4$, $\theta_1(W)=2$ and $\theta_2=1/2$. The index 2 appears 3 times in $W$, and the index $4$ appears once in $W$. As a result, $O(W)=\{i_2,i_4\}$.  In addition, we see that $\nu_2(W)=2$, with corresponding indices $i_3,i_5$.
	
	To see $\bb E \cal M(W)$ can be bounded as in \ref{lemma 3.4}, note that
	\begin{align*}
		\bb E \cal M(W)&=N^{-2}f^{-1/2}\sum_{i_1,...i_5}\bb E G_{i_1i_1}G_{i_2i_4}G_{i_1i_5}G_{i_2i_3}G_{i_3i_5}G_{i_1\b e}^3G_{i_2\b e}(G_{\b e \b e}-f^{-1})^6\\
		&\prec N^{-2}f^{-1/2}\sum_{i_1,...i_5}\bb E |G_{i_2i_3}G_{i_3i_5}G_{i_1\b e}^3G_{i_2\b e}(G_{\b e \b e}-f^{-1})^6|\\
		&\prec  N^{-2}f^{-1/2}\cdot (\phi f^{-1})^4 \sum_{i_1,...i_5}\bb E |G_{i_2i_3}G_{i_3i_5}(G_{\b e \b e}-f^{-1})^6|\\
		&\prec N^{-2}f^{-1/2}\cdot (\phi f^{-1})^4 \cdot \frac{N^4}{\eta} \bb E |(G_{\b e \b e}-f^{-1})^6|=N^{-2+5}f^{-1/2}\cdot (\phi f^{-1})^4 \cdot \frac{1}{N\eta}\cdot\bb E |(G_{\b e \b e}-f^{-1})^6|\,,
	\end{align*}
	which agrees with the RHS of \eqref{3.16}. Here in the second line we used $G_{i_1i_1}G_{i_2i_4}G_{i_1i_5}\prec 1$, which is from \eqref{3.192}; in the third line we used $G_{i_1\b e}^3G_{i_2\b e}\prec (\phi f^{-1})^4$ from \eqref{3.0}; in the last line we used
	\begin{align*}
		\sum_{i_1,...i_5} |G_{i_2i_3}G_{i_3i_5}| &\prec N^{2}\sum_{i_2,i_3,i_5} |G_{i_2i_3}G_{i_3i_5}| \prec N^{3}\sum_{i_2,i_3} |G_{i_2i_3}|^2+N^3\sum_{i_3i_5}|G_{i_3i_5}|^2\\
		&\prec N^3 \sum_{i_2} \frac{\im G_{i_2i_2}}{\eta}+ N^3 \sum_{i_5} \frac{\im G_{i_5i_5}}{\eta} \prec \frac{N^4}{\eta}\,,
	\end{align*}
	which makes use of Lemma \ref{ward} and \eqref{3.192}.
\end{example}

We have the following improvement of Lemma \ref{lemma 3.4} for $W \in \cal W_o$, which we delay the proof to Section \ref{section 3.1.2}.

\begin{lemma} \label{lemma 3.5}
	Let $W\in \cal W_o$. We have
	\[
	\bb E \cal M(W) \prec N^{-\theta_1+\nu}f^{-\theta_2}  (\phi f^{-1})^\omega   \phi N^{-1/2}\sum_{d=0}^{\nu_1}f^{-2d}\bb E |(G_{\b e\b e}-f^{-1})|^{\nu_1-d}\eqd \cal E_2(W)
	\]
\end{lemma}

Given Lemma \ref{lemma 3.5}, the proof of \eqref{3.14} becomes a relative simple matter. More precisely, by Lemma \ref{lemma 2.3}, $\widetilde{L}_s$ is a sum of finite many terms in the form
\begin{equation} \label{3.18}
\frac{1}{fN^{3/2}q^{s-1}}\sum_{ij} a_{ij}\bb E (\partial^{s_1}_{ij}G_{j\b e})\Big(\prod_{b=2}^{a}\partial_{ij}^{s_b} G_{\b e\b e}\Big) (G_{\b e \b e} -f^{-1})^{2n-a}
\end{equation}
where $s_1\geq 0$, $1\leq a\leq 2n-1$, $s_2,...,s_a\geq 1$, and $s_1+\cdots+s_a=s$. By \eqref{diff}, we see that
\begin{equation} \label{3.19}
\eqref{3.18}= \sum_{k=1}^\ell\bb E \cal M(W_k)
\end{equation}
for some fixed $\ell$, and each $W_k$ is in the form of \eqref{3.15}. Also by \eqref{diff}, it is clear that for every $W_k$, either $i$ or $j$ appears odd many times in its Green functions. In other words, we have $W_k\in \cal W_o$. In addition, the parameters of $W_k$ satisfies $\nu=2$, $\nu_1=2n-a$, $\omega=2a-1$, $\theta_1=3/2$, $\theta_2=s$. Thus Lemma \ref{lemma 3.5} shows that
\begin{equation*}
\begin{aligned}
\bb E \cal M(W_k)&\prec N^{-3/2+2}f^{-s} (\phi f^{-1})^{2a-1} \phi N^{-1/2}\sum_{d=0}^{2n-a} f^{-2d}\bb E |(G_{\b e\b e}-f^{-1})|^{2n-a-d}\\
 &\prec \phi^{2a}f^{1-s}\sum_{d=a}^{2n}f^{-2d}\bb E |(G_{\b e\b e}-f^{-1})|^{2n-d}\,,
\end{aligned}
\end{equation*}
where in the second step we used $q\asymp f$. Since $s\geq 2$, $a\leq s+1$, and $\phi\leq f^{0.01}$, we have $\phi^{2a}f^{1-s}\leq 1$. As a result, 
\begin{equation} \label{3.23}
\bb E \cal M(W_k)\prec\sum_{d=a}^{2n}f^{-2d}\bb E |(G_{\b e\b e}-f^{-1})|^{2n-d} \prec \cal E_1
\end{equation}
for all $W_k$ in \eqref{3.19}. This concludes \eqref{3.14}. Combining \eqref{3.122}, \eqref{3.12}, \eqref{L_1} and \eqref{3.14} yields the proof of \eqref{3.2}.

\subsubsection{Proof of Lemma \ref{lemma 3.5}} \label{section 3.1.2}

For $W\in \cal W_o$, let us denote
\begin{equation} \label{3.21}
\cal E_3(W)\deq N^{-\theta_1+\nu}f^{-\theta_2}  \Big(\frac{1}{N\eta}\Big)^{\nu_2/2}(\phi f^{-1})^\omega\sum_{d=0}^{\nu_1}f^{-2d}\bb E |(G_{\b e\b e}-f^{-1})|^{\nu_1-d} \,.
\end{equation}
From \eqref{3.16}, it is easy to see that 
\begin{equation} \label{3.22}
	\bb E \cal M(W) \prec \cal E_3(W)\,.
\end{equation}
As $\cal E_3(W)\phi N^{-1/2}\leq \cal E_2(W)$, our task here is to improve \eqref{3.21} by a factor of $\phi N^{-1/2}$. We shall show the following iterating result.

\begin{lemma} \label{lemma 3.6}
Let $W\in \cal W_o$. We have
\begin{equation} \label{3.24}
\bb E \cal M(W)\prec \sum_{k=1}^\ell \bb E \cal M(W_k)+\cal E_2(W)
\end{equation}
for some fixed $\ell$. Here $W_k\in \cal W_o$ and satisfy $\theta_1(W_k)-\nu(W_k)\geq \theta_1(W)-\nu(W)$ for all $k=1,2,...,\ell$. In addition, each $W_k$ satisfies one of the following conditions.
\begin{itemize}
	\item[] Condition 1. $\theta_2(W_k)=\theta_2(W)$, $\nu_2(W_k)\geq \nu_2(W)+1$, $\omega(W_k)=\omega(W)$, $\nu_1(W_k)=\nu_1(W)$.
	
	\item[] Condition 2. $\theta_2(W_k)=\theta_2(W)$, $\nu_2(W_k)\geq \nu_2(W)+1$, $\omega(W_k)=\omega(W)+2$, $\nu_1(W_k)=\nu_1(W)-1$.
	
		\item[] Condition 3. There exists $\mathfrak s \geq 2$, and $0\leq \mathfrak d \leq \mathfrak s\wedge \nu_1(W)$ such that $\theta_2(W_k)=\theta_2(W)+\mathfrak s-1
		$, $\nu_2(W_k)\geq\nu_2(W)$, $\omega(W_k)=\omega(W)+2\mathfrak d$, $\nu_1(W_k)=\nu_1(W)-\mathfrak d$.
\end{itemize}
\end{lemma}

We remark that the three conditions in Lemma \ref{lemma 3.6} arises from hitting different terms with a derivation in the cumulant expansion. By Lemma \ref{lemma 3.6} and applying \eqref{3.21}, \eqref{3.22} for $W=W_k$, as well as $\xi\leq \tau/100$, we see that for every $W_k$ in \eqref{3.24}, we always have
\[
\cal E_3(W_k)\leq\cal E_3(W)\cdot \phi^4\Big(\frac{1}{\sqrt{N\eta}}+\frac{1}{f}\Big)\leq \cal E_3(W)\cdot N^{-\tau/4}\,, \quad \mbox{and} \quad \cal E_2(W_k)\leq \cal E_2(W)\,.
\]
Thus for any given $W\in \cal W_o$, we can apply \eqref{3.24} finitely many times and obtain $\bb E \cal M(W)\prec \cal E_2(W)$ as desired.

\begin{proof}[Proof of Lemma \ref{lemma 3.6}]
	Let $W\in \cal W_o$ in the form of \eqref{3.15}. W.L.O.G.\,we can assume $x_1\in O(W)$ and $x_1\not\equiv y_1$, or $z_1\in O(W)$. Let us work under the first assumption; by examining the proof, it is clear that the case $z_1\in O(W)$ works in a very similar fashion. Note that we have the identity
	\[
	z\ul{G}G_{x_1 y_1}=\ul{AG}G_{x_1 y_1}-G_{x_1y_1}=\ul{G}(AG)_{x_1 y_1}-\ul{G}\delta_{x_1 y_1}\,.
	\]
	As a result, we get
	\begin{equation*}
		\begin{aligned}
\bb E \cal M(W)&=		\sum_{i_1,...,i_\nu}a_{i_1,...,i_\nu}N^{-\theta_1}f^{-\theta_2}\bb E\ul{AG}G_{x_1 y_1}G_{x_2y_2}\cdots G_{x_\sigma y_\sigma}G_{z_1 \b e}\cdots G_{z_\omega \b e} (G_{\b e \b e}-f^{-1})^{\nu_1}\\
&\ -\sum_{i_1,...,i_\nu}a_{i_1,...,i_\nu}N^{-\theta_1}f^{-\theta_2}\bb E\ul{G}(AG)_{x_1 y_1} G_{x_2y_2}\cdots G_{x_\sigma y_\sigma}G_{z_1 \b e}\cdots G_{z_\omega \b e} (G_{\b e \b e}-f^{-1})^{\nu_1}\\
&\ +\sum_{i_1,...,i_\nu}a_{i_1,...,i_\nu}N^{-\theta_1}f^{-\theta_2}\bb E \ul{G}\delta_{x_1 y_1} G_{x_2y_2}\cdots G_{x_\sigma y_\sigma}G_{z_1 \b e}\cdots G_{z_\omega \b e} (G_{\b e \b e}-f^{-1})^{\nu_1}\\
&\eqd \mbox{(I)+(II)+(III)}\,.
		\end{aligned}
	\end{equation*}
Observe that
	\begin{align*}
\hspace*{-0.3cm}\mbox{(I)}&=\sum_{i_1,...,i_\nu}a_{i_1,...,i_\nu}N^{-\theta_1}f^{-\theta_2}\bb E\ul{HG}G_{x_1 y_1}G_{x_2y_2}\cdots G_{x_\sigma y_\sigma}G_{z_1 \b e}\cdots G_{z_\omega \b e} (G_{\b e \b e}-f^{-1})^{\nu_1}\nonumber\\
&\quad+\sum_{i_1,...,i_\nu}a_{i_1,...,i_\nu}N^{-\theta_1-1}f^{-\theta_2+1}\bb EG_{\b e \b e}G_{x_1 y_1}G_{x_2y_2}\cdots G_{x_\sigma y_\sigma}G_{z_1 \b e}\cdots G_{z_\omega \b e} (G_{\b e \b e}-f^{-1})^{\nu_1}\nonumber\\
&=\sum_{i_1,...,i_\nu}a_{i_1,...,i_\nu}N^{-\theta_1}f^{-\theta_2}\bb E\ul{HG}G_{x_1 y_1}G_{x_2y_2}\cdots G_{x_\sigma y_\sigma}G_{z_1 \b e}\cdots G_{z_\omega \b e} (G_{\b e \b e}-f^{-1})^{\nu_1}+O_{\prec}(\cal E_2(W))\nonumber\\
&\eqd \mbox{(I')}+O_{\prec}(\cal E_2(W))\,,
\end{align*}
where in the second step we used $G_{\b e \b e}\prec\phi f^{-1}$ and estimate the rest factors similar to Lemma \ref{lemma 3.4}. In addition, by \eqref{3.0}, we have $((A-H)G)_{x_1 y_1}=fN^{-1/2}G_{\b e y_1}\prec \phi N^{-1/2}$, and thus
\begin{align*}
\mbox{(II)}=&-\sum_{i_1,...,i_\nu}a_{i_1,...,i_\nu}N^{-\theta_1}f^{-\theta_2}\bb E\ul{G}(HG)_{x_1 y_1} G_{x_2y_2}\cdots G_{x_\sigma y_\sigma}G_{z_1 \b e}\cdots G_{z_\omega \b e} (G_{\b e \b e}-f^{-1})^{\nu_1}+O_{\prec}(\cal E_2(W))\\
=&:\mbox{(II')}+O_{\prec}(\cal E_2(W))\,.
\end{align*}
As $x_1\not\equiv y_1$, one can easily show that (III)$\prec \cal E_2(W)$. Hence we have
\begin{equation} \label{3.242}
	\mathbb E \cal M(W)=\mbox{(I')+(II')}+O_{\prec}(\cal E_2(W))
\end{equation}

Now let us expand (I') and (II') via cumulant expansion. Let us abbreviate 
\[
X\deq G_{x_2y_2}\cdots G_{x_\sigma y_\sigma}G_{z_1 \b e}\cdots G_{z_\omega \b e} (G_{\b e \b e}-f^{-1})^{\nu_1}\,.
\] 
By Lemma \ref{lem:cumulant_expansion}, we have
\begin{equation*}
\begin{aligned}
\mbox{(I')}&=\sum_{s=1}^\ell\sum_{i_1,...,i_\nu,i,j}a_{i_1,...,i_\nu}N^{-\theta_1-1}f^{-\theta_2}\cal C_{s+1}(H_{ij})\bb E\partial _{ij}^s(G_{ij}G_{x_1 y_1}X)+O_{\prec}(\cal E_2(W)) \\
	&\eqd \sum_{s=1}^\ell L^{(1)}_s+O_{\prec}(\cal E_2(W))\,.
\end{aligned}
\end{equation*}
Note that by \eqref{diff}, we have $x_1\in O(W)$ for all terms in $L^{(1)}_{s}, s\geq 1$. In addition, taking $\partial_{ij}$ will not decrease the value of $\nu_2$, as $i,j$ are new indices. By \eqref{diff}, we see that 
\begin{equation*}
	\begin{aligned}
	L_1^{(1)}&=-\sum_{i_1,...,i_\nu}a_{i_1,...,i_\nu}N^{-\theta_1}f^{-\theta_2}\bb E\ul{G}^2G_{x_1 y_1}X -\sum_{i_1,...,i_\nu,i,j}a_{i_1,...,i_\nu}N^{-\theta_1-2}f^{-\theta_2}\bb EG^2_{ij}G_{x_1 y_1}X\\
	 &\quad +\sum_{i_1,...,i_\nu,i,j}a_{i_1,...,i_\nu}N^{-\theta_1-2}f^{-\theta_2}\bb EG_{ij}\partial _{ij}(G_{x_1 y_1}X)+O_{\prec}(\cal E_2(W))\\
	  &\eqd L^{(1)}_{1,1}+L^{(1)}_{1,2}+L^{(1)}_{1,3}+O_{\prec}(\cal E_2(W))\,.
	\end{aligned}
\end{equation*}
It is easy to see that $L_{1,2}^{(1)}$ is in the form of $\bb E \cal M(W^{(1)}_{1,2})$, where  $W^{(1)}_{1,2}\in \cal W_o$ and it satisfies Condition 1. In addition, $L_{1,3}^{(1)}=\sum_{k=1}^\ell \bb E \cal M(W_{1,3,k}^{(1)})$, where $k$ is fixed and each $W_{1,3,k}^{(1)}\in \cal W_o$. Here if $\partial_{ij}$ is applied to a factor of $(G_{\b e \b e}-f^{-1})$, then the corresponding $W_{1,3,k}^{(1)}$ satisfies Condition 2; otherwise, it satisfies Condition 1. When $s\geq 2$, by Lemma \ref{lemma 2.3} and \eqref{diff}, we see that
\[
L^{(1)}_s=\sum_{k=1}^\ell \bb E \cal M(W_{s,k}^{(1)})
\]
where each $W^{(1)}_{s,k}\in W_o$ satisfies Condition 3 with $\mathfrak s=s$, and $\mathfrak d$ corresponds to the number of $(G_{\b e \b e}-f^{-1})$ that $\partial^s_{ij}$ hit on. In summary, we have
\begin{equation} \label{3.25}
	\mbox{(I')}=-\sum_{i_1,...,i_\nu}a_{i_1,...,i_\nu}N^{-\theta_1}f^{-\theta_2}\bb E\ul{G}^2G_{x_1 y_1}X+\sum_{k=1}^\ell \bb E \cal M(W_{k}^{(1)})+O_{\prec}(\cal E_2(W))
\end{equation}
where each $W_k^{(1)}\in \cal W_o$ and satisfies Condition 1, 2, or 3.

The computation of (II') follows in a similar fashion. By Lemma \ref{lem:cumulant_expansion}, we have
\begin{equation*}
	\begin{aligned}
		\mbox{(II')}&=-\sum_{s=1}^\ell\sum_{i_1,...,i_\nu,i}a_{i_1,...,i_\nu}N^{-\theta_1}f^{-\theta_2}\cal C_{s+1}(H_{x_1 i})\bb E\partial _{x_1 i}^s(G_{i y_1}\ul{G}X)+O_{\prec}(\cal E_2(W)) \\
		&\eqd \sum_{s=1}^\ell L^{(2)}_s+O_{\prec}(\cal E_2(W))\,.
	\end{aligned}
\end{equation*}
In addition, taking $\partial_{x_1i}$ will not decrease the value of $\nu_2$, as $i$ is a new index, and $x_1$ originally appears odd many times in the Green functions. Note that by \eqref{diff}, when $s\geq 1$ is odd, we have $x_1\in O(W)$ for all terms in $L^{(2)}_{s}$; when $s\geq 2$ is even, we have $i\in O(W)$ for all terms in $L^{(2)}_{s}$. By \eqref{diff}, we see that 
\begin{equation*}
	\begin{aligned}
		L_1^{(2)}&=\sum_{i_1,...,i_\nu}a_{i_1,...,i_\nu}N^{-\theta_1}f^{-\theta_2}\bb E\ul{G}^2G_{x_1 y_1}X +\sum_{i_1,...,i_\nu,i}a_{i_1,...,i_\nu}N^{-\theta_1-1}f^{-\theta_2}\bb E\ul{G}G_{ i x_1}G_{i y_1}X\\
		&\quad +\sum_{i_1,...,i_\nu,i}a_{i_1,...,i_\nu}N^{-\theta_1-1}f^{-\theta_2}\bb EG_{i y_1}\partial _{x_1 i}(\ul{G}X)+O_{\prec}(\cal E_2(W))\\
		&\eqd L^{(2)}_{1,1}+L^{(2)}_{1,2}+L^{(2)}_{1,3}+O_{\prec}(\cal E_2(W))\,.
	\end{aligned}
\end{equation*}
The terms $L^{(2)}_{1,2}$, $L^{(2)}_{1,3}$, and $L_{s}^{(2)}$ can be handled exactly like $L^{(1)}_{1,2}$, $L^{(1)}_{1,3}$, and $L_{s}^{(1)}$. As a result, we have 
\begin{equation} \label{3.26}
	\mbox{(II')}=\sum_{i_1,...,i_\nu}a_{i_1,...,i_\nu}N^{-\theta_1}f^{-\theta_2}\bb E\ul{G}^2G_{x_1 y_1}X+\sum_{k=1}^\ell \bb E \cal M(W_{k}^{(2)})+O_{\prec}(\cal E_2(W))
\end{equation}
where each $W_k^{(2)}\in \cal W_o$ and satisfies Condition 1, 2, or 3. 

Note that we have a cancellation between the leading contributions of \eqref{3.25} and \eqref{3.26}. Combining \eqref{3.242} -- \eqref{3.26} yields the desired result. 
\end{proof}

\subsection{Proofs of \eqref{3.2.2} and \eqref{3.3}}
The proofs of \eqref{3.2.2} and \eqref{3.3} are similar to that of  \eqref{3.2}. More precisely, we follow the mechanism that for each term that carries an index that appears odd many times, we can always expand around it, and the resulting terms are either small enough, or they again carry an odd index and allow further expansion. We shall give the main steps, with an emphasis on the differences.

\subsubsection{Proof of \eqref{3.2.2}}
Let $\b x \in \mathbb X$ and fix $n\in \mathbb N_+$.  Abbreviate $\cal P_2\deq \| G_{\b e \b x}\|_{2n}$, and it suffices to show that
\begin{equation*}
	\cal P_2^{2n}=\bb E |G_{\b e \b x}|^{2n} \prec \sum_{a=1}^{2n} f^{-a} \cal P_2^{2n-a}\eqd \cal E_4\,.
\end{equation*}
By resolvent identity and the second relation of \eqref{3.0}, we have
\begin{equation} \label{qqqq}
	\begin{aligned}
G_{\b e \b x}&=-f^{-1}(z-f)G_{\b e \b x}+O_{\prec}(\phi f^{-2})\\
&=-f^{-1}\big((HG)_{\b e \b x} - \langle \b e , \b x\rangle\big)+O_{\prec}(\phi f^{-2})=-f^{-1}(HG)_{\b e \b x}+O_{\prec}(\phi f^{-2})\,.
	\end{aligned}
\end{equation}
Then we get from Lemma \ref{lem:cumulant_expansion} that
\begin{equation} \label{3.28}
	\begin{aligned}
		\cal P_2^{2n}
		&=-f^{-1}N^{-1/2}\sum_{ij}\bb E  H_{ij}G_{j\b x}G_{\b e\b e}^{n-1}(G^*_{\b e\b e})^{n}+O_{\prec}(\cal E_4)\\
		&=-f^{-1}N^{-1/2}\sum_{s=1}^{\ell}\sum_{ij} \cal C_{s+1}(H_{ij})\bb E \partial_{ij}^{s} (G_{j\b x}G_{\b e\b x}^{n-1}(G^*_{\b e\b x})^{n})+O_{\prec}(\cal E_4)\\
		&\eqd 	\sum_{s=1}^{\ell} L^{(3)}_s+O_{\prec}(\cal E_4)\,.
	\end{aligned}
\end{equation}
Similar to \eqref{L_1}, it is not hard to see that $L_1^{(3)} \prec \cal E_4.$ To estimate the rest contributions in \eqref{3.28}, we again drop the complex conjugates for simplicity, and we shall prove  \eqref{3.2.2} by showing that 
\begin{equation} \label{3.29}
	\widetilde{L}^{(3)}_s\deq -f^{-1}N^{-1/2}\sum_{s=1}^{\ell}\sum_{ij} \cal C_{s+1}(H_{ij})\bb E \partial_{ij}^{s} (G_{j\b x}G_{\b e\b x}^{2n-1})\prec\cal E_4
\end{equation}
for all fixed $s \geq 2$.

Similar to Definition \ref{def3.1}, Let $\{i_1,i_2,...\}$ be an infinite set of formal indices. To $\nu,\nu_1,\sigma,\omega,\gamma\in \bb N$, $\theta_1,\theta_2 \in \bb R$, $x_1,y_1,...,x_\sigma,y_\sigma,z_1,...,z_\omega,w_1,...,w_\gamma\in \{i_1,...i_\nu\}$, and a family $(a_{i_1,...,i_\nu})_{1\leq i_1,...,i_\nu \leq N}$ of uniformly bounded complex numbers we assign a formal monomial 
\begin{equation} \label{3.30}
U=a_{i_1,...,i_\nu}N^{-\theta_1}f^{-\theta_2}G_{x_1y_1}\cdots G_{x_\sigma y_\sigma}G_{z_1 \b e}\cdots G_{z_\omega \b e} G_{w_1 \b x}\cdots G_{w_{\gamma} \b x} G^{\nu_1}_{\b e \b x}\,,
\end{equation}
where $\nu(U) = \nu$, $\nu_1(U) = \nu_1$, $\omega(U) \deq \omega$, $\gamma(U)=\gamma$, $\theta_1(U) = \theta_1$ and $\theta_2(U)=\theta_2$. We use $\nu_2(U)$ to denote the number of indices that appear exactly twice and as off-diagonal indices in the Green functions in $U$. We denote by $\cal U$ the set of formal monomials $U$ of the form \eqref{3.30}.

Let $O(U)\subset \{i_1,...,i_\nu\}$ be the set of indices that appear odd many times in the Green functions of $U$.  We denote by $\cal U_o\subset\cal U$ the set of formal monomials such that $O(U)\neq \emptyset$ for all $U\in \cal U_o$. In addition, we denote 
\[
\cal M(U)\deq \sum_{i_1,\dots,i_{\nu}}  U_{i_1,\dots,i_{\nu}}\,.
\]
Similar to Lemma \ref{lemma 3.4}, we have the following priori estimate. 
\begin{lemma} \label{lemma 3.7}
	Let $U\in \cal U$. We have
	\begin{equation} \label{3.31}
		\bb E \cal M(U) \prec N^{-\theta_1+\nu}f^{-\theta_2} \cdot (\phi f^{-1})^\omega \phi^\gamma \Big(\frac{1}{N\eta}\Big)^{\nu_2/2} \cdot \bb E |G_{\b e\b x}|^{\nu_1} \,.
	\end{equation}
\end{lemma}

Similar to Lemma \ref{lemma 3.5}, our main technical step towards proving \eqref{3.29} is to show the following improved estimate.

\begin{lemma} \label{lemma 3.8}
	Let $U\in \cal U_o$. We have
	\[
	\bb E \cal M(U) \prec N^{-\theta_1+\nu}f^{-\theta_2}  (\phi f^{-1})^\omega   \phi^\gamma \phi N^{-1/2}\sum_{d=0}^{\nu_1}f^{-d}\bb E |G_{\b e\b x}|^{\nu_1-d}\eqd \cal E_5(U)\,.
	\]
\end{lemma}

Finally, similar to Lemma \ref{lemma 3.6}, we prove Lemma \ref{lemma 3.8} by establishing the following result.

\begin{lemma} \label{lemma 3.9}
	Let $U\in \cal U_o$. We have
	\begin{equation} \label{3.331}
		\bb E \cal M(U)\prec \sum_{k=1}^\ell \bb E \cal M(U_k)+\cal E_5(U)
	\end{equation}
	for some fixed $\ell$. Here $U_k\in \cal T_o$ and satisfy $\theta_1(U_k)-\nu(U_k)\geq \theta_1(U)-\nu(U)$ for all $k=1,2,...,\ell$. In addition, each $U_k$ satisfies one of the following conditions.
	\begin{itemize}
		\item[] Condition 1. $\theta_2(U_k)=\theta_2(U)$, $\nu_2(U_k)\geq \nu_2(U)+1$, $\omega(U_k)=\omega(U)$, $\gamma(U_k)=\gamma(U)$, $\nu_1(U_k)=\nu_1(U)$.
		
		\item[] Condition 2. $\theta_2(U_k)=\theta_2(U)$, $\nu_2(U_k)\geq \nu_2(U)+1$, $\omega(U_k)=\omega(U)+1$, $\gamma(U_k)=\gamma(U)+1$,  $\nu_1(U_k)=\nu_1(U)-1$.
		
		\item[] Condition 3. There exists $\mathfrak s \geq 2$, and $0\leq \mathfrak d \leq \mathfrak s\wedge \nu_1(U)$ such that $\theta_2(U_k)=\theta_2(U)+\mathfrak s-1
		$, $\nu_2(U_k)\geq\nu_2(U)$, $\omega(U_k)=\omega(U)+\mathfrak d$, $\gamma(U_k)=\gamma(U)+\mathfrak d$, $\nu_1(U_k)=\nu_1(U)-\mathfrak d$.
	\end{itemize}
\end{lemma}

The proofs of Lemmas \ref{lemma 3.7} -- \ref{lemma 3.9} are very similar to those of Lemmas \ref{lemma 3.4} -- \ref{lemma 3.6}, and we shall omit the details. With the help of Lemma \ref{lemma 3.8}, one can easily follow the argument of \eqref{3.19} -- \eqref{3.23} to conclude \eqref{3.29}. This finishes the proof of \eqref{3.2.2}.

We conclude with a remark on the underlying reasons for the difference between \eqref{3.2} and \eqref{3.2.2}. Firstly, for the computation of $G_{\b e \b e}$, in the step of applying the resolvent identity \eqref{qqq}, we made use of the relation $\langle \b e ,\b e \rangle =1$; for the estimate of $G_{\b e \b x}$, we instead used $|\langle \b e ,\b x \rangle |\leq N^{-1/2}$ in \eqref{qqqq}. This explains the additional term $-f^{-1}$ on LHS of \eqref{3.2}. In addition, note that
\[
\partial_{ij}G_{\b e \b e} \prec \phi^2 f^{-2}\,, \quad \mbox{while} \quad \partial_{ij}G_{\b e \b x} \prec \phi^2 f^{-1}\,,
\]
and this is why the changes of $\omega$ differ between Condition 2 (also Condition 3) of Lemmas \ref{lemma 3.6} and \ref{lemma 3.9}. This also explains the difference  between the RHS of \eqref{3.2} and \eqref{3.2.2}.

\subsubsection{Proof of \eqref{3.3}} \label{sec3.2}
Let $\b x, \b y \in \bb X$. As the case $\b x,\b y\in \{\b e_1,...,\b e_N\}$ is covered by Proposition \ref{refthm1}, it suffices to assume $\b x=\b v\in \bb S^{N-1}_\perp$. Let fix $n\in \mathbb N_+$.  Abbreviate $\langle G_{\b x \b y}\rangle \deq G_{\b x\b y}-\langle \b x,
\b y\rangle m_{sc}$, and $\cal P_3\deq \| \langle G_{\b x \b y}\rangle\|_{2n}$, and it suffices to show that
\begin{equation} \label{3.34}
	\cal P_3^{2n}=\bb E |G_{\b x\b y}-\langle \b x,
	\b y\rangle m_{sc}|^{2n} \prec \sum_{a=1}^{2n} \phi^{3a}\Big(\frac{1}{(N\eta)^{1/3}}+\frac{1}{q^{1/3}}\Big)^a \cal P_3^{2n-a}\eqd \cal E_6\,.
\end{equation}
By resolvent identity, $1+zm_{sc}+m_{sc}^2=0$ and  $\b x \perp \b e$, we have
\begin{equation} \label{qqqqg}
		G_{\b x\b y}-\langle \b x,
		\b y\rangle m_{sc}=-m_{sc}((z+m_{sc})G_{\b x\b y}+\langle \b x,
		\b y))=-m_{sc}\big((HG)_{\b x \b y} +m_{sc} G_{\b x\b y}\big)\,.
\end{equation}
Similar to \eqref{3.12}, we get
	\begin{align}
		\cal P_3^{2n}
		&=-m_{sc}\sum_{ij}\bb E \b x_i  H_{ij}G_{j\b y}\langle G_{\b x \b y}\rangle^{n-1}\overline{\langle G_{\b x \b y}\rangle}^{n}-m_{sc}^2\bb EG_{\b x \b y}\langle G_{\b x \b y}\rangle^{n-1}\overline{\langle G_{\b x \b y}\rangle}^{n}\nonumber\\
		&=-m_{sc}\sum_{s=1}^{\ell}\sum_{ij} \cal C_{s+1}(H_{ij})\bb E \b x_i \partial_{ij}^{s} ( G_{j\b y}\langle G_{\b x \b y}\rangle^{n-1}\overline{\langle G_{\b x \b y}\rangle}^{n})-m_{sc}^2\bb EG_{\b x \b y}\langle G_{\b x \b y}\rangle^{n-1}\overline{\langle G_{\b x \b y}\rangle}^{n}+O_{\prec}(\cal E_6)\nonumber\\
		&\eqd 	\sum_{s=1}^{\ell} L^{(4)}_s-m_{sc}^2\bb EG_{\b x \b y}\langle G_{\b x \b y}\rangle^{n-1}\overline{\langle G_{\b x \b y}\rangle}^{n}+O_{\prec}(\cal E_6)\label{3.35}\,.
	\end{align}
Similar to \eqref{3/15}, one can go through the standard computation of $L_1^{(4)}$ (which reveals the semicircle distribution),  and show that
\[
L_1^{(4)}=m_{sc}\bb E\ul{G}G_{\b x \b y}\langle G_{\b x \b y}\rangle^{n-1}\overline{\langle G_{\b x \b y}\rangle}^{n}+O_{\prec}(\cal E_6)\,.
\] 
Together with Proposition \ref{refthm1}, we get a cancellation on RHS of \eqref{3.35} and obtain
\begin{equation} \label{3.37}
L_1^{(4)}	-m_{sc}^2\bb EG_{\b x \b y}\langle G_{\b x \b y}\rangle^{n-1}\overline{\langle G_{\b x \b y}\rangle}^{n} \prec \cal E_6\,.
\end{equation}
To estimate the rest contributions in \eqref{3.35}, we again drop the complex conjugates for simplicity, and we shall prove  \eqref{3.3} by showing that 
\begin{equation} \label{3.38}
	\widetilde{L}^{(4)}_s\deq -m_{sc}\sum_{ij} \cal C_{s+1}(H_{ij})\bb E \b x_i \partial_{ij}^{s} (G_{j\b y}\langle G_{\b x \b y}\rangle^{2n-1})\prec\cal E_6
\end{equation}
for all fixed $s \geq 2$.

Similar to Definition \ref{def3.1}, Let $\{i_1,i_2,...\}$ be an infinite set of formal indices. To $\nu,\nu_1,\sigma,\omega,\gamma\in \bb N$, $\theta_1,\theta_2 \in \bb R$, $i,x_1,y_1,...,x_\sigma,y_\sigma,z_1,...,z_\omega,w_1,...,w_\gamma\in \{i_1,...i_\nu\}$, and a family $(a_{i_1,...,i_\nu})_{1\leq i_1,...,i_\nu \leq N}$ of uniformly bounded complex numbers we assign a formal monomial 
\begin{equation} \label{3.40}
	V=a_{i_1,...,i_\nu}\b x_iN^{-\theta_1}f^{-\theta_2}G_{x_1y_1}\cdots G_{x_\sigma y_\sigma}G_{z_1 \b x}\cdots G_{z_\omega \b x} G_{w_1 \b y}\cdots G_{w_{\gamma} \b y} \langle G_{\b x \b y}\rangle^{2n-1}\,,
\end{equation}
where $\nu(V) = \nu$, $\nu_1(V) = \nu_1$, $\omega(V) \deq \omega$, $\gamma(V)=\gamma$, $\theta_1(V) = \theta_1$ and $\theta_2(V)=\theta_2$. We use $\nu_2(V)$ to denote the number of indices that appear exactly twice and as off-diagonal indices in the Green functions of $V$, excluding the index $i$. We denote by $\cal V$ the set of formal monomials $V$ of the form \eqref{3.40}.

Let $O(V)\subset \{i_1,...,i_\nu\}$ be the set of indices that appear odd many times in the Green functions of $V$.  We denote by $\cal V_o\subset\cal V$ the set of formal monomials such that $O(V)\neq \emptyset$ for all $V\in \cal V_o$. In addition, we denote 
\[
\cal M(V)\deq \sum_{i_1,\dots,i_{\nu}}  V_{i_1,\dots,i_{\nu}}\,.
\]
We have the following priori estimate. 
\begin{lemma} \label{lemma 3.10}
	Let $V\in \cal V$. We have
	\begin{equation} \label{3.41}
		\bb E \cal M(V) \prec N^{-\theta_1+\nu}f^{-\theta_2} \cdot  \phi^{\omega+\gamma} \Big(\frac{1}{N\eta}\Big)^{\nu_2/2} N^{-1/2} \bb E |\langle G_{\b x \b y}\rangle|^{\nu_1} \,.
	\end{equation}
\end{lemma}

\begin{proof}
	The factors $N^{-\theta_1+\nu}f^{-\theta_2}$ and $\bb E |\langle G_{\b x \b y}\rangle|^{\nu_1}$ on RHS of \eqref{3.41} is self-explanatory. Let us explain the rest contributions. The factor $\phi ^{\omega+\gamma}$ comes from \eqref{3.1}. The factor $(N\eta)^{\nu_2/2}$ comes from Lemma \ref{ward}, \eqref{3.1} and \eqref{3.192}. Finally, the factor $N^{-1/2}$ comes from $\sum_i |\b x_i|\leq N^{1/2}$. This finishes the proof.
\end{proof}

Similar to Lemma \ref{lemma 3.5}, our main technical step towards proving \eqref{3.29} is to show the following improved estimate.

\begin{lemma} \label{lemma 3.11}
	Let $V\in \cal V_o$. We have
	\[
	\bb E \cal M(V) \prec  N^{-\theta_1+\nu}f^{-\theta_2} \phi^{\omega+\gamma}  N^{-1/2} (\phi N^{-1/2})\sum_{d=0}^{\nu_1}\phi^{3d}\Big(\frac{1}{(N\eta)^{1/3}}+\frac{1}{q^{1/3}}\Big)^d\bb E |\langle G_{\b x \b y}\rangle|^{\nu_1-d}\eqd \cal E_7(V)\,.
	\]
\end{lemma}

Finally, similar to Lemma \ref{lemma 3.6}, we prove Lemma \ref{lemma 3.11} by establishing the following result.

\begin{lemma} \label{lemma 3.12}
	Let $V\in \cal V_o$. We have
	\begin{equation} \label{3.42}
		\bb E \cal M(V)\prec \sum_{k=1}^\ell \bb E \cal M(V_k)+\cal E_7(V)
	\end{equation}
	for some fixed $\ell$. Here $V_k\in \cal T_o$ and satisfy $\theta_1(V_k)-\nu(V_k)\geq \theta_1(V)-\nu(V)$ for all $k=1,2,...,\ell$. In addition, each $V_k$ satisfies one of the following conditions.
	\begin{itemize}
		\item[] Condition 1. $\theta_2(V_k)=\theta_2(V)$, $\nu_2(V_k)\geq \nu_2(V)+1$, $\omega(V_k)=\omega(V)$, $\gamma(V_k)=\gamma(V)$, $\nu_1(V_k)=\nu_1(V)$.
		
		\item[] Condition 2. $\theta_2(V_k)=\theta_2(V)$, $\nu_2(V_k)\geq \nu_2(V)+1$, $\omega(V_k)=\omega(V)+1$, $\gamma(V_k)=\gamma(V)+1$,  $\nu_1(V_k)=\nu_1(V)-1$.
		
		\item[] Condition 3. There exists $\mathfrak s \geq 2$, and $0\leq \mathfrak d \leq \mathfrak s\wedge \nu_1(V)$ such that $\theta_2(V_k)=\theta_2(V)+\mathfrak s-1
		$, $\nu_2(V_k)\geq\nu_2(V)$, $\omega(V_k)=\omega(V)+\mathfrak d$, $\gamma(V_k)=\gamma(V)+\mathfrak d$, $\nu_1(V_k)=\nu_1(V)-\mathfrak d$.
	\end{itemize}
\end{lemma}

The proof of Lemma \ref{lemma 3.12} is very close to that of Lemma \ref{lemma 3.6}, and we shall omit the details. We shall prove Lemma \ref{lemma 3.11} using Lemma \ref{lemma 3.12}.

\begin{proof}[Proof of Lemma \ref{lemma 3.11}]
	.For $V\in \cal V_o$, let us denote
	\begin{equation} \label{3.43}
		\cal E_8(V)\deq  N^{-\theta_1+\nu}f^{-\theta_2} \cdot  \phi^{\omega+\gamma} \Big(\frac{1}{N\eta}\Big)^{\nu_2/2} N^{-1/2} \sum_{d=0}^{\nu_1}\phi^{3d}\Big(\frac{1}{(N\eta)^{1/3}}+\frac{1}{q^{1/3}}\Big)^d \bb E |\langle G_{\b x \b y}\rangle|^{\nu_1-d}\,.
	\end{equation}
	From Lemma \ref{lemma 3.10}, it is easy to see that 
	\begin{equation} \label{3.44}
		\bb E \cal M(V) \prec \cal E_8(V)\,.
	\end{equation}
	As $\cal E_8(V)\phi N^{-1/2}\leq \cal E_7(V)$, our task here is to improve \eqref{3.43} by a factor of $\phi N^{-1/2}$. By Lemma \ref{lemma 3.12}, and applying \eqref{3.43}, \eqref{3.44} for $V=V_k$, we see that for every $V_k$ in \eqref{3.42}, we always have
	\[
	\cal E_8(V_k)\leq \cal E_8(V)\cdot \Big(\frac{1}{(N\eta)^{1/6}}+\frac{1}{q^{1/6}}\Big)\leq \cal E_8(V)\cdot N^{-\tau/6}\,, \quad \mbox{and} \quad \cal E_7(V_k)\leq \cal E_7(V)\,.
	\]
	Thus for any given $V\in \cal V_o$, we can apply \eqref{3.42} finitely many times and obtain $\bb E \cal M(V)\prec \cal E_7(V)$ as desired.
\end{proof}

At last, let us see how to deduce \eqref{3.38} from Lemma \ref{lemma 3.11}, and thus concludes the proof of \eqref{3.3}. By Lemma \ref{lemma 2.3}, $\widetilde{L}^{(4)}_s$ is a sum of finite many terms in the form
\begin{equation} \label{3.45}
	\frac{1}{N q^{s-1}}\sum_{ij} a_{ij}\bb E \b x_i (\partial^{s_1}_{ij}G_{j\b y})\Big(\prod_{b=2}^{a}\partial_{ij}^{s_b}\langle G_{\b x \b y }\rangle\Big) \langle G_{\b x \b y} \rangle^{2n-a}
\end{equation}
where $s_1\geq 0$, $1\leq a\leq 2n-1$, $s_2,...,s_a\geq 1$. By \eqref{diff}, we see that
\begin{equation} \label{3.46}
	\eqref{3.45}= \sum_{k=1}^\ell\bb E \cal M(V_k)
\end{equation}
for some fixed $\ell$, and each $V_k$ is in the form of \eqref{3.40}. Also by \eqref{diff}, it is clear that for every $V_k$, either $i$ or $j$ appears odd many times in its Green functions. In other words, we have $V_k\in \cal V_o$. In addition, the parameters of $V_k$ satisfies $\nu=2$, $\nu_1=2n-a$, $\omega+\gamma=2a-1$, $\theta_1=1$, $\theta_2=s-1$. Thus Lemma \ref{lemma 3.11} shows that
\begin{equation*}
	\begin{aligned}
		\bb E \cal M(V_k)&\prec N^{-1+2}f^{-s+1} \phi^{2a-1}  N^{-1/2} (\phi N^{-1/2})\sum_{d=0}^{2n-a}\phi^{3d}\Big(\frac{1}{(N\eta)^{1/3}}+\frac{1}{q^{1/3}}\Big)^d\bb E |\langle G_{\b x \b y}\rangle|^{2n-a-d}\\
		&\prec \phi^{2a}q^{-s+1}\sum_{d=a}^{2n}\phi^{3(d-a)}\Big(\frac{1}{(N\eta)^{1/3}}+\frac{1}{q^{1/3}}\Big)^{d-a}\bb E |\langle G_{\b x \b y}\rangle|^{2n-d}\\
		&\prec q^{-s+1}\Big(\frac{1}{(N\eta)^{1/3}}+\frac{1}{q^{1/3}}\Big)^{-a}\sum_{d=a}^{2n}\phi^{3d}\Big(\frac{1}{(N\eta)^{1/3}}+\frac{1}{q^{1/3}}\Big)^{d}\bb E |\langle G_{\b x \b y}\rangle|^{2n-d}\,,
	\end{aligned}
\end{equation*}
where in the second step we used $q\asymp f$. Since $s\geq 2$ and $a\leq s+1$, we have
\[
q^{-s+1}\Big(\frac{1}{(N\eta)^{1/3}}+\frac{1}{q^{1/3}}\Big)^{-a}\leq q^{-s+1+a/3}\leq 1\,.
\] 
As a result, 
\begin{equation} \label{3.47}
	\bb E \cal M(V_k)\prec\sum_{d=a}^{2n}\phi^{3d}\Big(\frac{1}{(N\eta)^{1/3}}+\frac{1}{q^{1/3}}\Big)^{d}\bb E |\langle G_{\b x \b y}\rangle|^{2n-d} \prec \cal E_6
\end{equation}
for all $V_k$ in \eqref{3.46}. This concludes \eqref{3.38}. Combining \eqref{3.35}, \eqref{3.37} and \eqref{3.38} yields \eqref{3.34}, which finishes the proof of \eqref{3.3}.

\section{Applications of Theorem \ref{thm1.3}} \label{sec4}
\subsection{Proof of Theorem \ref{thmbulk}}
Fix $\b v \in \bb S_\perp^{N-1}$. By Theorem \ref{thm1.2}, we see that Assumptions 1.3, 1.4 of \cite{BHY} are satisfied with $H_0=A$, $m_0(z)=m_{sc}(z)$ and $\b q=\b v$. Theorem \ref{thmbulk} then follows directly from \cite[Theorem 1.5]{BHY} and the comparison method developed in \cite{HLY15}.

\subsection{Proof of Theorem \ref{thm1.4}} 
In this section we prove Theorem \ref{thm1.4} for $k=1$; the general case follows in a similar fashion.

Recall that $\cal{G}(z)\deq (H-z)^{-1}$. From \cite{Lee21,HY22}, we have the following results.

\begin{proposition} \label{prop4.1}
	There exists a probability measure $ \rho$ supported on $[-\cal L,\cal L]$ with Stieltjes transform 
	$$
	m\equiv m(z)\deq \int_{\bb R} \frac{\rho(x)}{x-z} \dd z
	$$
	such that the followings hold true.
	\begin{enumerate}
		\item The quantity $\cal L$ is random, and it is a polynomial whose variables are the entries of $H$. It satisfies the decomposition $\cal L=L+\cal Z$, where $$L=\bb E \lambda^H_1+O(N^{-1})=2+O(q^{-2})$$
		is deterministic, and $\cal Z$ is a centered random variable satisfying
		\[
		\frac{N\cal Z}{\sqrt{2\sum_{ij}\mathbb E H_{ij}^4}}\overset{d}{\longrightarrow} \cal N(0,1)\,, \quad \mbox{and} \quad \cal Z\prec \frac{1}{\sqrt{N}q}\,.
		\]
		More properties of $\cal L$ are given in Section \ref{sec5.1} below.

		\item We have
		\[
		\ul{\G}(z)-m(z) \prec \frac{1}{N\eta}
		\]
		uniformly for all $z \in\bf D$.
		
		\item Let us denote the classical eigenvalue locations of $\varrho$ as $\gamma_1>\gamma_2>\cdots >\gamma_N$\,,
		\[
		\frac{k}{N}=\int_{\gamma_k}^{\cal L} \rho(x)\, \dd x\,,\quad 1\leq k\leq N\,.
		\]
		Then we have
		\[
		\lambda^H_k-\gamma_k\prec \frac{1}{N^{2/3}\min\{k,N-k\}^{1/3}}\,, \quad 1\leq k \leq N\,. 
		\]
		
		\item Let $\kappa(E)\equiv \kappa\deq |E^2-\cal L^2|$. We have, for all $z \in \b D$ that
		\[
		\im m(z) \asymp \begin{cases}
			\sqrt{\kappa +\eta} &\mbox{if } E \in [-\cal L,\cal L] \\
			\eta/\sqrt{\kappa+\eta} &\mbox{otherwise.} 
		\end{cases}
		\]
\item Let $\mu_1\geq\mu_2,...,\geq \mu_N$ be the eigenvalues of $GOE$. Fix $k\geq 1$. For any smooth, compactly supported $F: \bb R^{k}\to \bb R$, we have
\[
\bb E F(N^{2/3}(\lambda_1^{H}-\cal L),...,N^{2/3}(\lambda^{H}_{k}-\cal L))=\bb E F(N^{2/3}(\mu_1-2),...,N^{2/3}(\mu_{k}-2))+O(N^{-\varepsilon})
\]
for some fixed $\varepsilon>0$.
	\end{enumerate}
\end{proposition} 

Let $\cal L$ be as in Proposition \ref{prop4.1}. Fix small $\xi\in (0,\tau/100)$. Let $E=O(N^{-2/3})$,  $\eta=N^{-2/3-\xi}$, and $f\in C_c^{\infty}(\bb R)$. Our main task is to show 
\begin{equation} \label{4.1}
\bb E \bigg[F\bigg(N\int_E^{N^{-2/3+\xi}}\ul{\widetilde{G}}(\cal L+x+\ii \eta)\dd x\bigg)\bigg]=\bb E \bigg[F\bigg(N\int_E^{N^{-2/3+\xi}}\ul{\widetilde{\G}}(\cal L+x+\ii \eta)\dd x\bigg)\bigg]+O(N^{-\varepsilon})
\end{equation}
for some fixed $\varepsilon>0$. Indeed, by \eqref{4.1} and a standard argument (e.g.\cite{KY1}), it is not hard to deduce that
\[
\bb E f(N^{2/3}(\lambda_2-\cal L))=\bb E f(N^{2/3}(\lambda^H_1-\cal L))+O(N^{-\varepsilon'})
\]
for some fixed $\varepsilon'>0$. Together with Proposition \ref{prop4.1}\,(v) we conclude the proof.

Thanks to Theorem \ref{thm1.3}, the proof of \eqref{4.1} is rather simple. Using resolvent identity, we get
$
\ul{\cal{G}}-\ul{G}=\ul{{G}(A-H)\cal{G}}.
$
By Ward identity, \eqref{1.2} and Corollary \ref{cor1.5}, we have
\[
\ul{{G}(A-H)\cal{G}}=\frac{f}{N}(G\cal{G})_{\b e \b e}  \prec \frac{f}{N\eta}(\im G_{\b e \b e}\cdot \im \cal{G}_{\b e \b e})^{1/2}\prec \frac{1}{N\eta} (\im \cal{G}_{\b e \b e})^{1/2} \prec \frac{1}{N\eta} (\im \ul{\cal G})^{1/2}
\]
for all $z\in \b D$. Thus
\begin{equation} \label{111}
\ul{\cal{G}}=\ul{G}+O_\prec\Big(\frac{1}{N\eta}\Big) (\im \ul{\cal{G}})^{1/2}
\end{equation}
for all $z \in \b D$. Together with Proposition \ref{prop4.1}\,(ii), (iv), we see that
\[
\ul{\widetilde{G}}(\cal L+x+\ii \eta)=\ul{\widetilde{\G}}(\cal L+x+\ii \eta)+O(N^{-1/2+2\xi})
\]
uniformly for all $|x|\leq N^{-2/3+\xi}$, which easily implies \eqref{4.1}. This finishes the proof of Theorem \ref{thm1.4}.

 As a direct consequence of Theorem \ref{thm1.4}, we also have the gap universality and the level-repulsion estimate  near the edge.
\begin{corollary} \label{cor4.2}
	Let $\mu_1\geq\mu_2,...,\geq \mu_N$ be the eigenvalues of $GOE$. Fix $k\geq 1$. For any smooth, compactly supported $F: \bb R^{k}\to \bb R$, we have
	\[
	\bb E F(N^{2/3}(\lambda_2-\lambda_3),...,N^{2/3}(\lambda_{k+1}-\lambda_{k+2}))=\bb E F(N^{2/3}(\mu_1-\mu_2),...,N^{2/3}(\mu_{k}-\mu_{k+1}))+O(N^{-\varepsilon})
	\]
	for some fixed $\varepsilon>0$. In addition, there exists fixed $\varepsilon_0>0$, such that for any $\varepsilon\in (0,\varepsilon_0)$, there exists $\alpha>0$ such that
	\begin{equation} \label{5.3}
		\max_{i=2,3,...,k+1}	\mathbb P(\lambda_{i}-\lambda_{i+1} \leq N^{-2/3-\varepsilon})=O (N^{-\varepsilon-\alpha})\,.
	\end{equation}\,.
\end{corollary}

Finally, by Proposition \ref{prop4.1} and \eqref{111}, we have the following result, which will also be useful in Section \ref{sec5}.

\begin{corollary} \label{lem4.1} (i) We have
		\[
		\ul{G}(z)-m(z) \prec \frac{1}{N\eta}
		\]
		uniformly for all $z \in\bf D$.
		
		(ii) We have
		\[
		\lambda_k-\gamma_k\prec \frac{1}{N^{2/3}\min\{k,N-k\}^{1/3}}\,, \quad 2\leq k \leq N\,. 
		\]
\end{corollary} 
\begin{proof}
	Part (i) follows from Proposition \ref{prop4.1}\,(ii) and \eqref{111}, and part (ii) is a simple consequence of Proposition \ref{prop4.1}\,(iii) and Cauchy-interlacing theorem.
\end{proof}

\section{Proof of Theorem \ref{thm1.2}} \label{sec5}

Throughout this section, let $\tau$ be as in the beginning of Section \ref{sec1}, $\cal L$ be as in Proposition \ref{prop4.1}, and $\varepsilon_0$ be as in Corollary \ref{cor4.2}. Fix parameters
\begin{equation} \label{parameterss}
	\xi \in (0,(\tau\wedge \varepsilon_0)/100) \quad \mbox{and} \quad  \delta\in (0,\xi/100)\,, 
\end{equation}
and set 
$$
\eta_+\deq N^{-2/3-\xi}\,, \quad \eta_-\deq N^{-2/3-6\xi}\,.
$$
We shall prove Theorem \ref{thm1.2} for $k=1$; the general case follows in a similar fashion. Let $\b v, \b w \in \bb S^{N-1}_\perp$ be deterministic and fix $T>0$. The goal of this section is to show that
\begin{equation} \label{5.1}
\bb E \exp\big(\ii tN\langle \b  v,\b u_{2}\rangle \langle \b  w,\b u_{2}\rangle \big)=\bb E  \exp\big(\ii t\langle \b  v,\b z\rangle \langle \b  w,\b z\rangle\big)+O(N^{-\delta/2})	
\end{equation} 
uniformly for all $t \in[-T,T]$. Here $\b z$ denotes the standard Gaussian vector in $\bb R^n$.  

By Corollary \ref{cor1.5}, it is not hard to see that
\begin{equation} \label{5.2}
N\langle \b  v,\b u_{2}\rangle \langle \b  w,\b u_{2}\rangle =\frac{\eta_+}{\pi}\int_{\lambda_2-N^{\delta}\eta_+}^{\lambda_2+N^{\delta}\eta_+}\frac{N \langle \b  v,\b u_{2}\rangle\langle \b  w,\b u_{2}\rangle}{(\lambda_2-E)^2+\eta_+^2}\,\dd E+O_\prec(N^{-\delta})\,.
\end{equation}
Together with Corollaries \ref{cor1.5}, \ref{cor4.2}, and \ref{lem4.1}, one can readily follow the arguments of \cite[Lemmas 3.1 and 3.2]{KY1} to show the following result.

\begin{lemma} \label{lemma 5.1}
	Let $q: \bb R \to \bb R_+$ be a smooth cut-off function such that $q(x)=1$ for $|x|\leq 1/3$ and $q(x)=0$ for $x\geq 2/3$. We abbreviate
	\begin{equation*} 
\cal 	I\deq [\cal L-N^{-2/3+\delta},\cal L+N^{-2/3+\delta}] \,,\quad \mbox{and}\quad \cal J_E\deq [E-N^{\delta}\eta_+,\cal L+N^{-2/3+\delta}]\,.
	\end{equation*}
Then
	\begin{equation*}
		\begin{aligned}
	&\bb E \exp\big(\ii tN\langle \b  v,\b u_{2}\rangle \langle \b  w,\b u_{2}\rangle \big)\\
	=\,&\bb E \exp\bigg(\frac{\ii tN}{\pi}\int _{\cal I} \widetilde{G}_{\b v\b w}(E+\ii \eta_+)q\bigg[\int_{\cal J_E} \tr \widetilde{G}(x+\ii \eta_-)\dd x\bigg]\,\dd E\bigg)+O(N^{-\delta})\eqd \bb E \exp( \ii tY)+O(N^{-\delta})\,.
			\end{aligned}
	\end{equation*}
\end{lemma}

Recall that for any matrix $M\in \bb C^{N\times N}$, we abbreviate
$\widehat{M}=\re M$ and $\widetilde{M}=\im M$. In the sequel, we abbreviate 
\begin{equation} \label{g and x}
g(t)\deq \bb E \exp(\ii t Y) \quad \mbox{and} \quad X_E\deq \int_{J_E} \tr \widetilde{G}(x+\ii \eta_-)\dd x\,.
\end{equation}
We shall often omit the arguments in the Green function, unless it is not clear. We start the computation with the identity
\[
z\ul{G}G_{\b v\b w}=\ul{AG}G_{\b v \b w}-G_{\b v \b w}=\ul{G}(AG)_{\b v\b w}-\ul{G}\langle \b v, \b w \rangle\,,
\]
which implies
\begin{equation}\label{5/5}
	\widetilde{G}_{\b v\b w}-\langle \b v, \b w \rangle \ul{\widetilde{G}}=\im( \ul{AG}G_{\b v \b w}- \ul{G}(AG)_{\b v\b w})=\im( \ul{HG}G_{\b v \b w}- \ul{G}(HG)_{\b v\b w})+O_{\prec}(N^{-1})\,.
\end{equation}
Here in the second step we used \eqref{1.2} and $\b v \perp \b e$. Thus
\begin{align}
g'(t)=\,&\ii \bb E [Y \exp(\ii t Y) ]=\ii \bb E \bigg[\frac{N}{\pi}\int_{\cal I} (\widetilde{G}_{\b v\b w}-\langle \b v, \b w \rangle \ul{\widetilde{G}}+\langle\b v, \b w \rangle \ul{\widetilde{G}})q(X_E)\dd E\cdot\exp( \ii tY)\bigg]\nonumber\\
=\,&\ii \bb E \bigg[\frac{N}{\pi}\int_{\cal I} \im( \ul{HG}G_{\b v \b w}- \ul{G}(HG)_{\b v\b w})q(X_E)\dd E\cdot \exp( \ii tY)\bigg]\label{5.5}\\
&+\ii \bb E \bigg[\frac{N}{\pi}\int_{\cal I}\langle \b v, \b w \rangle \ul{\widetilde{G}}\,q(X_E)\dd E\cdot \exp( \ii tY)\bigg]+O_{\prec}(N^{-2/3})\eqd (\mbox{I''})+ (\mbox{II''})+O_{\prec}(N^{-2/3})\nonumber\,.
\end{align}
Our main task in this section is the computation of (I"). As $H$ is a real matrix, we can apply Lemma \ref{lem:cumulant_expansion} and get
\begin{equation} \label{5.6}
	\begin{aligned}
	(\mbox{I''})&=\ii \bb E \bigg[\frac{N}{\pi}\int_{\cal I} \bigg(\frac{1}{N}\sum_{ij}H_{ij}\im(G_{ji}G_{\b v \b w})- \sum_{ij}\b v_i H_{ij}\im (G_{j\b w}\ul{G})\bigg)q(X_E)\dd E\cdot \exp( \ii tY)\bigg]\\
	&=\ii \bb E \bigg[\sum_{s=1}^{\ell}\sum_{ij}\cal C_{s+1}(H_{ij})\partial_{ij}^{s}\bigg(\frac{1}{\pi}\int_{\cal I}  \im(G_{ji}G_{\b v \b w})q(X_E)\dd E\cdot \exp( \ii tY)\bigg)\bigg]\\
	&\quad - \ii \bb E \bigg[\sum_{s=1}^{\ell}\sum_{ij}\cal C_{s+1}(H_{ij})\partial_{ij}^{s}\bigg(\frac{N}{\pi}\int_{\cal I}  \im(\b v_iG_{j\b w}\ul{G})q(X_E)\dd E\cdot \exp( \ii tY)\bigg)\bigg]+O(N^{-\delta})\\
	&\eqd\sum_{s=1}^{\ell} L^{(5)}_s+ \sum_{s=1}^{\ell} L^{(6)}_s+O(N^{-\delta})\,.
	\end{aligned}
\end{equation}

\subsection{More properties of the random edge} \label{sec5.1}

Before we proceed, observe that on the RHS of \eqref{5.6}, the integral domains $\cal I$ and $\cal J_E$ depend on the random variable $\cal L$, which is random, and sensitive to the differentiation $\partial/\partial H_{ij}$. Thus it is necessary to pause and cover more structural properties of $\cal L$. 

One good way to fully describe $\cal L$ is through weighted trees. To this end, we denote a weighted trees by $T\equiv T(\widehat{V}, \widehat{E})$, where $\widehat{V}$ are the set of vertices, and $\widehat{E}$ are the set of edges. For every $e\in \widehat{E}$, its weight is denoted by $w(e)$. We are interested in the set of weighted tree such that the weight of each edge is a positive even integer, and the total weight is not bigger than $\tau^{-1}$, i.e.
\[
\cal T\deq \Big\{T(\widehat{V}, \widehat{E}): w(e)\in 2\bb N_+ \mbox{ for all } e \in \widehat{E},\,\sum\nolimits_{e\in \widehat{E}}w(e)\leq \tau^{-1}\Big\}\,.
\]
Let $\widehat{n}\deq|\widehat{V}|$, and we use the formal indices $v_1,...,v_{\widehat{n}}$ to denote the vertices of $T$.
For any $e =v_iv_j$, we assign the formal variable $\cal H_e\deq H^{w(e)}_{v_iv_j}-\mathbb EH^{w(e)}_{v_iv_j}$. For each $T\in \cal T$, we set
\[
\cal H_T\deq \prod_{e \in \widehat{E}} \cal H_e\,.
\]

We assign to each $\cal H_T$ with its \textit{evaluation}, which is a random variable depending on the $\widehat{n}$-tuple $(v_1,...,v_{\widehat{n}})\in \{1,2,...,N\}^{\widehat{n}}$. It is obtained by replacing, in the formal monomial $\cal H_T$, the formal indices $v_1,\dots,v_{\widehat{n}}$ with the integers $v_1,\dots,v_{\widehat{n}}$ and the formal variables $H^{w(e)}_{v_iv_j}$ with elements of the matrix $H$ in Definition \ref{def_sparse}. We define
	\begin{equation*} 
	\cal M_* (\cal H_{T}) \deq \frac{1}{N} \psum_{v_1,\dots,v_{\widehat{n}}=1}^N \cal H_T\,,
	\end{equation*}
where $\sum^*$ is shorthand for distinct sum.

From the construction of $\cal L$ in \cite{Lee21,HY22} and Proposition \ref{prop4.1} (i), we have the following result.

\begin{lemma} \label{lemma 5.3}
	The random variable $\cal L$ satisfies the decomposition $\cal L=L+\cal Z$, where $L$ is deterministic, and $\cal Z$ is random and centered.  The random contribution $\cal Z$ is a linear combination of ${\cal M_*}(\cal H_T)$, $T \in \cal T$, with bounded coefficients.
\end{lemma}

In the sequel, for fixed $n\geq 1$, we use the abbreviation
\[
\D^n_{\b i\b j}\deq \partial_{i_1j_1}\cdots\partial_{i_nj_n}\,,
\]
where $\b i=(i_1,...,i_n), \b j=(j_1,...j_n)\in \{1,2,...,N\}^n$. We have the following estimates, which can be proved directly using Lemma \ref{lem:cumulant_expansion}.

\begin{lemma} \label{lemma 5.4}
	Let $T\in \cal T$. For any fixed $n\geq 1$, and $\b i,\b j\in \{1,2,...,N\}^n$, we have
	\[
	\D_{\b i\b j}^n\,\cal M_* (\cal H_{T}) \prec N^{-1}\,.
	\]	
	(ii) Fix $n\geq 0$ and let $\b v \in \bb S^{N-1}$ be deterministic. Suppose the index $i$ does not appear in $\b i$ or $\b j$, then 
	\[
	\sum_i\b v_i\partial_{ij}\D_{\b i\b j}^n\,\cal M_* (\cal H_{T}) \prec N^{-1}\,.
	\]
\end{lemma}
\begin{proof}
	Part (i) can be proved directly using Lemma \ref{lem:cumulant_expansion}. For Part (ii), the naive estimate using part (i) is $O_{\prec}(N^{-1/2})$. To see where the additional factor $N^{-1/2}$ comes from, note that the weight of every edge in $T$ is even, and the sum in $\cal M_*(\cal H_T)$ is distinct. As a result, any non-zero term in  $\partial_{ij}\D_{\b i\b j}^n\,\cal M_* (\cal H_{T})$ can be written in the form
	\[
	H_{ij}\cal X_{ij}\,,
	\]
	where $\D_{\b i\b j}^k \cal X_{ij}\prec N^{-1}$ for all fixed $k\geq 0$ (c.f.\,Example \ref{eg5.5} (ii) below). By Lemma \ref{lem:cumulant_expansion}, we can easily show that $\sum_i\b v_i H_{ij}\cal X_{ij}\prec N^{-1}$. This concludes the proof.
\end{proof}

\begin{example} \label{eg5.5}
(i) Let $T_1(\widehat{V},\widehat{E})\in \cal T$ with $\widehat{V}=\{v_1,v_2\}$, $\widehat{E}=\{v_1v_2\}$, and $w(v_1v_2)=2$. Then
\[
\cal M_*(\cal H_{T_1})=\frac{1}{N}\psum_{v_1v_2}\Big(H^2_{v_1v_2}-\frac{1}{N}\Big)\,,
\]
which is the leading contribution of $\cal Z$.

(ii) Let $T_2\in (\widehat{V},\widehat{E})\in \cal T$ with $\widehat{V}=\{v_1,v_2,v_3,v_4\}$, $\widehat{E}=\{v_1v_2,v_2v_3,v_3v_4\}$ and $w(v_1v_2)=w(v_1v_2)=w(v_1v_2)=2$. Then
\[
\cal M_*(\cal H_{T_2})=\frac{1}{N}\psum_{v_1v_2v_3v_4}\Big(H^2_{v_1v_2}-\frac{1}{N}\Big)\Big(H^2_{v_2v_3}-\frac{1}{N}\Big)\Big(H^2_{v_3v_4}-\frac{1}{N}\Big)\,,
\]
which is a third order term of $\cal Z$. 
\end{example}

\subsection{The leading terms}
Here we compute the leading contribution in \eqref{5.5}, which is
\[
L^{(5)}_1+L^{(6)}_1+\mbox{(II'')}\,.
\] 
Abbreviate 
\begin{equation}
	E_{\pm}\deq \cal L\pm N^{-2/3+\delta}\,, \quad \mbox{and} \quad z_\pm\deq E_\pm+\ii \eta_+\,.
\end{equation}
As $\cal C_2(H_{ij})=N^{-1}(1+O(\delta_{ij}))$, we get
\begin{equation} \label{5.7}
	\begin{aligned}
		L^{(6)}_1&=-\ii \bb E \bigg[\frac{1}{\pi }\sum_{ij} (1+\delta_{ij}) \partial_{ij}\bigg(\int_{\cal I} \im(\b v_i G_{j\b w}\ul{G})q(X_E)\dd E\cdot \exp( \ii tY)\bigg)\bigg]+O(N^{-\delta})\\
		&=-\ii \bb E \bigg[\frac{1}{\pi }\sum_{ij}\b v_i (\partial_{ij}\cal L)  \im( G_{j\b w}(z_+) \ul{G}(z_+)q(X_{E_+})-G_{j\b w}(z_-) \ul{G}(z_-)q(X_{E_-})) \exp( \ii tY)\bigg]\\
		&\quad -\ii \bb E \bigg[\frac{1}{\pi }\sum_{ij}(1+\delta_{ij})  \int_{\cal I} \im(\b v_i \partial_{ij}G_{j\b w}\cdot \ul{G})q(X_E)\dd E\cdot \exp( \ii tY)\bigg]\\
		&\quad - \ii \bb E \bigg[\frac{1}{\pi }\sum_{ij}(1+\delta_{ij})  \int_{\cal I} \im(\b v_i G_{j\b w}\cdot \partial_{ij}\ul{G})q(X_E)\dd E\cdot \exp( \ii tY)\bigg]\\
		&\quad -\ii \bb E \bigg[\frac{1}{\pi }\sum_{ij}(1+\delta_{ij}) \int_{\cal I} \im(\b v_i G_{j\b w}\ul{G}) \partial_{ij}q(X_E)\dd E\cdot \exp( \ii tY)\bigg]\\
		&\quad -\ii \bb E \bigg[\frac{1}{\pi }\sum_{ij}(1+\delta_{ij})  \int_{\cal I} \im(\b v_i G_{j\b w}\ul{G}) q(X_E)\dd E\cdot \partial_{ij}\exp( \ii tY)\bigg]+O(N^{-\delta})\\
		&\eqd L^{(6)}_{1,0}+\cdots+L^{(6)}_{1,4}+O(N^{-\delta})\,.
	\end{aligned}
\end{equation}
By Lemma \ref{lemma 5.4} (ii), we know that
$
\sum_{i} \b v_i (\partial_{ij}\cal M_*(\cal H_T)) \prec N^{-1}
$
for all $T \in \cal T$. Together with Lemma \ref{lemma 5.3} we get
\begin{equation} \label{5/9}
	\sum_{i} \b v_i (\partial_{ij}\cal L) \prec N^{-1}\,.
\end{equation}
In addition, Corollary \ref{cor1.5}, Proposition \ref{prop4.1}, and Corollary \ref{lem4.1} yield
\begin{equation} \label{5/10}
\im( G_{j\b w}(z)\ul{G}(z)) \prec \im G_{j\b w}(z) + \im \ul{G}(z) \prec \im\ul{G}(z)\prec N^{-1/3+\xi}
\end{equation}
for all $z=E+\ii \eta_+$ with $E \in \cal I$. Thus we get
\begin{equation} \label{5/11}
	L_{1,0}^{(6)} \prec N\cdot N^{-1}\cdot N^{-1/3+\xi}\prec N^{-1/3+\xi}\,.
\end{equation}
By \eqref{diff}, we have
\begin{equation*}
\begin{aligned}
L^{(6)}_{1,1}&=\ii \bb E \bigg[\frac{1}{\pi }\sum_{ij} \int_{\cal I} \im((\b v_iG_{i \b w}G_{jj}+\b v_i G_{ij}G_{j \b w})\cdot \ul{G})q(X_E)\dd E\cdot \exp( \ii tY)\bigg]\\
&=\ii \bb E \bigg[\frac{1}{\pi} \int_{\cal I} \im((NG_{\b v\b w}\ul{G}+(G^2)_{\b v\b w}) \ul{G})q(X_E)\dd E\cdot \exp( \ii tY)\bigg]\,.
\end{aligned}
\end{equation*}
By spectral decomposition, Corollary \ref{cor1.5}, Proposition \ref{prop4.1}, and Corollary \ref{lem4.1},
\begin{equation} \label{isotropic delocalization application}
	\begin{aligned}
		(G^2)_{\b v\b w}&=\sum_{a=1}^N \frac{\langle \b v, \b u_a\rangle \langle \b u_a ,\b w \rangle}{(\lambda_a-E-\ii \eta_+)^2} \prec\sum_{a=1}^N \frac{|\langle \b v, \b u_a\rangle \langle \b u_a ,\b w \rangle|}{(\lambda_a-E)^2+\eta_+^2} \\
		&\prec\frac{1}{N}\sum_{a=1}^N \frac{1}{(\lambda_a-E)^2+\eta_+^2}=\eta_+^{-1}\im 
		\ul{G} \prec \eta_+^{-1}( |\im \ul{G}-\im m|+\im m)\prec N^{1/3+2\xi}\,,
	\end{aligned}
\end{equation} 
and thus 
\begin{equation}\label{5.9}
	L^{(6)}_{1,1}=\ii \bb E \bigg[\frac{N}{\pi} \int_{\cal I} \im(\ul{G}^2 G_{\b v \b w})q(X_E)\dd E\cdot \exp( \ii tY)\bigg]+O_{\prec}(N^{-\delta})\,.
\end{equation}
Similarly, $(G^3)_{\b v\b w}\prec N^{1+3\xi},$ which implies
\begin{equation} \label{5.10}
	L^{(6)}_{1,2}=\ii \bb E \bigg[\frac{2}{\pi }  \int_{\cal I} \im(G^3)_{\b v \b w}q(X_E)\dd E\cdot \exp( \ii tY)\bigg]\prec N^{-\delta}\,.
\end{equation}
For $L_{1,3}^{(6)}$, recall the definition of $X_E$ in \eqref{g and x}. As the interval $\cal J_E$ contain $\cal L$, it will also be affected by the differentiation $\partial_{ij}$. However, similar to \eqref{5/9} and \eqref{5/11}, one can easily apply Lemmas \ref{lemma 5.3} and \ref{lemma 5.4} to show that this contribution is negligible. More precisely, we have
\begin{equation} \label{5/16}
	\begin{aligned}
	L^{(6)}_{1,3}=&-\ii \bb E \bigg[\frac{1}{\pi }\sum_{ij}(1+\delta_{ij}) \int_{\cal I} \im(\b v_i G_{j\b w}\ul{G}) q'(X_E) \bigg(\int_{\cal J_E} \partial_{ij}\tr \widetilde{G}^{(x)}\,\dd x\bigg)\dd E\cdot \exp( \ii tY)\bigg]\\
	&\,-\ii \bb E \bigg[\frac{1}{\pi }\sum_{ij}(1+\delta_{ij}) \int_{\cal I} \im(\b v_i G_{j\b w}\ul{G}) q'(X_E)  (\partial_{ij}\cal L)\tr \widetilde{G}(E_++\ii \eta_-)\dd E\cdot \exp( \ii tY)\bigg]\\
	=&\,\ii \bb E \bigg[\frac{2}{\pi }\sum_j \int_{\cal I} \im( G_{j\b w}\ul{G}) q'(X_E) \bigg(\int_{\cal J_E} \im ((G^{(x)})^2)_{j \b v}\dd x\bigg)\dd E\cdot \exp( \ii tY)\bigg]+O_{\prec}(N^{-\delta})\,.
	\end{aligned}
\end{equation}
Here we abbreviate $G^{(x)}\deq G(x+\ii \eta_-)$. Similar to \eqref{isotropic delocalization application}, we see that
\[
\int_{J_E} \im ((G^{(x)})^2)_{j \b v}\dd x \prec  N^{-2/3+\delta}\cdot N^{1/3+12\xi}\prec N^{-1/3+13\xi}\,.
\]
Together with \eqref{5/10} and \eqref{5/16}, we get
\begin{equation} \label{5.11}
		L^{(6)}_{1,3}
		\prec N\cdot |\cal I|\cdot  N^{-1/3+\xi}\cdot N^{-1/3+13\xi} \prec N^{-\delta}\,.
\end{equation}

The leading contribution of $L^{(6)}_1$ is contained in $L^{(6)}_{1,4}$. Similar to \eqref{5.9} and \eqref{5.11}, $L^{(6)}_{1,4}$ can be computed as
	\begin{align} \label{5.12}
	-&\ii \bb E \bigg[\frac{1}{\pi }\sum_{ij} (1+\delta_{ij}) \int_{\cal I} \im(\b v_i G_{j\b w}\ul{G}) q(X_E)\dd E \,\frac{\ii t N}{\pi}\int_\cal I (\partial_{ij}\widetilde{G}_{\b v\b w}) q(X_E)\dd E\cdot \exp( \ii tY)\bigg]\nonumber\\
	-&\ii \bb E \bigg[\frac{1}{\pi }\sum_{ij} (1+\delta_{ij}) \int_{\cal I} \im(\b v_i G_{j\b w}\ul{G}) q(X_E)\dd E\,  \frac{\ii t N}{\pi}\int_\cal I \widetilde{G}_{\b v\b w} \partial_{ij}q(X_E)\dd E \cdot\exp( \ii tY)\bigg]\\
	=-&\frac{tN}{\pi^2}\bb E \bigg[\sum_{j} \int_{\cal I^2} \im( G_{j\b w}\ul{G}) \im (G'_{\b v\b v}G'_{j \b w}+G'_{\b vj}G'_{\b v\b w}) q(X_E)q(X_{E'})\dd E\dd E'\cdot\exp( \ii tY)\bigg]+O_{\prec}(N^{-\delta})\,,\nonumber
	\end{align}
where in the second step we used \eqref{isotropic delocalization application} and the abbreviation $G'=G(E'+\ii \eta_+)$. Note that
\begin{equation} \label{5.13}
	\begin{aligned}
	&\,\sum_{j} \im( G_{j\b w}\ul{G})\im (G'_{\b v\b v}G'_{j \b w})=\sum_{j} \widetilde{G}_{j\b w}\widehat{\ul{G}}\im (G'_{\b v\b v}G'_{j \b w})+\sum_{j} \widehat{G}_{j\b w}\widetilde{\ul{G}}\im (G'_{\b v\b v}G'_{j \b w})\\
	=&\,\sum_{j} \widetilde{G}_{j\b w}\widehat{\ul{G}}\im (G'_{\b v\b v}G'_{j \b w})+O_{\prec}(N^{3\xi})=\sum_{j} \widetilde{G}_{j\b w}\widehat{\ul{G}}\widehat{G}'_{\b v\b v}\widetilde{G}'_{j \b w}+O_{\prec}(N^{3\xi})\\
	=&\sum_{j} \widetilde{G}_{j\b w}\widetilde{G}'_{j \b w}+O_{\prec}(N^{1/3-\tau/2})\,.
	\end{aligned}
\end{equation}
Here in the second step we used 
\begin{align*}
&\sum_{j} \widehat{G}_{j\b w}\widetilde{\ul{G}}\im (G'_{\b v\b v}G'_{j \b w})\prec \sum_j |G_{j\b w} G'_{j \b w}| \cdot |\widetilde{\ul{G}}| \prec \frac{(\im G_{\b w\b w}\cdot \im G'_{\b w\b w})^{1/2}}{\eta} \cdot \frac{1}{N\eta_+} \\
\prec&\, \frac{(\widetilde{\ul{G}}\cdot \widetilde{\ul{G}}')^{1/2}}{\eta} \cdot \frac{1}{N\eta_+}  \prec N^{3\xi}\,,
\end{align*}
and in the third step we used 
\[
\sum_{j} \widetilde{G}_{j\b w}\widehat{\ul{G}}\widetilde{G}'_{\b v\b v}\widehat{G}'_{j \b w} \prec \sum_j |G_{j\b w} G'_{j \b w}| \cdot |\widetilde{G}'_{\b v\b v}| \prec \sum_j |G_{j\b w} G'_{j \b w}| \cdot |\widetilde{\ul{G}}'|\prec N^{3\xi}\,.
\]
In both cases we applied Corollary \ref{cor1.5}, Proposition \ref{prop4.1}, and Corollary \ref{lem4.1}. In the last step of \eqref{5.13} we used Theorem \ref{thm1.3}. Similarly,
\begin{equation} \label{5.14}
\sum_{j} \im( G_{j\b w}\ul{G})\im (G'_{\b vj}G'_{\b v \b w})=\langle \b v, \b w \rangle \sum_{j} \widetilde{G}_{j\b w}\widetilde{G}'_{\b vj}+O_{\prec}(N^{1/3-\tau/2})\,.
\end{equation}
Combining \eqref{5.12} -- \eqref{5.14}, we get
\begin{equation} \label{5.15}
L^{(6)}_{1,4}=	-\frac{tN}{\pi^2}\bb E \bigg[\sum_{j} \int_{\cal I^2} \widetilde{G}_{j\b  w} (\widetilde{G}'_{j\b w}+\langle \b v,\b w \rangle \widetilde{G}'_{\b v j}) q(X_E)q(X_{E'})\dd E\dd E'\exp( \ii tY)\bigg]+O_{\prec}(N^{-\delta})\,.
\end{equation}
Inserting \eqref{5/10}, \eqref{5.9} -- \eqref{5.11} and \eqref{5.15} into \eqref{5.7}, we get
\begin{equation} \label{5.16}
	\begin{aligned}
	L^{(6)}_{1}=\,&\ii \bb E \bigg[\frac{N}{\pi} \int_{\cal I} \im(\ul{G}^2 G_{\b v \b w})q(X_E)\dd E\cdot \exp( \ii tY)\bigg]\\
	&-\frac{tN}{\pi^2}\bb E \bigg[\sum_{j} \int_{\cal I^2} \widetilde{G}_{j\b  w} (\widetilde{G}'_{j\b w}+\langle \b v,\b w \rangle \widetilde{G}'_{\b v j}) q(X_E)q(X_{E'})\dd E\dd E'\exp( \ii tY)\bigg]+O_{\prec}(N^{-\delta})\,.
		\end{aligned}
\end{equation}
Similarly, we can also show that
\begin{equation} \label{5.17}
	\begin{aligned}
		L^{(5)}_1&=\ii \bb E \bigg[\frac{1}{N\pi }\sum_{ij} (1+\delta_{ij}) \partial_{ij}\bigg(\int_{\cal I} \im(G_{ji} G_{\b v\b w})q(X_E)\dd E\cdot \exp( \ii tY)\bigg)\bigg]+O_{\prec}(N^{-\delta})\\
		&=-\ii \bb E \bigg[\frac{N}{\pi} \int_{\cal I} \im(\ul{G}^2 G_{\b v \b w})q(X_E)\dd E\cdot \exp( \ii tY)\bigg]+O_{\prec}(N^{-\delta})\,.
	\end{aligned}
\end{equation}
Note that there is a cancellation between \eqref{5.16} and \eqref{5.17}, and this yields
\begin{equation} \label{5.18}
\hspace{-0.45cm}	L^{(6)}_{1}+L^{(5)}_1=	-\frac{tN}{\pi^2}\bb E \bigg[\sum_{j} \int_{\cal I^2} \widetilde{G}_{j\b  w} (\widetilde{G}'_{j\b w}+\langle \b v,\b w \rangle \widetilde{G}'_{\b v j}) q(X_E)q(X_{E'})\dd E\dd E'\exp( \ii tY)\bigg]+O_{\prec}(N^{-\delta})\,.
\end{equation}

As $\eta_+\ll N^{-2/3}$, it is $\emph{not possible}$ to compute the RHS of \eqref{5.18} through local laws. Instead, observe that
\[
\sum_j \widetilde{G}_{j\b  w}\widetilde{G}'_{j\b w}=\sum_j \sum_{a} \frac{\b u_a(j)\langle \b u_a, \b w\rangle\eta_+}{(\lambda_a-E)^2+\eta_+^2}\sum_{a'}\frac{\b u_{a'}(j)\langle \b u_{a'}, \b w\rangle\eta_+}{(\lambda_{a'}-E')^2+\eta_+^2}= \sum_a \frac{\langle \b u_a , \b w \rangle^2\eta_+^2}{((\lambda_{a}-E)^2+\eta_+^2)((\lambda_{a}-E')^2+\eta_+^2)}\,.
\]  
Similar to Lemma \ref{lemma 5.1}, one can use  Corollaries \ref{cor1.5}, \ref{cor4.2}, and \ref{lem4.1} to show that
\begin{equation} \label{5.19}
\hspace{-0.2cm}\begin{aligned}
&\,\frac{tN}{\pi^2}\bb E \bigg[\sum_{j} \int_{\cal I^2} \widetilde{G}_{j\b  w} \widetilde{G}'_{j\b w} q(X_E)q(X_{E'})\dd E\dd E'\exp( \ii tY)\bigg]\\
	=&\,\frac{tN}{\pi^2}\bb E \bigg[\int_{\cal I^2}  \sum_a \frac{\langle \b u_a , \b w \rangle^2\eta_+^2}{((\lambda_{a}-E)^2+\eta_+^2)((\lambda_{a}-E')^2+\eta_+^2)} q(X_E)q(X_{E'}) \,\dd E \,\dd E'\exp( \ii tY)\bigg]\\
	=&\,\frac{tN}{\pi^2}\bb E \bigg[\int_{\cal I^2}   \frac{\langle \b u_2 , \b w \rangle^2\eta_+^2}{((\lambda_{2}-E)^2+\eta_+^2)((\lambda_{2}-E')^2+\eta_+^2)} q(X_E)q(X_{E'}) \,\dd E \,\dd E'\exp( \ii tY)\bigg]+O_{\prec}(N^{-\delta})	\\
	=&\,\frac{tN}{\pi}\bb E \bigg[\int_{\cal I}   \frac{\langle \b u_2 , \b w \rangle^2\eta_+}{(\lambda_{2}-E)^2+\eta_+^2} q(X_E) \,\dd E \exp( \ii tY)\bigg]+O_{\prec}(N^{-\delta})\\
	=&\,\frac{tN}{\pi}\bb E \bigg[\int_{\cal I}  \sum_a \frac{\langle \b u_a , \b w \rangle^2\eta_+}{(\lambda_{a}-E)^2+\eta_+^2} q(X_E) \,\dd E \exp( \ii tY)\bigg]+O_{\prec}(N^{-\delta})\\
	\eqd&\, t \bb E (Y^{(1)}\exp(\ii t Y))	+O_{\prec}(N^{-\delta})\,,
\end{aligned}
\end{equation}
where 
$
Y^{(1)}\deq \bb E \Big[\frac{N}{\pi}\int_{\cal I} \widetilde{G}_{\b w\b w}q(X_E)\,\dd E\Big]. 
$ Here in the third step we used 
\begin{align*}
&\,\frac{1}{\pi^2}\int_{\cal I^2}   \frac{\eta_+^2}{((\lambda_{2}-E)^2+\eta_+^2)((\lambda_{2}-E')^2+\eta_+^2)} q(X_E)q(X_E') \,\dd E \,\dd E'\\
=&\,\bigg(\frac{1}{\pi}\int_{\cal I} \frac{\eta_+}{(\lambda_{2}-E')^2+\eta_+^2} q(X_E) \dd E\bigg)^2=\frac{1}{\pi}\int_{\cal I} \frac{\eta_+}{(\lambda_{2}-E')^2+\eta_+^2} q(X_E) \dd E+O_{\prec}(N^{-\delta})
\end{align*}
and Corollary \ref{cor1.5}. Analogously, 
\begin{equation*} 
	\frac{tN}{\pi^2}\bb E \bigg[\sum_{j} \int_{\cal I^2} \widetilde{G}_{j\b  w} \langle \b v ,\b w \rangle\widetilde{G}'_{\b v j} q(X_E)q(X_{E'})\dd E\dd E'\exp( \ii tY)\bigg]=t \langle \b v,\b w \rangle\bb E (Y\exp(\ii t Y))+O_{\prec}(N^{-\delta})\,.
\end{equation*}
Inserting the above and \eqref{5.19} into \eqref{5.18} yields 
\begin{equation*}
	L^{(5)}_1+L^{(6)}_1=-t \bb E (Y^{(1)}\exp(\ii t Y))-t \langle \b v,\b w \rangle\bb E (Y\exp(\ii t Y))+O_{\prec}(N^{-\delta})\,.
\end{equation*}
Similar to \eqref{5.19}, we can also show that
\[
\mbox{(II'')}=\ii\bb E \Big[\frac{N}{\pi}\int_{\cal I} \langle \b v, \b w \rangle \ul{\widetilde{G}}\,q(X_E)\dd E\cdot \exp(\ii tY)\Big]=\ii \langle \b v ,\b w\rangle \bb E \exp(\ii t Y)+O_{\prec}(N^{-\delta})=\ii \langle \b v ,\b w\rangle g(t)+O_{\prec}(N^{-\delta})\,.
\]
As a result
\begin{equation}\label{5.21}
	L^{(5)}_1+L^{(6)}_1+\mbox{(II'')}=-t \bb E (Y^{(1)}\exp(\ii t Y))-t \langle \b v,\b w \rangle\bb E (Y\exp(\ii t Y))+\ii \langle \b v ,\b w\rangle g(t)+O_{\prec}(N^{-\delta})\,.
\end{equation}

\subsection{The error terms} \label{sec5.3}
What remains to be done, is estimating the error terms on RHS of \eqref{5.6}, i.e.\,proving that
\begin{equation} \label{5.28}
	\sum_{s=2}^{\ell} L^{(5)}_s+ \sum_{s=2}^{\ell} L^{(6)}_s \prec N^{-\delta}\,.
\end{equation}
The steps are similar to the error estimates in the proof of the isotropic local law we saw in Section \ref{sec 3}. The two main differences are this time we have integrals, and we need to consider the $H$-dependence of $\cal L$.

Let $\b v,\b w$ be as in \eqref{5.1}. Similar to Definition \ref{def3.1}, let $\{i_1,i_2,...\}$ be an infinite set of formal indices. To fixed $\nu,\nu_3,\nu_4,u,r,\zeta,b,c\in \bb N$, $\theta_1,\theta_2 \in \bb R$, $\sigma_1,...,\sigma_r\in \bb N_+$, $w_{1},y_{1},...,w_{u},y_{u}, w'_{1},y'_{1},...,w'_{r},y'_{r}\in \{i_1,...,i_\nu,\b v,\b w\}$, $i\in  \{i_1,...,i_\nu\}$, and a family $(a_{i_1,...,i_\nu})_{1\leq i_1,...,i_\nu \leq N}$ of uniformly bounded complex numbers, we assign a formal monomial
\begin{equation}\label{3.151}
		\begin{aligned}
		Q=&\,a_{i_1,...,i_\nu}\b v_iN^{-\theta_1}f^{-\theta_2} \int_{\cal I^{\nu_3}}\int_{\cal J_{E_1}}\cdots \int_{\cal J_{E_{\nu_4}}} \widetilde{G}_{w_1y_1}\cdots \widetilde{G}_{w_uy_u}\widehat{G}_{w'_1y'_1}\cdots\widehat{G}_{w'_ry'_r}\cdot \D_{\b i^{(1)}\b j^{(1)}}^{n_1}\,\cal M_* (\cal H_{T_1}) \\
		&\cdots  \D_{\b i^{(\zeta)}\b j^{(\zeta)}}^{n_\zeta}\,\cal M_* (\cal H_{T_\zeta})  \cdot q^{(m_1)}(X_{E'_1})\cdots q^{(m_c)}(X_{E'_{c}})
		\dd x_1\cdots \dd x_{\nu_4}\dd E_1\cdots \dd E_{\nu_3}\exp(\ii t Y)\,.
		\end{aligned}
	\end{equation}
Here $c,m_1,...,m_c\in \bb N$ are fixed. We denote $\nu(Q) = \nu$, $\nu_3(Q) = \nu_3$, $\nu_4(Q) = \nu_4$, $u(Q) = u$, $\zeta(Q)=\zeta$, $\theta_1(Q) = \theta_1$ and $\theta_2(Q)=\theta_2$. We use $\nu_2(Q)$ to denote the number of indices satisfying the following conditions.
	\begin{enumerate}
		\item This index is not $i$.
		
		\item This index appear exactly twice in the Green functions.
		
		\item This index appear at least once, as an off-diagonal index  in some real part of the Green functions (i.e.\,$\widehat{G}_{w'_1y'_1},\dots,\widehat{G}_{w'_ry'_r}$).
	\end{enumerate}  We denote by $\cal Q$ the set of formal monomials $Q$ of the form \eqref{3.151}.
	
	Let $O(Q)\subset \{i_1,...,i_\nu\}$ be the set of indices that appear odd many times in the Green functions of $Q$.  We denote by $\cal Q_o\subset\cal Q$ the set of formal monomials such that $O(Q)\neq \emptyset$ for all $Q\in \cal Q_o$. Let ${\bf E}\deq [\cal L-2N^{-2/3+\delta},\cal L+N^{-2/3+\delta}]$. Define the random spectral domain
	\[
	{\bf S}\equiv {\bf S}_{\delta,\xi}\deq \{z=E+\ii \eta: E\in {\bf E}, \eta\in [\eta_-,\eta_+]\}\,.
	\]
	We assign to each monomial $Q \in  \cal Q$ with its \emph{evaluation}
	\begin{equation*}
		Q_{i_1,\dots,i_{\nu}} \equiv Q_{i_1,\dots,i_{\nu}}({\bf S, E})\,,
	\end{equation*}
	which is a random variable depending on an $\nu$-tuple $(i_1,\dots,i_{\nu})\in \{1,2,\dots,N\}^{\nu}$. It is obtained by replacing, in the formal monomial $Q$, the formal indices $i_1,\dots,i_{\nu_1}$ with the integers $i_1,\dots,i_{\nu_1}$, the formal variables $G_{wy}$ with elements of the Green's function \eqref{2.1} with parameters $z \in \b S$, and formal variables $q(X_E)$ with random variables $q(X_E)$ defined in \eqref{g and x} with parameters $E \in \bf E$. Here the parameters may be different for each Green function and $q(X_E)$, and they may  or may not depend on the integration variables  $x_1,...,x_{\nu_4}$ and $E_1,...,E_{\nu_3}$. We define
	\begin{equation*} 
		\cal M (Q) \deq \sum_{i_1,\dots,i_{\nu}}  Q_{i_1,\dots,i_{\nu}}\,.
	\end{equation*}

\begin{lemma} \label{lemma 5.5}
	Let $Q\in \cal Q$. We have
	\begin{equation} \label{5.65}
		\bb E \cal M(Q) \prec N^{-\theta_1+\nu}f^{-\theta_2} \cdot (N^{-2/3+\delta})^{\nu_3+\nu_4} \cdot (N^{-1/3+6\xi})^{u} \cdot (N^{-1/3+6\xi})^{\nu_2/2} \cdot N^{-\zeta}\cdot N^{-1/2}\eqd \cal E_8(Q) \,.
	\end{equation}
\end{lemma}

\begin{proof}
	The factor $N^{-\theta_1+\nu}f^{-\theta_2}$ on RHS of \eqref{5.65} is self-explanatory. Let us explain the rest contributions. The factor $(N^{-2/3+\delta})^{\nu_3+\nu_4}$ comes from the integrals w.r.t.\,$x_1,...,x_{\nu_3}$ and $E_1,...,E_{\nu_4}$. The factor $(N^{-1/3+6\xi})^{u}$ comes from Corollary \ref{cor1.5}, Proposition \ref{prop4.1} (iv), and Corollary \ref{lem4.1}. The factor $(N^{-1/3+6\xi})^{\nu_2/2}$ comes from Lemma \ref{ward}, Corollary \ref{cor1.5}, Proposition \ref{prop4.1} (iv), and Corollary \ref{lem4.1}. The factor $N^{-\zeta}$ comes from Lemma \ref{lemma 5.4}. Finally, the factor $N^{-1/2}$ comes from $\sum_i |\b v_i|\leq N^{1/2}$. This finishes the proof.
\end{proof}

We have the following improvement of Lemma \ref{lemma 5.5} for $Q \in \cal Q_o$, which we delay the proof to Section \ref{section 5.3.1}.

\begin{lemma}  \label{lemma 5.6}
	Let $Q\in \cal Q_o$. We have
	\[
\bb E \cal M(Q) \prec	N^{-\theta_1+\nu}f^{-\theta_2} \cdot (N^{-2/3+\delta})^{\nu_3+\nu_4} \cdot (N^{-1/3+6\xi})^{u}  \cdot N^{-\zeta}\cdot N^{-1+13\xi}\eqd \cal E_9(Q)\,.
	\]
\end{lemma}

Given Lemma \ref{lemma 5.6}, we can now prove \eqref{5.28}. Let us look at the estimate of $L_s^{(6)}$ closely. By \eqref{diff}, we have
\begin{equation} \label{Q_k}
L_s^{(6)}=\sum_{k=1}^\ell\bb E\cal M(Q_k)
\end{equation}
for some fixed $\ell$, and each $Q_k$ is in the form of \eqref{3.151}. These $\cal M(Q_k)$ are obtained through applying $\partial^s_{ij}$ on 
\begin{align*}
&\,\ii\,\cal C_{s+1}(H_{ij})\frac{N}{\pi}\int_{\cal I}  \im(\b v_iG_{j\b w}\ul{G})q(X_E)\dd E\cdot \exp( \ii tY)\\
=&\,\ii\,\cal C_{s+1}(H_{ij})\frac{N}{\pi}\int_{\cal I}  \b v_i(\widetilde{G}_{j\b w}\widehat{\ul{G}}+\widehat{G}_{j\b w}\widetilde{\ul{G}})q(X_E)\dd E\cdot \exp( \ii tY)\eqd Q_*+Q_{**}
\end{align*}
and then apply $\sum_{ij}$. Clearly, $Q_*\in \cal Q_o$, $\nu(Q_*) = 2$, $\nu_3(Q_*) = 1$, $\nu_4(Q_*) = 0$, $u(Q_*) = 1$, $\zeta(Q_*)=0$, $\theta_1(Q_*) = 1-1=0$ and $\theta_2(Q_*)=s-1$. As a result, $\cal E_9(Q_*)=f^{1-s}N^{\delta+19\xi}$. The same thing can also be said for $Q_{**}$. To get from $Q_*$ or $Q_{**}$ to $Q_k$, we need the following result, whose proof follows directly from our construction.

\begin{lemma} \label{lemma 5.7}
	Let $Q\equiv Q_{i_1,...,i_\nu}\in \cal Q$ be in the form of \eqref{3.151}. Suppose $j,j' \in 
	\{i_1,...,i_\nu\}$, and let $\cal M(Q')$ be a term in $\cal M(\partial_{jj'} Q)$. The parameters $\nu,\nu_3,\nu_4,u,\zeta,\theta_1,\theta_2$ of $Q$ and $Q'$ are related basing on the following cases.
\begin{enumerate}
	\item If $Q'$ is generated through applying $\partial_{jj'}$ on 
	$$
	\widetilde{G}_{w_1y_1}\cdots \widetilde{G}_{w_uy_u}\widehat{G}_{w'_1y'_1}\cdots\widehat{G}_{w'_ry'_r}\D_{\b i^{(1)}\b j^{(1)}}^{n_1}\,\cal M_* (\cal H_{T_1}) \cdots  \D_{\b i^{(\zeta)}\b j^{(\zeta)}}^{n_\zeta}\cal M_* (\cal H_{T_\zeta})\,, 
	$$
	 then $u(Q')\geq u(Q)$, and other parameters are unchanged.
	
	\item If $Q'$ is generated through applying $\partial_{jj'}$ on $\int_\cal I$, then $\nu_3(Q')=\nu_3(Q)-1$, $\zeta(Q')=\zeta(Q)+1$, and other parameters are unchanged. 
	
		\item If $Q'$ is generated through applying $\partial_{jj'}$ on $\int_{\cal J_E}$, then $\nu_4(Q')=\nu_4(Q)-1$, $\zeta(Q')=\zeta(Q)+1$, and other parameters are unchanged.

	\item If $Q'$ is generated through applying $\partial_{jj'}$ on $q^{(m)}(X_E)$ and $\tr \widetilde{G}$, then $\nu(Q')=\nu(Q)+1$, $u(Q')=u(Q)+1$, $\nu_4(Q')=\nu_4(Q)+1$, and other parameters are unchanged.
	
		\item If $Q'$ is generated through applying $\partial_{jj'}$ on $q^{(m)}(X_E)$ and $\int_{\cal J_E}$, then $\nu(Q')=\nu(Q)+1$, $u(Q')=u(Q)+1$, $\zeta(Q')=\zeta(Q)+1$, and other parameters are unchanged.

	\item If $Q'$ is generated through applying $\partial_{jj'}$ on $\exp(\ii t Y)$ and $\widetilde{G}_{\b v\b w}$, then $u(Q')=u(Q)+1$, $\nu_3(Q')=\nu_3(Q)+1$, $\theta_1(Q')=\theta_1(Q)-1$, and other parameters are unchanged.
	
		\item If $Q'$ is generated through applying $\partial_{jj'}$ on $\exp(\ii t Y)$ and $\int _{\cal I}$, then $u(Q')=u(Q)+1$, $\zeta(Q')=\zeta(Q)+1$, $\theta_1(Q')=\theta_1(Q)-1$, and other parameters are unchanged. 	
		
			\item If $Q'$ is generated through applying $\partial_{jj'}$ on $\exp(\ii t Y)$, $q(X_E)$ and $\tr \widetilde{G}$, then $\nu(Q')=\nu(Q)+1$, $u(Q')=u(Q)+2$, $\nu_3(Q')=\nu_3(Q)+1$, $\nu_4(Q')=\nu_4(Q)+1$, $\theta_1(Q')=\theta_1(Q)-1$, and other parameters are unchanged.
			
			\item If $Q'$ is generated through applying $\partial_{jj'}$ on $\exp(\ii t Y)$, $q(X_E)$ and $\int_{\cal J_E}$, then $\nu(Q')=\nu(Q)+1$, $u(Q')=u(Q)+2$, $\nu_3(Q')=\nu_3(Q)+1$, $\zeta(Q')=\zeta(Q)+1$, $\theta_1(Q')=\theta_1(Q)-1$, and other parameters are unchanged. 	 				
\end{enumerate}
In all the above cases, we have
\[
\cal E_9(Q')\leq \cal E_9(Q)\cdot N^{2\delta+12\xi} \,.
\]
\end{lemma}
 Applying Lemma \ref{lemma 5.7} $s$ number of times, we see that each $Q_k$ in \eqref{Q_k} satisfies
\[
\cal E_9(Q_k)\leq \cal E_9(Q_*)\cdot N^{(2\delta+12\xi)s}=f^{1-s}N^{\delta+19\xi}\cdot N^{(2\delta+12\xi)s}\leq N^{-\tau/2}\,,
\]
where in the last step we used \eqref{parameters} and $s\geq 2$. Thus we have proved that $\sum_{s=2}^{\ell} L^{(6)}_s \prec N^{-\delta}$.

On the other hand, the estimate of $L_s^{(5)}$ is relatively easy: it does not contain the isotropic factor $\b v_i$. Applying \eqref{diff}, Corollary \ref{cor1.5}, Proposition \ref{prop4.1}, and Corollary \ref{lem4.1}, it is not hard to see that $\sum_{s=2}^{\ell} L^{(5)}_s \prec N^{-\delta}$. This finishes the proof of \eqref{5.28}.

\subsubsection{Proof of Lemma \ref{lemma 5.6}} \label{section 5.3.1}

Let $Q\in \cal Q_o$ in the form of \eqref{3.151}. Our task here is to improve Lemma \ref{lemma 5.5} by a factor of $N^{-1/2}$. We split the discussion into the following two situations.

\subsubsection*{Situation 1}
The odd index shows up as an off-diagonal index of the real part of the Green function. Let 
\begin{equation}\label{xi10}
\cal E_{10}(Q)\deq \cal E_8(Q)\cdot N^{-1/2}\leq \cal E_{9}(Q)\cdot N^{-13\xi}\,.
\end{equation} 
We shall prove that
\begin{equation} \label{5.32}
	\bb E \cal M(Q)\prec \sum_{k=1}^\ell \bb E \cal M(Q_k)+\cal E_{10}(Q)
\end{equation}
for some fixed $\ell$. Here $Q_k\in \cal Q_o$ and satisfy 
\begin{equation} \label{5.33}
	\cal E_8(Q_k)\leq \cal E_8(Q)\cdot N^{-\tau/2} \,.
\end{equation}
The reason that we need this stronger estimate will become apparent in Situation 2. 

W.L.O.G.\,we can assume $w'_1\in O(Q)$ and $w'_1\not\equiv y'_1$. Similar to \eqref{5/5}, we have
	\[
	\widehat{G}_{w'_1y'_1}-\langle \b e_{w'_1}, y'_1 \rangle \ul{\widehat{G}}=\re( \ul{AG}{G}_{w'_1y'_1}- \ul{G}(AG)_{w'_1y'_1})=\re( \ul{HG}{G}_{w'_1y'_1}- \ul{G}(HG)_{w'_1y'_1})+O_{\prec}(N^{-1})\,.
	\]
	We write $\langle \b e_{w'_1}, y'_1 \rangle$ in the above to include the case that $y'_1\in \{\b v,\b w\}$; when $y'_1$ is an index, $\langle \b e_{w'_1}, y'_1 \rangle=\delta_{w'_1y'_1}$.  For $Q \in \cal Q$, let us abbreviate 
	\[
	\int_{\bf Q} \deq \int_{\cal I^{\nu_3}}\int_{\cal J_{E_1}}\cdots \int_{\cal J_{E_{\nu_4}}} \dd x_1\cdots \dd x_{\nu_4}\dd E_1\cdots \dd E_{\nu_3}\,,
	\] 
	and the random variable $R\equiv R(Q)$ is defined such that $Q=\int_{\bf Q} R$. Set $\mathring{R}\deq R/\widehat{G}_{w'_1y'_1}$. We have
	\begin{equation} \label{5.34}
		\begin{aligned}
			\bb E \cal M(Q)=&\,		\sum_{i_1,...,i_\nu,w,y} N^{-1}\bb E H_{wy} \int_{\bf Q} \re (G_{yw}G_{w'_1y'_1}) \mathring{R} - \sum_{i_1,...,i_\nu,w} \bb E H_{w'_1w} \int_{\bf Q}\re (G_{wy'_1}\ul{G}) \mathring{R}\\
			&+\sum_{i_1,...,i_\nu} \bb E \langle \b e_{w'_1}, y'_1 \rangle \int_{\bf Q}\widehat{\ul{G}} \mathring{R}+O_\prec(N^{-1})\sum_{i_1,...,i_\nu} \bb E \Big| \int_{\bf Q} \mathring{R}\Big|\,.
		\end{aligned}
	\end{equation}

As $\sum_{w'_1} |\langle \b e_{w'_1}, y'_1 \rangle  |\leq N^{1/2}$, it is easy to follow the proof strategy of Lemma \ref{lemma 5.5} (i.e.\,applying Lemma \ref{ward}, Corollary \ref{cor1.5}, Proposition \ref{prop4.1}, and Corollary \ref{lem4.1}) and show that the last two terms on RHS of \eqref{5.34} are bounded by $O_{\prec}(\cal E_{10}(Q))$. Together with Lemma \ref{lem:cumulant_expansion}, we get
\begin{equation} \label{5.35}
	\begin{aligned}
		\bb E \cal M(Q)=&\,		\bb E\bigg[\sum_{s=1}^\ell\sum_{i_1,...,i_\nu,w,y} N^{-1}\cal C_{s+1}(H_{wy}) \partial_{wy}^s\bigg(\int_{\bf Q} \re (G_{yw}G_{w'_1y'_1}) \mathring{R}\bigg)\bigg] \\
		&- \bb E\bigg[\sum_{s=1}^\ell \sum_{i_1,...,i_\nu,w} \cal C_{s+1}(H_{w'_1w})\partial^s_{w'_1w}\bigg(\int_{\bf Q} \re (G_{wy'_1}\ul{G}) \mathring{R}\bigg)\bigg]+O_{\prec}(\cal E_{10}(Q))\\
		=&\,		\bb E\bigg[\sum_{s=1}^\ell\sum_{i_1,...,i_\nu,w,y} N^{-1}\cal C_{s+1}(H_{wy}) \partial_{wy}^s\bigg(\int_{\bf Q}(\widehat{G}_{yw}\widehat{G}_{w'_1y'_1}-\widetilde{G}_{yw}\widetilde{G}_{w'_1y'_1}) \mathring{R}\bigg)\bigg]\\
		-&\bb E\bigg[\sum_{s=1}^\ell \sum_{i_1,...,i_\nu,w} \cal C_{s+1}(H_{w'_1w})\partial^s_{w'_1w} \bigg(\int_{\bf Q}(\widehat{G}_{wy'_1}\widehat{\ul{G}}-\widetilde{G}_{wy'_1}\widetilde{\ul{G}}) \mathring{R}\bigg)\bigg]+O_{\prec}(\cal E_{10}(Q))\\
		\eqd&\,\sum_{s=1}^{\ell}L_{s}^{(7)}+\sum_{s=1}^{\ell}L_{s}^{(8)}+\sum_{s=1}^{\ell}L_{s}^{(9)}+\sum_{s=1}^{\ell}L_{s}^{(10)}+O_{\prec}(\cal E_{10}(Q))\,.
	\end{aligned}
\end{equation}
To streamline the formula, in the above we use $L^{(7)}_s, L^{(8)}_s, L^{(9)}_s$ and $L^{(10)}_s$ to denote the terms that contain the factors $\widehat{G}_{yw}\widehat{G}_{w'_1y'_1}$, $\widetilde{G}_{yw}\widetilde{G}_{w'_1y'_1}$,   $\widehat{G}_{wy'_1}\widehat{\ul{G}}$ and $\widetilde{G}_{wy'_1}\widetilde{\ul{G}}$ respectively. 

Let us look at $L_{s}^{(9)}$ carefully, which is the most tricky term. By \eqref{diff}, we have
	\begin{align}
		L_{1}^{(9)}=&\,\bb E\bigg[ \sum_{i_1,...,i_\nu,w} N^{-1} \int_{\bf Q}\re(G_{ww}{G}_{w'_1y'_1})\widehat{\ul{G}}\mathring{R}\bigg]+\bb E\bigg[ \sum_{i_1,...,i_\nu,w} N^{-1} \int_{\bf Q}\re(G_{ww'_1}{G}_{wy'_1})\widehat{\ul{G}}\mathring{R}\bigg]\nonumber\\
		&-\bb E\bigg[ \sum_{i_1,...,i_\nu,w} N^{-1} \int_{\bf Q}\widehat{G}_{wy'_1}\partial_{wy}(\widehat{\ul{G}}\mathring{R})\bigg]-\bb E\bigg[ \sum_{i_1,...,i_\nu,w} N^{-1} \partial_{wy}\bigg(\int_{\bf Q}\bigg)\widehat{G}_{wy'_1}\widehat{\ul{G}}\mathring{R}\bigg]+O_{\prec}(\cal E_{10}(Q))\nonumber\\
		&\eqd L_{1,1}^{(9)}+\cdots +L_{1,4}^{(9)}+O_{\prec}(\cal E_{10}(Q))\nonumber\,.
	\end{align}
It is easy to see that 
\[
L_{1,2}^{(9)}=\bb E\bigg[ \sum_{i_1,...,i_\nu,w} N^{-1} \int_{\bf Q}(\widehat{G}_{ww'_1}\widehat{G}_{wy'_1}-\widetilde{G}_{ww'_1}\widetilde{G}_{wy'_1})\widehat{\ul{G}}\mathring{R}\bigg]\eqd\sum_{k=1}^{2}\bb E\cal M(Q^{(9)}_{1,2,k})\,,
\]
where each $Q^{(9)}_{1,2,k}$ lies in $\cal Q_o$, and they satisfy either $\nu_2(Q^{(9)}_{1,2,1})=\nu_2(Q)+1$ and $u(Q^{(9)}_{1,2,2})=u(Q)+2$, while other parameters are unchanged. Hence each $Q^{(9)}_{1,2,k}$ satisfies \eqref{5.33}. For 
$$
L_{1,3}^{(9)}=\sum_{k=1}^{\ell}\bb E\cal M(Q^{(9)}_{1,3,k})\,,
$$ 
we can follow the strategy of Lemma \ref{lemma 5.7} to discuss the resulting terms. More precisely, let $Q^{(9)}_{1,3}\deq N^{-1} \int_{\bf Q}\widehat{G}_{wy'_1}\widehat{\ul{G}}\mathring{R}$. Note that
\[
\nu({Q})-\theta_1(Q)=\nu({Q^{(9)}_{1,3}})-\theta_1(Q^{(9)}_{1,3}),
\]
while other parameters of $Q$ and $Q^{(9)}_{1,3}$ are the same. Thus
\begin{equation} \label{5.37}
\cal E_8(Q_{1,3}^{(9)})= \cal E_8(Q)\,.
\end{equation}
We get each $Q_{1,3,k}^{(9)}$ from $Q_{1,3}^{(9)}$ through applying $\partial_{wy}$ on different factors of $\widehat{\ul{G}}\mathring{R}$. The changes of the parameters $\nu,\nu_3,\nu_4,u,\zeta,\theta_1,\theta_2$ between $Q_{1,3,k}^{(9)}$ and $Q_{1,3}^{(9)}$ are described by Lemma \ref{lemma 5.7}. In cases (i), (iv), (vi), (viii) of Lemma \ref{lemma 5.7}, i.e.\,the derivative $\partial_{wy}$ hits a Green function, the index $w$ appear twice, and as $\widehat{G}_{wy'_1}$ is a real part, we get $\nu_2(Q^{(9)}_{1,3,k})=\nu_2(Q_{1,3}^{(9)})+1$. In cases (v), (vii) we have $\nu_2(Q^{(9)}_{1,3,k})=\nu_2(Q_{1,3}^{(9)})$, but the value of $\zeta$ will increase. Hence 
\[
\cal E_8(Q_{1,3,k}^{(9)})\leq \cal E_8(Q_{1,3}^{(9)})\cdot N^{-\tau/2}\,.
\]
Combining \eqref{5.37} and the above shows that all $Q_{1,3,k}^{(9)}$ satisfy \eqref{5.33}. By the cases (ii), (iii) in Lemma \ref{lemma 5.7}, we can also estimate $L_{1,4}^{(9)}$ in a similar fashion. As a result, we have reached the relation
\begin{equation} \label{5.38}
		L_{1}^{(9)}=\bb E\bigg[ \sum_{i_1,...,i_\nu,w} N^{-1} \int_{\bf Q}\re(G_{ww}{G}_{w'_1y'_1})\widehat{\ul{G}}\mathring{R}\bigg]+\sum_{k=1}^{\ell}\bb E \cal M(Q_{1,k}^{(9)})+O_{\prec}(\cal E_{10}(Q))\,,
\end{equation}
where each $Q_{1,k}^{(9)}$ satisfy \eqref{5.33}. In addition, it is also easy to see from \eqref{diff} that $Q_{1,k}^{(9)}\in \cal Q_o$. 

When $s\geq 2$, let 
\[
Q^{(9)}\deq-\cal C_{s+1}(H_{w'_1w})\int_{\bf Q} \widehat{G}_{wy'_1}\widehat{\ul{G}} \mathring{R}\,,\quad \mbox{and} \quad L_s^{(9)}=\sum_{k=1}^\ell\bb E \cal M(Q_{s,k}^{(9)})\,.
\]
It is easy to see that
\[
\nu(Q^{(9)})=\nu(Q)+1\,, \quad  \theta_1(Q^{(9)})=\theta_1(Q)+1\,, \quad \theta_2(Q^{(9)})=\theta_2(Q)+s-1\,,
\]
while parameter $\nu_3,\nu_4,u,\zeta$ are the same between $Q$ and $Q^{(9)}$. Thus
\begin{equation*} 
	 \cal E_8(Q^{(9)}) \cdot (N^{-1/3+6\xi})^{-\nu_2(Q^{(9)})}\leq \cal E_8(Q)\cdot f^{1-s}\cdot (N^{-1/3+6\xi})^{-\nu_2(Q)}\,.
\end{equation*} 
We get each $Q_{s,k}^{(9)}$ from $Q^{(9)}$ through applying $\partial_{wy}^s$. Everytime we apply one $\partial_{wy}$, the changes of the parameters $\nu,\nu_3,\nu_4,u,\zeta,\theta_1,\theta_2$ are described by Lemma \ref{lemma 5.7}. Thus
\begin{equation*} 
	\cal E_8(Q^{(9)}_{s,k})\cdot (N^{-1/3+6\xi})^{-\nu_2(Q^{(9)}_{s,k})} \leq \cal E_8(Q^{(9)}) \cdot (N^{-1/3+6\xi})^{-\nu_2(Q^{(9)})}\cdot N^{(2\delta+12\xi)s}\,.
\end{equation*}
Note that $\nu_2(Q_{s,k}^{(9)})\geq \nu_2(Q)$. Combining the above two equations, we see that each $Q_{s,k}^{(9)}$ satisfies \eqref{5.33}. In addition, it is also easy to see from \eqref{diff} that $Q_{s,k}^{(9)}\in \cal Q_o$. Combining with \eqref{5.38}, we get
\begin{equation} \label{5.41}
	\sum_{s=1}^{\ell}L_{s}^{(9)}=\bb E\bigg[ \sum_{i_1,...,i_\nu,w} N^{-1}\int_{\bf Q} \re(G_{ww}{G}_{w'_1y'_1})\widehat{\ul{G}}\mathring{R}\bigg]+\sum_{k=1}^{\ell}\bb E \cal M(Q_{k}^{(9)})+O_{\prec}(\cal E_{10}(Q))\,,
\end{equation}
where each $Q_k^{(9)}\in \cal Q_o$ satisfy \eqref{5.33}. 

Similar computation can be done for other terms in \eqref{5.35}. We get
\begin{equation} \label{5.42}
	\sum_{s=1}^{\ell}L_{s}^{(7)}=-\bb E\bigg[ \sum_{i_1,...,i_\nu,w,y} N^{-2} \int_{\bf Q}\re(G_{ww}{G}_{w'_1y'_1}){\widehat{G}_{yy}}\mathring{R}\bigg]+\sum_{k=1}^{\ell}\bb E \cal M(Q_{k}^{(7)})+O_{\prec}(\cal E_{10}(Q))
\end{equation}
and
\begin{equation} \label{5.43}
	\sum_{s=1}^{\ell}L_{s}^{(8)}=\sum_{k=1}^{\ell}\bb E \cal M(Q_{k}^{(8)})+O_{\prec}(\cal E_{10}(Q))\,,\quad \sum_{s=1}^{\ell}L_{s}^{(10)}=\sum_{k=1}^{\ell}\bb E \cal M(Q_{k}^{(10)})+O_{\prec}(\cal E_{10}(Q))\,.
\end{equation}
Here each $Q_k^{(7)},Q_k^{(8)}, Q_k^{(10)}\in \cal Q_o$ satisfy \eqref{5.33}. Note the cancellation between $L_{s}^{(7)}$ and $L_s^{(9)}$. Inserting \eqref{5.41} -- \eqref{5.43} into \eqref{5.35} yields the desired result. 

\subsubsection*{Situation 2}
The odd index shows up as an off-diagonal index of the imaginary part of the Green function. We shall prove that
\begin{equation} \label{5.321}
	\bb E \cal M(Q)\prec \sum_{k=1}^\ell \bb E \cal M(Q_k)+\cal E_{9}(Q)
\end{equation}
for some fixed $\ell$. Here $Q_k\in \cal Q_o$ and satisfy 
\begin{equation} \label{5.331}
	\cal E_8(Q_k)\leq \cal E_8(Q)\cdot N^{-\tau/3} \,.
\end{equation}
Iterating \eqref{5.32} and \eqref{5.321} concludes the proof of Lemma \ref{lemma 5.6}. 

W.L.O.G.\,we can assume $w_1\in O(Q)$ and $w_1\not\equiv y_1$. Let $\mathring{Q}\deq Q/\widetilde{G}_{w_1y_1}$. Similar to \eqref{5.35}, we have
\begin{equation} \label{5.44}
	\begin{aligned}
		\bb E \cal M(Q)=&\,		\bb E\bigg[\sum_{s=1}^\ell\sum_{i_1,...,i_\nu,w,y} N^{-1}\cal C_{s+1}(H_{wy}) \partial_{wy}^s\bigg(\int_{\bf Q} (\widetilde{G}_{yw}\widehat{G}_{w_1y_1}+\widehat{G}_{yw}\widetilde{G}_{w_1y_1}) \mathring{R}\bigg)\bigg]\\
		-&\bb E\bigg[\sum_{s=1}^\ell \sum_{i_1,...,i_\nu,w} \cal C_{s+1}(H_{w_1w})\partial^s_{w_1w} \bigg(\int_{\bf Q}(\widetilde{G}_{wy_1}\widehat{\ul{G}}+\widehat{G}_{wy_1}\widetilde{\ul{G}}) \mathring{R}\bigg)\bigg]+O_{\prec}(\cal E_{10}(Q))\\
		\eqd&\,\sum_{s=1}^{\ell}L_{s}^{(11)}+\sum_{s=1}^{\ell}L_{s}^{(12)}+\sum_{s=1}^{\ell}L_{s}^{(13)}+\sum_{s=1}^{\ell}L_{s}^{(14)}+O_{\prec}(\cal E_{10}(Q))\,.
	\end{aligned}
\end{equation}
Let us look at $L_1^{(13)}$, which will reveal the reason of distinguishing the real and imaginary cases. Comparing with $L_1^{(9)}$, we see that now have the imaginary part of the Green function $\widetilde{G}_{wy_1}$, instead of $\widetilde{G}_{wy'_1}$. As a result, if the derivative $\partial_{w_1w}$ hits some $\widetilde{G}$, the parameter $\nu_2$ may not increase. For instance, if the derivative hits $\exp(\ii t Y)$, $q(X_E)$ and $\tr \widetilde{G}=\sum_y \widetilde{G}_{yy}$, one of the terms will be
\[
\bb E\cal M(Q_{1}^{(13)})\deq \sum_{i_1,...,i_\nu,w,y} \int_{\bf Q}\int_{\cal I}\int_{\cal J_E}\widetilde{G}_{wy_1}\widehat{\ul{G}} \mathring{R}\cdot \widetilde{G}_{\b v \b w}q'(X_E)\widetilde{G}_{y w}\widehat{G}_{w_1y}\,,
\]
where $Q_1^{(13)}\in\cal Q_o$. As $\widetilde{G}_{wy_1}$ and $\widetilde{G}_{y w}$ are both the imaginary part, $\nu_2(Q_1^{(13)})=\nu_2(Q)$. In addition, $\nu(Q_1^{(13)})=\nu(Q)+2,\nu_3(Q_1^{(13)})=\nu_3(Q)+1,\nu_4(Q_1^{(13)})=\nu_4(Q)+1,u(Q_1^{(13)})=u(Q)+2$, and other parameters are unchanged. As a result,
\begin{equation} \label{5.45}
\cal E_8(Q_{1}^{(13)})=\cal E_8(Q)\cdot N^{2\delta+12\xi}\quad \mbox{and} \quad \cal E_{10}(Q_{1}^{(13)})=\cal E_{10}(Q)\cdot N^{2\delta+12\xi}\leq \cal E_9(Q)\,.
\end{equation}
where in the last step we used \eqref{xi10}. However, $w_1$ is still an index that appear odd many times in $Q_1^{(13)}$, and now it appears in the real part of a Green function. Thus we can re-expand $\bb E\cal M(Q_{1}^{(13)})$ using \eqref{5.32}, which leads to 
\begin{equation*} 
	\bb E \cal M(Q_1^{(13)})\prec \sum_{k=1}^\ell \bb E \cal M(Q_{1,k}^{(13)})+\cal E_{10}(Q_{1}^{(13)})\prec \sum_{k=1}^\ell \bb E \cal M(Q_{1,k}^{(13)})+\cal E_{9}(Q)\,,
\end{equation*}
where each $Q_{1,k}^{(13)}\in \cal Q_o$ satisfy 
\begin{equation*} \label{5//33}
	\cal E_8(Q_{1,k}^{(13)})\leq \cal E_{8}(Q_1^{(13)})\cdot N^{-\tau/2}\leq \cal E_8(Q)\cdot N^{-\tau/3}\,.
\end{equation*}
Here in the above two estimates we used \eqref{5.45}. In other words, each $Q_{1,k}^{(13)}$ satisfies \eqref{5.331}. Aside from this, it can be checked that other terms on RHS of \eqref{5.44}, including the higher-order cumulant terms, can all be treated using the argument in Situation 1. This leads to
\begin{equation*}
	\sum_{s=1}^{\ell}L_{s}^{(11)}+\sum_{s=1}^{\ell}L_{s}^{(12)}=-\bb E\sum_{i_1,...,i_\nu,w,y} N^{-2}\int_{\bf Q} \im({G}_{yy}G_{ww}{G}_{w_1y_1}) \mathring{R}+\sum_{k=1}^\ell \bb E \cal M(Q^{(11,12)}_k)+O_{\prec}(\cal E_{10}(Q))
\end{equation*}
and 
\begin{equation*}
	\sum_{s=1}^{\ell}L_{s}^{(13)}+\sum_{s=1}^{\ell}L_{s}^{(14)}=\bb E\sum_{i_1,...,i_\nu,w} N^{-1}\int_{\bf Q} \im(\ul{G}G_{ww}{G}_{w_1y_1}) \mathring{R}+\sum_{k=1}^\ell \bb E \cal M(Q^{(13,14)}_k)+O_{\prec}(\cal E_{9}(Q))\,,
\end{equation*}
where each $Q^{(11,12)}_k, Q^{(13,14)}_k $ satisfies \eqref{5.331}. Combining the above two relations yields \eqref{5.321} as desired. This finishes the proof of Lemma \ref{lemma 5.6}.

	\subsection{Conclusion}
By \eqref{5.5}, \eqref{5.6}, \eqref{5.21}, and  \eqref{5.28}, we get
	\[
	g'(t)=\ii \bb E [Y \exp(\ii t Y) ]=-t \bb E (Y^{(1)}\exp(\ii t Y))+\ii t\langle \b v,\b w \rangle g'(t)+\ii \langle \b v ,\b w\rangle g(t)+O_{\prec}(N^{-\delta})\,.
	\]
	Similarly, we can also show that
	\[
	\bb E (Y^{(1)}\exp(\ii t Y))=\ii t \langle \b v,\b w \rangle \bb E(Y^{(1)}\exp(\ii t Y))+  t g'(t)+g(t)+O_{\prec}(N^{-\delta})\,.
	\]
The error term is uniform for all $t\in[-T,T]$. Combining the above two equations yields the ODE
\[
g'(t)=\frac{\ii \langle \b v,\b w\rangle+(\langle \b v,\b w\rangle^2-1)t}{1-2\ii \langle \b v ,\b w \rangle t+(1-\langle\b v,\b w \rangle^2)t^2}g(t)+O_{\prec}(N^{-\delta})\,.
\]
 As $g(0)=1$, we get
	\[
	g(t)=(1-2\ii \langle \b v ,\b w \rangle t+(1-\langle\b v,\b w \rangle^2)t^2)^{-1/2}+O_{\prec}(N^{-\delta})\,.
	\]
Note that $\langle \b  v,\b z\rangle \langle \b  w,\b z \rangle\overset{d}{=}Z_1Z_2$, where $Z_1$ and $Z_2$ are standard Gaussian random variables with $\mathrm{Cov}(Z_1,Z_2)=\langle \b v,\b w \rangle$. Thus 
\[
\bb E  \exp\big(\ii t\langle \b  v,\b z\rangle \langle \b  w,\b z\rangle\big)=(1-2\ii \langle \b v ,\b w \rangle t+(1-\langle\b v,\b w \rangle^2)t^2)^{-1/2}\,.
\]
This finished the proof of \eqref{5.1}.

\section{Proof of Theorem \ref{thm1.9}} \label{appA}

In this section, we prove Theorem \ref{thm1.9} for $k=1$; the general case follows in a similar fashion. More precisely, we shall show that there exists fixed $\varepsilon>0$ such that
\begin{equation} \label{A1}
\bb E \exp\bigg(\ii  t\frac{N}{\sqrt{2\tr B^2}} \langle {\b u^H_1}, B {\b u^H_1}\rangle \bigg)=\bb E  \exp(\ii t Z)+O(N^{-\varepsilon})	
\end{equation}
uniformly for all $t\in [-T,T]$. Here $B\equiv B_1$ and $Z\equiv Z_1$. The steps are very close to those in Section \ref{sec5}. In fact, as the higher order terms decay very fast in the dense case, we do not need to expand recursively as in Section \ref{sec5.3}, which makes the argument simpler. We will make use of the following priori estimates.

\begin{proposition} \label{propA1}
	(i) (Isotropic local law and delocalization, \cite{KY2}) Recall the definition of $\cal G$ from \eqref{2.1}. Fix $\b v,\b w \in \bb S^{N-1}$. We have
	\begin{equation*} 
	|\langle \b v,\cal{G}(z) \b w \rangle-\langle \b v, \b w \rangle m_{\mathrm{sc}}(z)| \prec \sqrt{\frac{\im m_{sc}(z)}{N\eta}}+\frac{1}{N\eta} \quad\mbox{and} \quad \ul{\cal{G}}(z)-m_{sc}(z) \prec  \frac{1}{N\eta}
	\end{equation*} 
	uniformly in $z  \in \{z=E+\ii \eta: E\in \bb R, \eta >0\}$. In addition,
	\[
	\max_{1\leq i\leq N}|\langle \b u^H_i,\b v \rangle| \prec N^{-1/2}\,,
	\]
	and
	\begin{equation} \label{rigidity}
\lambda^H_k-\gamma_{k}^{sc}\prec \frac{1}{N^{2/3}\min\{k,N-k\}^{1/3}}\,, \quad 1\leq k \leq N\,. 
	\end{equation}
Here $\gamma_{k}^{sc}$ is defined through 
$$
\frac{k-1/2}{N}=\int_{\gamma_k^{sc}}^{2}\rho_{sc}(x)\dd x\,.
$$
	(ii) (Level repulsion, \cite[Proposition 2.4]{KY1}) There exists fixed $\varepsilon_0>0$, such that the following holds. For any $\varepsilon\in (0,\varepsilon_0)$, there exists $\alpha>0$ such that
	\begin{equation} \label{A2}
			\mathbb P(\lambda^{H}_{1}-\lambda^{H}_{2} \leq N^{-2/3-\varepsilon})=O (N^{-\varepsilon-\alpha})\,.
	\end{equation}
(iii) (Eigenstate Thermalization Hypothesis, \cite[Theorem 1.2]{CEH1}) Let $B \in \bb R^{N\times N}$ be deterministic and traceless. Then
\[
\max_{1\leq i,j\leq N}|\langle {\b u^H_i}, B {\b u^H_j}\rangle | \prec \frac{\sqrt{\tr |B|^2}}{N}\,.
\]
\end{proposition}

Now let us start the proof of \eqref{A1}. As  the statement is homogeneous in $B$, it suffices to assume $\|B\|=1$. Let $\varepsilon_0$ be as in Proposition \ref{propA1} (iii). Fix parameters $\xi \in (0,\varepsilon_0/100)$ and $\delta\in (0,\xi/100)$, and set 
$$
\eta_+\deq N^{-2/3-\xi}\,, \quad \eta_-\deq N^{-2/3-6\xi}\,.
$$
Similar to Lemma \ref{lemma 5.1}, we have the following result. Let $q: \bb R \to \bb R_+$ be a smooth cut-off function such that $q(x)=1$ for $|x|\leq 1/3$ and $q(x)=0$ for $x\geq 2/3$. We abbreviate
\begin{equation*} 
	\cal 	I\deq [2-N^{-2/3+\delta},2+N^{-2/3+\delta}] \,,\quad \mbox{and}\quad \cal J_E\deq [E-N^{\delta}\eta_+,2+N^{-2/3+\delta}]\,.
\end{equation*}
Similar to Lemma \ref{lemma 5.1}, we can use Proposition \ref{propA1} to show that
	\begin{equation*}
		\begin{aligned}
			&\bb E \exp\bigg(\frac{\ii  tN}{\sqrt{\tr B^2}} \langle {\b u^H_1}, B {\b u^H_1}\rangle \bigg)\\
			=\,&\bb E \exp\bigg(\frac{\ii tN^2}{\pi \sqrt{\tr B^2}}\int _{\cal I} \ul{B\widetilde{\cal G}}(E+\ii \eta_+)q(\cal X_E)\,\dd E\bigg)+O(N^{-\delta})
			\eqd\bb E \exp( \ii t\cal Y)+O(N^{-\delta})\,,
		\end{aligned}
	\end{equation*}
where
$
\cal X_E\deq \int_{\cal J_E} \tr \widetilde{\cal G}(x+\ii \eta_-)\dd x.
$
In the sequel, we abbreviate $\frak g(t)\deq \bb E \exp(\ii  t\cal Y)$. Similar to \eqref{5/5}, we can use resolvent identity and $\tr B=0$ to show that
\[
\underline{\widetilde{\cal G} B}=\im (\ul{H\cal G}\cdot\ul{\cal GB}-\ul{\cal G}\cdot\ul{H\cal GB})\,.
\]
Together with Lemma \ref{lem:cumulant_expansion}, we get
	\begin{align}
		{\frak g}'(t)
		&=\ii \bb E \bigg[\frac{N^2}{\pi\sqrt{2\tr B^2}}\int_{\cal I}\bigg( \frac{1}{N}
		\sum_{ij}H_{ij}\im( \cal G_{ji}\cdot \ul{\cal G B})-\frac{1}{N}\sum_{ij}H_{ij}\im ( (\cal GB)_{ji}\cdot \ul{\cal G})\bigg)q(\cal X_E)\dd E\cdot \exp( \ii t\cal Y)\bigg]\nonumber\\
		&=\ii \bb E \bigg[\sum_{s=1}^{\ell}\sum_{ij}\cal C_{s+1}(H_{ij})\partial_{ij}^{s}\bigg(\frac{N}{\pi\sqrt{2\tr B^2}}\int_{\cal I}  \im(\cal G_{ji}\cdot \ul{\cal GB})q(\cal X_E)\dd E\cdot \exp( \ii t\cal Y)\bigg)\bigg]\nonumber\\
		& \quad - \ii \bb E \bigg[\sum_{s=1}^{\ell}\sum_{ij}\cal C_{s+1}(H_{ij})\partial_{ij}^{s}\bigg(\frac{N}{\pi\sqrt{2\tr B^2}}\int_{\cal I}  \im((\cal GB)_{ji}\cdot \ul{\cal G})q(\cal X_E)\dd E\cdot \exp( \ii t\cal Y)\bigg)\bigg]+O_{\prec}(N^{-\delta})\nonumber\\
	&\eqd\sum_{s=1}^{\ell} L^{(15)}_s+ \sum_{s=1}^{\ell} L^{(16)}_s+O_{\prec}(N^{-\delta})\label{A3}\,.
	\end{align}
The rest of this section computes the RHS of \eqref{A3}.

\subsection{The leading terms, part I} \label{secA1}
Let us first compute $L_1^{(16)}$. Similar to the computation of \eqref{5.7}, together with \eqref{diff} and $\cal C_{2}(H_{ij})=(1+O(\delta_{ij}))/N$, it is not hard to see that 
\begin{align}
L_1^{(16)}=&\,\ii \bb E \bigg[\frac{N^2}{\pi \sqrt{2\tr B^2}}\int_\cal I \im (\ul{\cal GB}\cdot \ul{\cal G}^2)q(\cal X_E)\dd E\cdot \exp(\ii t\cal Y)\bigg]\nonumber\\
&+\ii \bb E \bigg[\frac{N}{\pi \sqrt{2\tr B^2}}\int_\cal I \im (\ul{\cal G^2B}\cdot \ul{\cal G})q(\cal X_E)\dd E\cdot \exp(\ii t\cal Y)\bigg]\nonumber\\
&+\ii \bb E \bigg[\frac{2}{\pi \sqrt{2\tr B^2}}\int_\cal I \im (\ul{\cal G^3B})q(\cal X_E)\dd E\cdot \exp(\ii t\cal Y)\bigg]\nonumber\\
&+\ii \bb E \bigg[\frac{2}{N\pi\sqrt{2\tr B^2} }\sum_{ij} \int_{\cal I} \im( (\cal GB)_{ji}\cdot\ul{\cal G}) q'(X_E) \bigg(\int_{\cal J_E} \im ((\cal G^{(x)})^2)_{j i}\dd x\bigg)\dd E\cdot \exp( \ii t\cal Y)\bigg]\nonumber\\
&- \bb E \bigg[\frac{tN}{\pi^2 \tr B^2}\sum_{ij}\int_{\cal I^2}\im( (\cal GB)_{ji}\cdot\ul{\cal G})\im((\cal G'B\cal G')_{ji}) q(\cal X_E) q(\cal X_{E'})\dd E \dd E'\exp(\ii t\cal Y)\bigg]\nonumber\\
&+ \bb E \bigg[\frac{tN}{\pi^2 \tr B^2}\sum_{ij}\int_{\cal I^2}\im( (\cal GB)_{ji}\cdot\ul{\cal G})\im(\ul{\cal G'B}) q(X_E) q'(\cal X_{E'})\nonumber\\
&\quad \quad \quad \cdot \bigg(\int_{\cal J_{E'}} \im ((\cal G^{(x)})^2)_{j i}\dd x\bigg)\dd E \dd E'\exp(\ii t\cal Y)\bigg]+O(N^{-\delta})\nonumber\\
\eqd&\,L_{1,1}^{(16)}+\cdots +L_{1,6}^{(16)}+O(N^{-\delta})\label{A4}\,.
	\end{align}
Here we used the abbreviation $\cal G'\deq\cal G(E'+\ii \eta_+)$ and $\cal G^{(x)}\deq\cal G(x+\ii \eta_-)$. $L_{1,1}^{(16)}$ is a leading term which will later be canceled by a term in $L_1^{(15)}$. To estimate $L_{1,2}^{(16)}$, note that Proposition \ref{propA1}  imply
\[
\ul{\cal G^2B} = \frac{1}{N}\sum_{ij\alpha}\frac{\b u^{H}_\alpha(i)\b u^{H}_{\alpha}(j)}{(\lambda_\alpha-E-\ii \eta_+)^2} B_{ji}=\frac{1}{N} \sum_{\alpha}\frac{\langle \b u^{H}_\alpha , B \b u^{H}_\alpha \rangle}{(\lambda_{\alpha}-E-\ii \eta_+)^2} \prec \eta^{-1}_+\ul{\widetilde{\cal G}}\cdot \frac{\sqrt{\tr B^2}}{N}\prec N^{-2/3+2\xi} \sqrt{\tr B^2}\,.
\]
It is not hard to deduce that 
$$
L_{1,2}^{(16)}\prec N\cdot N^{-2/3+\delta}\cdot N^{-2/3+2\xi}\prec N^{-\delta}\,.
$$
Similarly, we can also show that $L_{1,3}^{(16)}\prec N^{-2/3+6\xi}$\,. Moreover, as
\[
\frac{1}{N}\sum_{ij} (\cal G B)_{ji} ((\cal G^{(x)})^2)_{ji}=\ul{(\cal G^{(x)})^2\cal GB}
\]
and $|\cal J_E| \prec N^{-2/3+\delta}$, we can again use spectral decomposition and Proposition \ref{propA1} to show that $L_{1,4}^{(16)}\prec N^{-4/3+20\xi}$. Note that $L_{1,4}^{(16)}$ is even smaller than $L_{1,3}^{(16)}$: the derivative hits $ q(\cal X_E)$ and creates an integration over $\cal J_E$. The same kind of smallness also occurs in $L_{1,6}^{(16)}$, and one can deduce that $L^{(16)}_{1,6}\prec N^{-4/3+30\xi}$. For $L_{1,5}^{(16)}$, the leading contribution comes from taking the two imaginary parts on the first $\cal G$ and $\cal G'$. More precisely, by spectral decomposition and Proposition \ref{propA1}, we have
\begin{equation*}
	\begin{aligned}
&\frac{1}{N}\sum_{ij}(\widetilde{\cal G}B)_{ji} \cdot \ul{\widehat{\cal G}}\cdot (\widehat{\cal G'}B\widetilde{\cal G'})_{ji} = \ul{\widetilde{\cal G'}B\widehat{\cal G'}\widetilde{\cal G}B}\cdot \ul{\widehat{\cal G}} \prec |\ul{\widetilde{\cal G'}B\widehat{\cal G'}\widetilde{\cal G}B}|\\
\leq &\,\bigg|\frac{1}{N}\sum_{\alpha,\alpha'} \frac{\eta_+}{(\lambda_\alpha-E')+\eta_+^2}\langle \b u^{H}_\alpha , B \b u^H_{\alpha'}\rangle \frac{1}{|(\lambda_{\alpha'}-E'-\ii \eta_+)(\lambda_{\alpha'}-E-\ii \eta_+)|}\langle \b u^{H}_{\alpha'} , B \b u^H_{\alpha}\rangle\bigg|\\
\prec & \,\ul{\widetilde{\cal G'}}\cdot N\eta_+^{-1}(\widetilde{\ul{\cal G'}}+\widetilde{\ul{\cal G}})\cdot \frac{\tr B^2}{N^2}\prec \tr B^2 \cdot N^{-1+2\xi}\,.
	\end{aligned}
\end{equation*}
 In addition, note that by \eqref{rigidity} and a dyadic decomposition, we have
 \begin{equation} \label{sum}
 \frac{1}{N}\sum_\alpha \frac{1}{|\lambda_{\alpha}-E'-\ii \eta_+|} \prec 1\,.
 \end{equation}
 Thus
\begin{equation*}
	\begin{aligned}
		&\frac{1}{N}\sum_{ij}(\widehat{\cal G}B)_{ji} \cdot \ul{\widetilde{\cal G}}\cdot (\widetilde{\cal G'}B\widehat{\cal G'})_{ji}  \prec |\ul{\widehat{\cal G'}B\widetilde{\cal G'}\widehat{\cal G}B}|\cdot N^{-1/3+\xi}\\
		\leq &\,\bigg|\frac{1}{N}\sum_{\alpha,\alpha'} \frac{1}{|\lambda_{\alpha}-E'-\ii \eta_+|}\langle \b u^{H}_\alpha , B \b u^H_{\alpha'}\rangle \frac{1}{|(\lambda_{\alpha'}-E'-\ii \eta_+)(\lambda_{\alpha'}-E-\ii \eta_+)|}\langle \b u^{H}_{\alpha'} , B \b u^H_{\alpha}\rangle\bigg|\cdot N^{-1/3+\xi}\\
		\prec & \,\cdot N\eta_+^{-1}(\widetilde{\cal G'}+\widetilde{\cal G})\cdot \frac{\tr B^2}{N^2}\cdot N^{-1/3+\xi}\prec \tr B^2 \cdot N^{-1+3\xi}\,.
	\end{aligned}
\end{equation*}
Similarly, $\frac{1}{N}\sum_{ij}(\widehat{\cal G}B)_{ji} \cdot \ul{\widetilde{\cal G}}\cdot (\widehat{\cal G'}B\widetilde{\cal G'})_{ji} \prec \tr B^2 \cdot N^{-4/3+3\xi}$. Inserting the above estimates into $L_{1,5}^{(16)}$ yields
\begin{equation} \label{A5}
	L_{1,5}^{(16)}=\bb E \bigg[\frac{tN^2}{\pi^2 \tr B^2}\int_{\cal I^2} \ul{\widehat{\cal G'}B\widetilde{\cal G'}\widetilde{\cal G}B}\cdot \ul{\widehat{\cal G}}  \,q(\cal X_E) q(\cal X_{E'})\dd E \dd E'\exp(\ii t\cal Y)\bigg]+O(N^{-\delta})\,.
\end{equation}
Inserting the estimates of $L_{1,2}^{(16)}$, $L_{1,3}^{(16)}$, $L_{1,4}^{(16)}$, $L_{1,6}^{(16)}$ as well as \eqref{A5} into \eqref{A4}, we get
\begin{equation}  \label{whatever}
	\begin{aligned}
		L_1^{(16)}=&\,\ii \bb E \bigg[\frac{N^2}{\pi \sqrt{2\tr B^2}}\int_\cal I \im (\ul{\cal GB}\cdot \ul{\cal G}^2)q(\cal X_E)\dd E\cdot \exp(\ii t\cal Y)\bigg]\\
		&+\bb E \bigg[\frac{tN^2}{\pi^2 \tr B^2}\int_{\cal I^2} \ul{\widehat{\cal G'}B\widetilde{\cal G'}\widetilde{\cal G}B}\cdot \ul{\widehat{\cal G}}  \,q(\cal X_E) q(\cal X_{E'})\dd E \dd E'\exp(\ii t\cal Y)\bigg]\,.
	\end{aligned}
\end{equation}
We can estimate $L_{1}^{(15)}$ through very similar estimates, and show that
\begin{equation} \label{whatever2}
	L_1^{(15)}=-\ii \bb E \bigg[\frac{N^2}{\pi \sqrt{2\tr B^2}}\int_\cal I \im (\ul{\cal GB}\cdot \ul{\cal G}^2)q(\cal X_E)\dd E\cdot \exp(\ii t\cal Y)\bigg]+O(N^{-\delta})\,.
\end{equation}
Note the cancellation among $L_{1}^{(15)}$ and $L_{1}^{(16)}$. Combining the above two results, together with $\ul{\widehat{\G}}'=-1+O_{\prec}(N^{-1/3+\xi})$ and Proposition \ref{propA1} yield
\begin{equation} \label{7.8}
	\begin{aligned}
		L^{(15)}_1+L_1^{(16)}&=
		\,\bb E \bigg[\frac{tN^2}{\pi^2 \tr B^2}\int_{\cal I^2} \ul{\widehat{\cal G'}B\widetilde{\cal G'}\widetilde{\cal G}B} \,q(\cal X_E) q(\cal X_{E'})\dd E \dd E'\exp(\ii t\cal Y)\bigg]+O(N^{-\delta})\\
		&\eqd L_{1,5,1}^{(16)}+O(N^{-\delta})\,.
	\end{aligned}
\end{equation}

\subsection{Higher order terms} \label{sec7.2}
The following prior estimates will be handy for us. 
\begin{lemma} \label{lemmaA2}
	Let $\G\equiv G(E+\ii \eta_+)$ with $E \in \cal I$. We have the following results.
	\begin{enumerate}
		\item $\sum_{ij} |(\G B)_{ji}|^2 \prec N^{1/3+\xi} \tr B^2$.
		\item $\sum_i |(\G B)_{ii}-m_{sc}B_{ii}|^2 \prec N^{-2/3+2\xi}\tr B^2$ and $\sum_i |(\G B)_{ii}|^2 \prec \tr B^2$.
		\item $\sum_i |(\widetilde{\G} B)_{ii}|^2 \prec N^{-2/3+\xi}\tr B^2$.
		\item $(\widetilde{\G}B{\G})_{ij}  \prec  N^{-1/3+\xi}\sqrt{\tr B^2}$ uniformly in $i,j$.
		\item $(\widehat{\G}B\widehat{\G})_{ij}\prec \sqrt{\tr B^2}$ uniformly in $i,j$.
		\item $\sum_{i} |(\G B\widetilde{\G})_{ij}|^2\prec N^{-1/3+\xi}\tr B^2 $ and $\sum_{i} |(\widehat{\G} B\widehat{\G})_{ij}|^2\prec N^{1/3+\xi}\tr B^2 $ uniformly in $j$.
		\item $\partial_{ij}^{r}q(\cal X_E) \prec N^{-1/3+10\xi}$ for any fixed $r\geq 1$.
		\item $\ul{\widetilde{\G}B} \prec  N^{-4/3+\xi}\sqrt{\tr B^2} $.
	\end{enumerate}
\end{lemma}
\begin{proof}
	\begin{enumerate}
		\item By spectral decomposition, we get
		\begin{equation*}
			\sum_{ij} |(\G B)_{ji}|^2=\sum_{i} (B\G^* \G B)_{ii}\prec N\eta_+^{-1}\im \ul{\G}\max_{\alpha}\sum_{ik} (B_{ik}\b u_\alpha(k))^2=N\eta_+^{-1}\im\ul{\G} \max_{\alpha}\langle \b u^{H}_\alpha, B^2\b u^H_{\alpha}\rangle\,.
		\end{equation*}
		Note that
		\begin{equation} \label{A11}
			\begin{aligned}
				\langle \b u^{H}_\alpha, B^2\b u^H_{\alpha}\rangle-\frac{\tr B^2}{N}&=\langle \b u^{H}_\alpha, (B^2-\tr B^2/N\cdot I)\b u^H_{\alpha}\rangle\\
				&\prec \frac{\sqrt{\tr  (B^2-\tr B^2/N\cdot I)^2} }{N} \prec \frac{\tr B^2}{N^{1+\tau/2}}\,,
			\end{aligned}
		\end{equation}
		where in the second step we used Proposition \ref{propA1} (iii), and in the last step we used $\tr B^2\geq N^{\tau}$ and $\|B\|=1$. Combining the above two estimates as well as $ \ul{\widetilde{\G}}\prec N^{-1/3+\xi/2}$ yield the desired result.
		
		\item By Proposition \ref{propA1}\,(i), $(\G B)_{ii}= m_{sc} B_{ii}+O_{\prec}((N\eta_+)^{-1}) \|B_{\cdot i}\|$, which implies
		\[
		\sum_{i} |(\G B)_{ii}-m_{sc}B_{ii}|^2 \prec (N\eta_+)^{-2}\sum_{ij} B^2_{ij}\prec N^{-2/3+\xi}\tr B^2\,.
		\]
	This proves the first estimate. The second estimate follows from $\sum_{i} B_{ii}^2 \leq \tr B^2$ and a triangle inequality.
		\item By spectral decomposition and Proposition \ref{propA1}, we have  
\begin{align*}
	\sum_i |(\widetilde{\G} B)_{ii}|^2 &\prec \sum_i\bigg(\frac{1}{\sqrt{N}} \sum_{\alpha} \frac{\eta_+}{(\lambda_\alpha-E)^2+\eta^2_+} \Big|\sum_{k}\b u^H_\alpha(k)B_{ki}\Big| \bigg)^2\\
	&\prec N \ul{\widetilde{\G}}^2 \max_{\alpha}\langle \b u^{H}_\alpha, B^2\b u^H_{\alpha}\rangle \prec N^{-2/3+\xi}\tr B^2\,.
\end{align*}
Here in the last step we used \eqref{A11}.

\item By spectral decomposition, Proposition \ref{propA1}, \eqref{sum} and $\ul{\widetilde{\G}}\prec N^{-1/3+\xi/2}$, we get
\begin{equation*}
	\begin{aligned}
		(\widetilde{\G}B{\G})_{ij}\prec  \ul{\widetilde{\G}}\cdot  \frac{1}{N}\sum_\alpha \frac{1}{|\lambda_{\alpha}-E'-\ii \eta_+|}\cdot \sqrt{\tr B^2} \prec N^{-1/3+\xi}\sqrt{\tr B^2}
	\end{aligned}
\end{equation*}

\item The proof is very similar to (v), by replacing the estimate $\ul{\widetilde{\G}} \prec N^{-1/3+\xi/2}$ with \eqref{sum}.

\item Once again, by spectral decomposition, Proposition \ref{propA1} and $\im \ul{\G}\prec N^{-1/3+\xi/2}$, we have
\begin{equation*}
	\begin{aligned}
		\sum_{i} |(\G B\widetilde{\G})_{ij}|^2= (\widetilde{\G}B\G^*\G B \widetilde{\G})_{jj}\prec \eta^{-1}_+ \ul{\widetilde{\G}B\widetilde{\G} B \widetilde{\G}} \prec \eta_+^{-2} \ul{\widetilde{\G}B\widetilde{\G} B }\prec N^{-1/3}\tr B^2\,.
	\end{aligned}
\end{equation*}
This proves the first relation. The second estimate follows in a similar fashion by using \eqref{sum}.

\item By \eqref{diff} and Proposition \ref{propA1}, it is not hard to see that
\begin{equation*}
	\partial_{ij}^{r}q(\cal X_E) \prec |\cal J_E| \cdot \max_{i_1i_2}|\im (\G^2)_{i_1i_2}|\prec |\cal J_E| \cdot \eta_-^{-1} |\im \ul{\G}|\prec N^{-1/3+10\xi}\,.
\end{equation*}

\item By spectral decomposition and Proposition \ref{propA1}, we see that
\[
\ul{\widetilde{\G}B} \prec \ul{\widetilde{\G}}\cdot \frac{\sqrt{\tr B^2}}{N}\prec N^{-4/3+\xi}\sqrt{\tr B^2}.
\]
\end{enumerate}
This finishes the proof.
\end{proof}

As a consequence of Lemma \ref{lemmaA2}, we have the following bounds. By Proposition \ref{propA1} and Lemma \ref{lemmaA2} (iv), (vii), (viii), it is not hard to see that
\begin{equation}  \label{7.10}
	\begin{aligned}
	\partial_{ij} \cal Y&\prec  \frac{N}{\sqrt{\tr B^2}}\int_{\cal I}|  \im ({\G}B{\G})_{ij}+\im ({\G}B{\G})_{ji}|q(\cal X_E)+ |\ul{\widetilde{\G}B}\cdot \partial_{ij}q(\cal X_E)|\dd E\prec  N^{2\xi} 
	\end{aligned}
\end{equation}
Similarly, for $r\geq 2$, we can use Proposition \ref{propA1} and Lemma \ref{lemmaA2} (iv), (v), (vii), (viii) to show that
\begin{equation} \label{7.11}
	\partial^{r}\cal Y \prec N^{2\xi}+\frac{N}{\sqrt{\tr B^2}}\int_{\cal I}\max_{i_1,...,i_4}| (\widehat{\G}B\widehat{\G})_{i_1i_2}\widetilde{G}_{i_3i_4}|q(\cal X_E)\dd E \prec N^{2\xi}\,.
\end{equation}
In addition, by Proposition \ref{propA1} (i), (iii) and Lemma \ref{lemmaA2} (vii), (viii), we can also show that
\begin{equation*}
\begin{aligned}
	\partial_{ij} \cal Y&\prec \frac{N}{\sqrt{\tr B^2}}\int_{\cal I}|  \im ({\G}B{\G})_{ij}+\im ({\G}B{\G})_{ji}|q(\cal X_E)\dd E+N^{-1/3+10\xi}\\
&\eqd \cal Y_1^{(ij)}+N^{-1/3+10\xi}.
\end{aligned}
\end{equation*}
As Lemma \ref{lemmaA2} (vi) implies $\sum_{ij}|\cal Y_1^{(ij)}|^2 \prec N^2\cdot |\cal I|^2 \cdot N^{2/3+2\xi}  \prec N^{4/3+3\xi} $, we have
\begin{equation} \label{7.12}
	\sum_{ij}|\partial _{ij}\cal Y|^2 \prec N^{4/3+20\xi} \,.
\end{equation}
Now let 
\[
\cal Y_2^{(ij)}\deq \frac{N}{\sqrt{\tr B^2}} \int_{\cal I} \im ((\cal G B)_{ji}\cdot \ul{\G}) q(\cal X_E)\dd E\,.
\]
By Lemma \ref{lemmaA2} (i), we have $\sum_{ij} |\cal Y_2^{(ij)}|^2\prec N^2\cdot |\cal I|^2\cdot N^{1/3+2\xi} \prec N^{1+3\xi}$. In addition, by Lemma \ref{lemmaA2} (i) -- (iii), together with $ \max_i\widetilde{\G}_{ii}\prec N^{-1/3+\xi}$, we see that 
$
\sum_{ij} |\partial^r_{ij}\cal Y_2^{(ij)}|^2 \prec N^{1+3\xi} 
$
for all fixed $r\geq 1$. Hence
\begin{equation} \label{A14}
	\sum_{ij} |\partial^r_{ij}\cal Y_2^{(ij)}|^2 \prec N^{1+3\xi} 
\end{equation}
for all fixed $r\geq 0$. 

In the sequel, we shall estimate $L_{s}^{(16)}$ for fixed $s \geq 2$; the estimates of $L_{s}^{(15)}$ follow in a similar fashion. 

\subsubsection*{The estimate of $L_{s}^{(16)}, s\geq 2$}Note $L_s^{(16)}$ is bounded by a finite linear combination of the terms in the form 
\begin{equation} \label{A15}
\bb E N^{-(s+1)/2}\sum_{ij} |\partial^{r_0}_{ij}\cal Y_2^{(ij)} \partial^{r_1}_{ij}\cal Y\cdots \partial^{r_m}\cal Y |\,,
\end{equation}
where $r_0,m\geq 0$, $r_1,...,r_m\geq 1$ and $r_0+r_1+...+r_m=s$. 

\textit{Case 1.} Suppose $m=0$, then $r_0=s$, and by \eqref{A14} we get
\begin{equation*}
	\eqref{A15} \prec N^{-(s+1)/2} \cdot N \cdot \bigg(\sum_{ij}\big|\partial^r_{ij}\cal Y_2^{(ij)}\big|^2\bigg)^{1/2} \prec N^{(2-s)/2+3\xi}\prec N^{-\delta}\,.
\end{equation*}

\textit{Case 2.} Suppose $m\geq 1$ and $r_1=1$. Then by \eqref{7.11} -- \eqref{A14} we have
\begin{equation*} 
\eqref{A15} \prec N^{-(s+1)/2}  \bigg(\sum_{ij}\big|\partial^r_{ij}\cal Y_2^{(ij)}\big|^2\bigg)^{1/2}  \Big(\sum_{ij}|\partial _{ij}\cal Y|^2\Big)^{1/2}(N^{2\xi})^{m-1} \prec N^{(4-3s)/6+20s\xi} \prec N^{-\delta}\,.
\end{equation*}

\textit{Case 3.} The remaining case is $m\geq 1$ and $r_1,...,r_m\geq 2$. By \eqref{7.11} and \eqref{A14}, we get
\begin{equation} \label{7.15}
	\begin{aligned}
\eqref{A15} &\prec  N^{-(s+1)/2} \cdot N \cdot  \bigg(\sum_{ij}\big|\partial^r_{ij}\cal Y_2^{(ij)}\big|^2\bigg)^{1/2}  (N^{2\xi})^{m}\prec N^{(6-3s)/6+3s\xi}\,.
	\end{aligned}
\end{equation}
Clearly, $\eqref{7.15}\prec N^{-\delta}$ when $s\geq 3$. However, when $s=2$, \eqref{7.15} is not enough for us.

Summarizing the above four cases, we see that the only term that cannot be directly estimated appears when $s=2$, $m=1$ and $r_1=2$. Together with Lemma \ref{lemmaA2} (vii), we have
\begin{equation*}
	\sum_{s=2}^\ell L_{s}^{(16)}\prec \bigg|\bb E\bigg[ \frac{N^{3/2}}{ \tr B^2}\sum_{ij} c_{ij}\int_{\cal I^2} \im ((\G B)_{ij}\ul{\G})\cdot(\partial^2_{ij} \ul{\widetilde{\G}'B}) )q(\cal X_E)q(\cal X_{E'}) \dd E \dd E'\exp(\ii t\cal Y)\bigg] \bigg|+N^{-\delta}
\end{equation*}
for some uniformly bounded constants $c_{ij}$. 

\subsubsection*{The last error term} By \eqref{diff}, $\ul{\widetilde{\G}} \prec N^{-1/3+\xi}$, and Lemma \ref{lemmaA2} (in particular part (vi)), we see that whenever there is the off-diagonal term $(\G' B\G')_{ij}$ or $(\G' B \G')_{ji}$ coming from $\partial^2_{ij} \ul{\widetilde{\G}'B}$, the corresponding terms can always be bounded by $O_{\prec}(N^{-\delta})$. As a result, we have
	\begin{align}
		\sum_{s=2}^\ell L_{s}^{(16)}&\prec \bigg|\bb E\bigg[ \frac{N^{1/2}}{ \tr B^2}\sum_{ij} c_{ij}\int_{\cal I^2} \im ((\G B)_{ij}\ul{\G})\im ((\G' B \G')_{jj}\G'_{ii} )q(\cal X_E)q(\cal X_{E'}) \dd E \dd E'\exp(\ii t\cal Y)\bigg] \bigg|\nonumber\\
		&\prec \bigg|\bb E\bigg[ \frac{N^{1/2}}{ \tr B^2}\sum_{ij} c_{ij}\int_{\cal I^2} \im ((\G B)_{ij}\ul{\G})(\widehat{\G}' B \widehat{\G}')_{jj}\widetilde{\G}'_{ii} q(\cal X_E)q(\cal X_{E'}) \dd E \dd E'\exp(\ii t\cal Y)\bigg] \bigg|\nonumber+N^{-\delta}\nonumber\\
	&\eqd L_{2,1}^{(16)}+N^{-\delta}\label{7.17}\,.
\end{align}
Here in the second step we used Proposition \ref{propA1} and Lemma \ref{lemmaA2} to bound the other terms by $O_{\prec}(N^{-\delta})$. For instance,we have
\begin{align*}
	&\bb E\bigg[ \frac{N^{1/2}}{\tr B^2}\sum_{ij} c_{ij}\int_{\cal I^2} \im ((\G B)_{ij}\ul{\G})(\widehat{\G}' B \widetilde{\G}')_{jj}\widehat{\G}'_{ii} q(\cal X_E)q(\cal X_{E'}) \dd E \dd E'\exp(\ii t \cal Y)\bigg]\\
	\prec &\,\bigg|\bb E\bigg[ \frac{N^{1/2}}{ \tr B^2}\sum_{ij} c_{ij}\int_{\cal I^2} \im ((\G B)_{ij}\ul{\G})(\widehat{\G}' B \widetilde{\G}')_{jj} q(\cal X_E)q(\cal X_{E'}) \dd E \dd E'\exp(\ii t \cal Y)\bigg]\bigg|\\
	&+\bb E\bigg[ \frac{N^{1/2}}{\tr B^2}\sum_{ij} c_{ij}\int_{\cal I^2} \big|\im ((\G B)_{ij}\ul{\G})(\widehat{\G}' B \widetilde{\G}')_{jj}(\widehat{\G}'_{ii}-\re m_{sc}(z'))q(\cal X_E)q(\cal X_{E'}) \big| \dd E \dd E'\bigg]\\
	\prec &\,\bigg|\bb E\bigg[ \frac{N}{ \tr B^2}\sum_{j} \int_{\cal I^2} \im ((\G B)_{\b c^jj}\ul{\G})(\widehat{\G}' B \widetilde{\G}')_{jj} q(\cal X_E)q(\cal X_{E'}) \dd E \dd E'\exp(\ii t \cal Y)\bigg]\bigg|+N^{-\delta} \prec N^{-\delta}\,,
\end{align*}
where $\b c^j_i\deq N^{-1/2}c_{ij}$ and $\|\b c^j\|=O(1)$. In \eqref{7.17}, the estimates of $L_{2,1}^{(16)}$ and $L_{2,2}^{(16)}$ both rely on the following improvement of Lemma \ref{lemmaA2} (v), which we prove in the next section.

\begin{lemma} \label{lemma 7.3}
	Let $\G\equiv G(E+\ii \eta_+)$ with $E \in \cal I$. We have
	\[
\cal K\deq 	\sum_{i}|(\widehat{\G}B\widehat{\G})_{ii}|^2 \prec N^{2/3+\xi}\tr B^2\eqd \cal E_{11}\,.
	\]
\end{lemma}
Indeed, by Lemma \ref{lemmaA2} (i), Lemma \ref{lemma 7.3} and $\widetilde{\G}'_{ii}\prec N^{-1/3+\xi}$, we have
$
L_{2,1}^{(16)} +L_{2,2}^{(16)} \prec N^{-\delta}.
$
Together with \eqref{7.17}, we get $\sum_{s=2}^\ell L_{s}^{(16)}\prec N^{-\delta}$. Estimating $L_{s}^{(15)}$ in a similar manner, we arrive at
\begin{equation} \label{7error}
	\sum_{s=2}^\ell L_{s}^{(16)}+\sum_{s=2}^\ell L_{s}^{(16)} \prec N^{-\delta}\,.
\end{equation}

\subsubsection{Proof of Lemma \ref{lemma 7.3}}
Note that $(\widehat{\G}B\widehat{\G})_{ii}$ is real and $(\widehat{\G}B\widehat{\G})_{ii}^2\geq 0$. Fix $n\in \bb N_+$. We aim to show that
\begin{equation}
	\cal B^{n} \deq \bb E \Big[\Big( \sum_{i}(\widehat{\G}B\widehat{\G})_{ii}^2\Big)^{n}\Big]\prec N^{-c}\cal B^n+ \cal E_{11}^{1/2}\cal B^{n-1/2}+\cal E_{11}^{n}\eqd \cal E_{12}
\end{equation}
for some fixed $c>0$, which trivially implies the desired result. By resolvent identity and Lemma \ref{lem:cumulant_expansion}, we have
\begin{equation} \label{7.20}
	\begin{aligned}
	z\cal B^{n}&= \bb E \sum_{i} (H\widehat{\G}B\widehat{\G})_{ii} (\widehat{\G}B\widehat{\G})_{ii}\cal K^{n-1}-\bb E \sum_{i} (B\widehat{\G})_{ii} (\widehat{\G}B\widehat{\G})_{ii}\cal K^{n-1}\\
	&=\sum_{s=1}^{\ell}\sum_{ij} \cal C_{s+1}(H_{ij})\bb E \partial^s_{ij}((\widehat{\G}B\widehat{\G})_{ji} (\widehat{\G}B\widehat{\G})_{ii}\cal K^{n-1})+O_{\prec}(\cal E_{12})\eqd\sum_{s=2}^{\ell}L_{s}^{(17)}+O_{\prec}(\cal E_{12})\,,
	\end{aligned}
\end{equation}
where in the second step we used Lemma \ref{lemmaA2} (ii). Note that by Lemma \ref{lemmaA2} (vi) we always have
\begin{equation} \label{7.21}
	\partial_{ij}^r \cal K \prec \cal E_{11}^{1/2}\cal K^{1/2}+\cal E_{11}
\end{equation}
for all fixed $r\geq 1$. In addition, Lemma \ref{lemmaA2} (v), (vi) imply
\begin{equation*}
	\sum_{ij}\Big|\partial^r_{ij}((\widehat{\G}B\widehat{\G})_{ji} (\widehat{\G}B\widehat{\G})_{ii})\Big| \prec N (\cal K+\cal E_{11})
\end{equation*}
for all $r\geq 0$. Thus when $s\geq 2$, we have
\begin{equation}
	L_{s}^{(17)}\prec N^{-(s+1)/2}\cdot N \bb E \bigg[(\cal K+\cal E_{11})\cdot  \sum_{a=0}^{n-1}(\cal E_{11}^{1/2}\cal K^{1/2}+\cal E_{11})^a \cal K^{n-1-a}\bigg] \prec \cal E_{12}\,.
\end{equation}
When $s=1$, by Lemma \ref{lemmaA2} (iv), (vi) and \eqref{7.21}, it is not hard to see that
\begin{equation} \label{7.22}
	L_{1}^{(17)}= -\bb E \ul{\widehat{G}} \cal K^{n}-\bb E \bigg[\sum_{i} \widehat{\G}_{ii}(\widehat{\G}B\widehat{\G})_{ii} \ul{\widehat{G}B\widehat{G}}\cal K^{n-1}\bigg]+O_{\prec}(\cal E_{12})=-m_{sc}\cal B^{n}+O_\prec(\cal E_{12})\,,
\end{equation}
Here in the second step we used Proposition \ref{propA1} to show that $\ul{\widehat{\G}B\widehat{\G}}\prec N^{-2/3+\xi}\cdot \sqrt{\tr B^2}$. Inserting \eqref{7.21} and \eqref{7.22} into \eqref{7.20}, we get
\[
(z+m_{sc})\cal B^n\prec \cal E_{12}\,.
\]
As $(z+m_{sc})^{-1}=-m_{sc}$ is bounded by 1, we obtain the desired result.

\subsection{The leading term, part II} Combining \eqref{A3}, \eqref{7.8} and \eqref{7error}, we obtain
\begin{equation} \label{7.26}
	{\frak g}'(t)=L_{1,5,1}^{(16)}+O_{\prec}(N^{-\delta})\,,
\end{equation}
where the RHS is defines in \eqref{7.8}. Unlike the isotropic case, here we need to expand again. Fortunately, it is very close to what we have seen in \eqref{A3}. By resolvent identity, we see that
\begin{equation*}
\begin{aligned}
L_{1,5,1}^{(16)}&=		\bb E \bigg[\frac{tN^2}{\pi^2 \tr B^2}\int_{\cal I^2}\Big(\ul{\widehat{\cal G'}}\cdot \ul{B\widetilde{\cal G'}\widetilde{\cal G}B} +\ul{H\widehat{\G'}}\cdot \ul{\widehat{\G'}B\widetilde{\cal G'}\widetilde{\cal G}B}-\ul{H\widetilde{\G'}}\cdot \ul{\widetilde{\G'}B\widetilde{\cal G'}\widetilde{\cal G}B}-\ul{\widehat{\G}}'\cdot \ul{H\widehat{\G}'B\widetilde{\cal G'}\widetilde{\cal G}B}\\
&\quad \quad +\ul{\widetilde{\G}}'\cdot \ul{H\widetilde{\G}'B\widetilde{\cal G'}\widetilde{\cal G}B}\Big) \,q(\cal X_E) q(\cal X_{E'})\dd E \dd E'\exp(\ii t\cal Y)\bigg]\eqd L^{(16)}_{1,5,1,1}+\cdots+L^{(16)}_{1,5,1,5}\,.
	\end{aligned}
\end{equation*}
Then we expand $L^{(16)}_{1,5,1,2},...,L^{(16)}_{1,5,1,5}$ via Lemma \ref{lem:cumulant_expansion}. By applying \eqref{diff}, as in \eqref{whatever} and \eqref{whatever2}, the leading terms of $L^{(16)}_{1,5,1,2}$ and $L^{(16)}_{1,5,1,4}$ will cancel each other, and so will the leading terms of $L^{(16)}_{1,5,1,3}$ and $L^{(16)}_{1,5,1,5}$. Moreover, similar to Sections \ref{secA1} and \ref{sec7.2}, we can use Proposition \ref{propA1} and Lemma \ref{lemmaA2} and show that the rest terms in $L^{(16)}_{1,5,1,2},...,L^{(16)}_{1,5,1,5}$ are bounded by $O_{\prec}(N^{-\delta})$. As a result, we have
\begin{equation*}
	L_{1,5,1}^{(16)}=L_{1,5,1,1}^{(16)}+O_\prec(N^{-\delta})=-\bb E \bigg[\frac{tN^2}{\pi^2 \tr B^2}\int_{\cal I^2} \ul{B\widetilde{\cal G'}\widetilde{\cal G}B}q(\cal X_E) q(\cal X_{E'})\dd E \dd E'\exp(\ii t\cal Y)\bigg]+O_\prec(N^{-\delta})\,,
\end{equation*}
where in the last step we used $\ul{\widetilde{\G}}'=-1+O_{\prec}(N^{-1/3+\xi})$ and Proposition \ref{propA1}. Now we proceed as in \eqref{5.19}. By \eqref{rigidity}, \eqref{A2} and \eqref{A11}, we can get
	\begin{align*}
	\hspace{-0.3cm}	L_{1,5,1}^{(16)}&=-\bb E \bigg[\frac{tN}{\pi^2 \tr B^2}\int_{\cal I^2}\sum_{\alpha}\frac{\langle \b u^{H}_\alpha, B^2\b u^H_{\alpha}\rangle \eta_+^2}{(\lambda_{a}-E)^2+\eta_+^2)((\lambda_{a}-E')^2+\eta_+^2)}q(\cal X_E) q(\cal X_{E'})\dd E \dd E'\exp(\ii t\cal Y)\bigg]+O_\prec(N^{-\delta})\nonumber\\
		&=-\bb E \bigg[\frac{tN}{\pi^2 \tr B^2}\int_{\cal I^2}\frac{\langle \b u^{H}_1, B^2\b u^H_{1}\rangle \eta_+^2}{(\lambda_{1}-E)^2+\eta_+^2)((\lambda_{1}-E')^2+\eta_+^2)}q(\cal X_E) q(\cal X_{E'})\dd E \dd E'\exp(\ii t\cal Y)\bigg]+O_\prec(N^{-\delta})\nonumber\\
		&=-\bb E \bigg[\frac{tN}{\tr B^2}\langle \b u^{H}_1, B^2\b u^H_{1}\rangle \exp(\ii t \cal Y)\bigg]+O_\prec(N^{-\delta})=-\bb E \big[t \exp(\ii t \cal Y)\big]+O_\prec(N^{-\delta})\,.
	\end{align*}
Inserting the above relation into \eqref{7.26} yields
\[
\frak g'(t)=-\bb E \big[t \exp(\ii t \cal Y)\big]+O_\prec(N^{-\delta})=-t\frak g(t)+O_\prec(N^{-\delta})
\]
uniformly for all $t \in [-T,T]$. As $\frak g(0)=1$, we obtain \eqref{A1} as desired. This finishes the proof. 

	{\small
	
	\bibliography{bibliography} 
	
	\bibliographystyle{amsplain}
}

\end{document}